\documentclass{article}
\input{macros-arXiv.sty}

\usepackage{url}            % simple URL typesetting
\usepackage{booktabs}       % professional-quality tables
\usepackage{amsfonts, bm, amssymb}       % blackboard math symbols
\usepackage{nicefrac}       % compact symbols for 1/2, etc.
\usepackage{microtype}      % microtypography
\usepackage{tikz} 
\usepackage{float}

\usepackage{array}
\usepackage{diagbox}
\usepackage{mathtools}
\usepackage{multirow, makecell}
\usepackage{rotating}
\usepackage{graphicx}
\usepackage{subfigure}
\usepackage{mathdots}

\usepackage{float}
\usepackage{caption}
\usepackage{setspace}
\usepackage[letterpaper, left=0.9in, right=0.9in, top=0.8in, bottom=1in]{geometry}

\usepackage[bottom]{footmisc}
\usepackage[colorlinks,linkcolor={blue!75!black},urlcolor=red,citecolor=ForestGreen]{hyperref}
\usepackage{footnotebackref}

\usepackage{threeparttable}

\usepackage{arydshln}

\usepackage{enumerate}
\usepackage{enumitem} 

\usepackage{natbib}

\title{\bf The Complexity of Nonconvex-Strongly-Concave\\
Minimax Optimization}

\author{
Siqi Zhang \footnote{Equal contribution.} \thanks{Coordinated Science Laboratory (CSL), Department of Industrial and Enterprise Systems Engineering (ISE), University of Illinois at Urbana-Champaign (UIUC), Urbana, IL, USA. \texttt{siqiz4@illinois.edu}, \texttt{junchiy2@illinois.edu}}
\and
Junchi Yang \footnotemark[1]\ \footnotemark[2]
\and
Crist\'{o}bal Guzm\'{a}n \thanks{Department of Applied Mathematics, University of Twente, The Netherlands. And 
Institute for Mathematical and Computational Engineering, Pontificia Universidad Cat\'olica de Chile. \texttt{c.a.guzmanparedes@utwente.nl}}
\and
Negar Kiyavash
\thanks{College of Management of Technology, EPFL, Switzerland. \texttt{negar.kiyavash@epfl.ch}}
\and
Niao He \thanks{Department of Computer Science, ETH Z\"urich, Switzerland. \texttt{niao.he@inf.ethz.ch}}
}

\begin{document}
\maketitle

\begin{abstract}
    This paper studies the complexity for finding approximate stationary points of \textit{nonconvex-strongly-concave (NC-SC)} smooth minimax problems, in both \emph{general} and  \emph{averaged smooth finite-sum} settings. We establish nontrivial lower complexity bounds of $\Omega(\sqrt{\kappa}\Delta L\epsilon^{-2})$ and $\Omega(n+\sqrt{n\kappa}\Delta L\epsilon^{-2})$ for the two settings, respectively, where $\kappa$ is the condition number, $L$ is the smoothness constant, and $\Delta$ is the initial gap. Our result reveals substantial gaps between these limits and best-known upper bounds in the literature. To  close these gaps, we introduce a \textit{generic acceleration scheme} that deploys existing gradient-based methods to solve a sequence of crafted strongly-convex-strongly-concave subproblems. In the general setting, the complexity of our proposed algorithm nearly matches the lower bound; in particular, it removes an additional poly-logarithmic dependence on accuracy present in previous works. In the averaged smooth finite-sum setting, our proposed algorithm improves over previous algorithms by providing a nearly-tight dependence on the condition number.
\end{abstract}

\section{Introduction}

In this paper, we consider general minimax problems of the form ($n, d_1, d_2\in\mathbb{N}^+$):
\begin{equation}\label{eq:main_problems}
    \min_{x\in \mathbb{R}^{d_1}}\max_{y\in \mathbb{R}^{d_2}}\ f(x,y),
\end{equation}
as well as their finite-sum counterpart:%the finite-sum structure:
\begin{equation}\label{eq:main_problems_2}
    \min_{x\in \mathbb{R}^{d_1}}\max_{y\in \mathbb{R}^{d_2}}\ f(x,y)\triangleq\frac{1}{n}\sum_{i=1}^{n}f_i(x,y),
\end{equation}
where $f,f_i$ are continuously differentiable and $f$ is $L$-Lipschitz smooth jointly in $x$ and $y$. We focus on the setting when the function $f$ is $\mu$-strongly concave in $y$ and perhaps nonconvex in $x$, i.e., $f$ is \textit{nonconvex-strongly-concave (NC-SC)}. Such problems arise ubiquitously in machine learning, e.g., GANs with regularization \citep{sanjabi2018convergence,lei2020sgd}, Wasserstein robust models \citep{sinha2018certifying}, robust learning over multiple domains \citep{qian2019robust}, and off-policy reinforcement learning \citep{dai2017learning,dai2018sbeed,huang2020convergence}. Since the problem is nonconvex in general, a natural goal is to find an approximate stationary point $\Bar{x}$, such that $\autonorm{\nabla\Phi(\Bar{x})}\leq\epsilon$, for a given accuracy $\epsilon$, where $\Phi(x)\triangleq\max_y f(x,y)$ is the primal function. This goal is meaningful for the aforementioned applications, e.g.,~in adversarial models the primal function quantifies the worst-case loss for the learner, with respect to adversary's actions.

There exists a number of algorithms for solving NC-SC problems in the general setting, including GDmax \citep{nouiehed2019solving}, GDA~\citep{lin2020gradient}, alternating GDA~\citep{yang2020global,boct2020alternating,xu2020unified}, Minimax-PPA~\citep{lin2020near}. 
Specifically, GDA and its alternating variant both achieve the complexity of $ O(\kappa^2 \Delta L\epsilon^{-2}) $  \citep{lin2020gradient,yang2020global}, where $\kappa\triangleq \frac{L}{\mu}$ is the condition number and $\Delta\triangleq\Phi(x_0)-\inf_{x}\Phi(x)$ is the initial function gap.  Recently, \citep{lin2020near} provided the best-known complexity of $ O\autopar{\sqrt{\kappa} \Delta L \epsilon^{-2}\cdot \log^2(\frac{\kappa L}{\epsilon})}$ achieved by Minimax-PPA, which improves the dependence on the condition number but suffers from an extra poly-logarithmic factor in $\frac{1}{\epsilon}$.

In the finite-sum setting, several algorithms have been proposed recently, e.g., SGDmax~\citep{jin2020local}, PGSMD~\citep{rafique2018non}, Stochastic GDA~\citep{lin2020gradient}, SREDA and its variants~\citep{luo2020stochastic,xu2020enhanced}. In particular, 
\citep{lin2020gradient} proved that Stochastic GDA attains the complexity of $O(\kappa^3\epsilon^{-4})$. \citep{luo2020stochastic} recently showed the state-of-the-art result achieved by SREDA: when $n \geq \kappa^2$, the complexity is $\tilde{O}\autopar{n\log\frac{\kappa}{\epsilon}+\sqrt{n}\kappa^2 \Delta L\epsilon^{-2}}$, which is sharper than the batch Minimax-PPA algorithm;  when $n \leq \kappa^2$, the complexity is $O\autopar{\autopar{n\kappa+\kappa^2} \Delta L\epsilon^{-2}}$, which is sharper than Stochastic GDA. 

\begin{table*}
	\centering
	\small
    % \footnotesize
	\renewcommand{\arraystretch}{1.5}
	\begin{threeparttable}[b]
		\begin{tabular}{l | c | c | c}
			\hline
			\hline
			\textbf{Setting} & \textbf{Our Lower Bound} & \textbf{Our Upper Bound} & \textbf{Previous Upper Bound}
            \\
			\hline 
			\hline
			\multirow{2}{*}{NC-SC, general}
			& \multirow{2}{*}{
    			\makecell[c]{
    			$ \Omega\big(\sqrt{\kappa} \Delta L\epsilon^{-2}\big) $\vspace{0.1em}\\
    			Theorem \ref{THM:LB_NCSC_DETER}
    			}
			}
			& \multirow{2}{*}{
    			\makecell[c]{$ \Tilde{O}(\sqrt{\kappa} \Delta L\epsilon^{-2}) $ \vspace{0.1em}\\ Section \ref{subsec algorithms}
    			}
			}
			& \multirow{2}{*}{
			    \makecell[c]{
			    $ O(\kappa^2 \Delta L\epsilon^{-2}) $ \citep{lin2020gradient}\vspace{0.25em}\\
			    $ \tilde{O}\autopar{\sqrt{\kappa} \Delta L \epsilon^{-2}\log^2\frac{1}{\epsilon}} $ \citep{lin2020near}
			    }
			}
			\\
			& & &
			\\
			\hline
			\multirow{3}{*}{NC-SC, FS, AS\tnote{1}\ \ }
			& \multirow{3}{*}{
			    \makecell[c]{
			    $ \Omega\autopar{n+\sqrt{n\kappa} \Delta L \epsilon^{-2}} $\vspace{0.1em}\\
			    Theorem \ref{THM:LB_NCSC_FS_AS}
			    }
			}
			& \multirow{3}{*}{\makecell[c]{$ \tilde{O}\autopar{\autopar{n+n^{\frac{3}{4}}\sqrt{\kappa}} \Delta L \epsilon^{-2}} $\vspace{0.1em} \\ Section \ref{subsec algorithms}}}
			& \multirow{3}{*}{
			    \makecell[c]{
			    $
    			\begin{cases}
    			    \tilde{O}(n+\sqrt{n}\kappa^2 \Delta L\epsilon^{-2}) & n \geq \kappa^2\\
    			    O\autopar{\autopar{n\kappa+\kappa^2} \Delta L\epsilon^{-2}} & n \leq \kappa^2
    			\end{cases}
    			$\vspace{0.1em}\\ 
    			\citep{luo2020stochastic,xu2020enhanced}
			    }
			}
			\\
			& & &
			\\
			& & &
			\\
			\hline
			\hline
		\end{tabular}
		\begin{tablenotes}
 			\item[1] FS: finite-sum, AS: averaged smooth; see Section \ref{SEC:MAIN_PRELIM} for definitions.
 		\end{tablenotes}
	\end{threeparttable}
	\caption{Upper and lower complexity bounds for finding an approximate stationary point. Here $\tilde{O}(\cdot)$ hides poly-logarithmic factor in $L,\mu$ and $\kappa$. $L$: Lipschitz smoothness parameter, $\mu$: strong concavity parameter, $\kappa$: condition number $\frac{L}{\mu}$; $\Delta$: initial gap of the primal function.}
    \label{table:summary_results}
\end{table*}

Despite this active line of research,  whether these state-of-the-art complexity bounds can be further improved remains elusive. 
As a special case by restricting the domain of $y$ to a singleton, lower bounds for nonconvex smooth minimization, e.g., \citep{carmon2019lower,carmon2019lowerII,fang2018spider,zhou2019lower,arjevani2019lower}, hardly capture the dependence on the condition number $\kappa$, which plays a crucial role in the complexity for general NC-SC smooth minimax problems. 
In many of the aforementioned machine learning applications, the condition number is often proportional to the inverse of the regularization parameter, and could be quite large in practice. For example, in statistical learning, where $n$ represents the sample size, the optimal regularization parameter (i.e.~with optimal empirical/generalization trade-off) leads to $\kappa=\Omega(\sqrt{n})$ \citep{shalev2014understanding}.

This motivates the following fundamental questions: \textit{What is the complexity limit for NC-SC problems in the general and finite-sum settings? Can we design new algorithms to meet the performance limits and attain optimal dependence on the condition number?}

\subsection{Contributions} 
\label{sec:contribution}
Our contributions, summarized in Table \ref{table:summary_results}, are  as follows:
\begin{itemize}
    \item We establish nontrivial lower complexity bounds for finding an approximate stationary point of nonconvex-strongly-concave (NC-SC) minimax problems. In the general setting, we prove an $ \Omega\big(\sqrt{\kappa} \Delta L\epsilon^{-2}\big) $ lower complexity bound which applies to arbitrary deterministic linear-span algorithms interacting with the classical first-order oracle. In the finite-sum setting, we prove an $ \Omega\autopar{n+\sqrt{n\kappa} \Delta L \epsilon^{-2}} $ lower complexity bound (when $\kappa=\Omega(n)$)\footnote{A concurrent work by \cite{han2021lower} appeared on arXiv two weeks ago, and provided a similar lower bound result for finite-sum NC-SC problems under probabilistic arguments based on geometric random variables.} for the class of {\em averaged smooth functions} and arbitrary linear-span algorithms interacting with a (randomized) incremental first-order oracle (precise definitions in Sections \ref{SEC:MAIN_PRELIM} and \ref{sec:main_LB_NCSC}). 
    
    Our lower bounds build upon two main ideas: first, we start from an NC-SC function whose primal function mimics the lower bound construction in smooth nonconvex minimization \citep{carmon2019lower}. Crucially, the smoothness parameter of this primal function is boosted by an $\Omega(\kappa)$ factor, which strengthens the lower bound. 
    Second, the NC-SC function has an alternating zero-chain structure, as utilized in lower bounds for convex-concave settings \citep{ouyang2019lower}. The combination of these features leads to a hard instance for our problem.

    \item To bridge the gap between the lower bounds and existing upper bounds in both settings, we introduce a generic Catalyst acceleration framework for NC-SC minimax problems, inspired by~\citep{lin2018catalyst,yang2020catalyst}, which applies existing gradient-based methods to solving a sequence of crafted strongly-convex-strongly-concave (SC-SC) minimax subproblems. When combined with the extragradient method, the resulting algorithm achieves an $ \tilde{O}(\sqrt{\kappa} \Delta L\epsilon^{-2}) $ complexity in terms of gradient evaluations, which tightly matches the lower bound in the general setting (up to logarithmic terms in constants) and shaves off the extra poly-logarithmic term in $\frac{1}{\epsilon}$ required by the state-of-the-art \citep{lin2020near}. When combined with stochastic variance-reduced method, the resulting algorithm achieves an overall $ \tilde{O}\autopar{(n+n^{3/4}\sqrt{\kappa}) \Delta L \epsilon^{-2}} $ complexity for averaged smooth finite-sum problems, which has nearly-tight dependence on the condition number and improves on the best-known upper bound when $n\leq \kappa^4$. 
\end{itemize}

\subsection{Related Work}

\paragraph{Lower bounds for minimax problems.} 
Information-based complexity (IBC) theory \citep{traub1988information}, which derives the minimal number of oracle calls to attain an approximate solution with a desired accuracy, is often used in lower bound analysis of optimization algorithms. Unlike the case of minimization \citep{blair1985problem,nesterov2018lectures,agarwal2009information,woodworth2016tight,foster2019complexity,carmon2019lower,carmon2019lowerII,fang2018spider,zhou2019lower,arjevani2019lower}, lower bounds for minimax optimization are far less understood; only a few recent works provided lower bounds for finding an approximate saddle point of (strongly)-convex-(strongly)-concave minimax problems \citep{ouyang2019lower,zhang2019lower, ibrahim2020linear,xie2020lower,yoon2021accelerated}. Instead, this paper considers lower bounds for NC-SC minimax problems of finding an approximate stationary point, which requires different techniques for constructing zero-chain properties. Note that there exists another line of research on the purely stochastic setting, e.g., \citep{rafique2018non,luo2020stochastic,xu2020enhanced}; constructing lower bounds in that setting is out of the scope of this paper.

\paragraph{Complexity of making gradient small.} 
In nonconvex optimization, most lower and upper complexity bound results are presented in terms of the gradient norm (see a recent survey \citep{danilova2020recent} and references therein for more details). For convex optimization, the optimality gap based on the objective value is commonly used as the convergence criterion. The convergence in terms of gradient norm, albeit easier to check, are far less studied in the literature until recently; see e.g., \citep{nesterov2012make,allen2018make,foster2019complexity,carmon2019lowerII, diakonikolasguzman2021} for convex minimization and \citep{diakonikolas2020halpern,diakonikolas2021potential,yoon2021accelerated} for convex-concave smooth minimax problems. 

\paragraph{Nonconvex minimax optimization.} 
In NC-SC setting, as we mentioned, there has been several substantial works. Among them, \citep{lin2020near} achieved the best dependency on condition number by combining proximal point algorithm with accelerated gradient descent. \citep{luo2020stochastic} introduced a variance reduction algorithm, SREDA, and \citep{xu2020enhanced} enhanced the analysis to allow bigger stepsize. \citep{yuan2021federated,guo2020communication} provided algorithms for NC-SC minimax formulation of AUC maximization problems with an additional assumption that the primal function satisfies Polyak-{\L}ojasiewicz condition. 
In addition, nonconvex-concave minimax optimization, i.e., the function $f$ is only concave in $y$, is extensively explored by \citep{zhang2020single, ostrovskii2020efficient, thekumparampil2019efficient, zhao2020primal,nouiehed2019solving,yang2020catalyst}.  Recently, \citep{daskalakis2020complexity} showed that for general smooth nonconvex-nonconcave objectives the computation of approximate first-order locally optimal solutions is intractable. Therefore, another line of research is devoted to searching for solutions under additional structural properties \citep{yang2020devil,zhou2017stochastic, yang2020global, song2020optimistic,mertikopoulos2019optimistic,malitsky2019golden,diakonikolas2020efficient, lin2018solving}.

\paragraph{Catalyst acceleration.} 
The catalyst framework was initially studied in \citep{lin2015universal} for convex minimization and extended to nonconvex minimization in \citep{paquette2018catalyst} to obtain accelerated algorithms. A similar idea to accelerate SVRG appeared in \citep{frostig2015regularizing}. These work are rooted on the proximal point algorithm (PPA) \citep{rockafellar1976monotone,guler1991convergence} and inexact accelerated PPA \citep{guler1992new}. Recently, \citep{yang2020catalyst} generalized the idea  and obtained state-of-the-art results for solving strongly-convex-concave and nonconvex-concave minimax problems. In contrast, this paper introduces a new catalyst acceleration scheme in the nonconvex-strongly-concave setting, which relies on completely different parameter settings and stopping criterion.

\section{Preliminaries} \label{SEC:MAIN_PRELIM}

\paragraph{Notations} 
Throughout the paper, we use $\dom F$ as the domain of a function $F$,   $\nabla F=\autopar{\nabla_x F, \nabla_y F}$ as the full gradient, $\autonorm{\cdot}$ as the $\ell_2$-norm. We use $0$ to represent zero vectors or scalars, $e_i$ to represent unit vector with the $i$-th element being $1$. For nonnegative functions $f(x)$ and $g(x)$, we say $f=O\autopar{g}$ if $f(x)\leq cg(x)$ for some $c>0$, and further write $f=\tilde{O}\autopar{g}$ to omit poly-logarithmic terms on constants $L,\mu$ and $\kappa$, while $f=\Omega\autopar{g}$ if $f(x)\geq cg(x)$ (see more in Appendix \ref{sec:Apdx_notations}).

We introduce definitions and assumptions used throughout.

\begin{definition}[Primal and Dual Functions]
    For a function $f(x,y)$, we define $\Phi(x)\triangleq\max_y f(x,y)$ as the primal function, and $\Psi(y)\triangleq\min_x f(x,y)$ as the dual function. We also define the primal-dual gap at a point $(\Bar{x},\Bar{y})$ as $
    \gap_f(\Bar{x}, \Bar{y}) \triangleq \max_{y \in\mathbb{R}^{d_2}} f(\Bar{x},y) - \min_{x \in \mathbb{R}^{d_1}} f(x,\Bar{y}).
$
\end{definition}

\begin{definition}[Lipschitz Smoothness]
    We say a function $f(x,y)$ is $L$-Lipschitz smooth ($L$-S) jointly in $x$ and $y$ if it is differentiable and for any $(x_1,y_1), (x_2,y_2)\in\mathbb{R}^{d_1}\times\mathbb{R}^{d_2}$, $\autonorm{\nabla_x f(x_1,y_1)-\nabla_x f(x_2,y_2)}\leq L(\|x_1-x_2\|+\|y_1-y_2\|)$ and $\autonorm{\nabla_y f(x_1,y_1)-\nabla_y f(x_2,y_2)}\leq L(\|x_1-x_2\|+\|y_1-y_2\|)$, for some $L>0$.
\end{definition}

\begin{definition}[Average / Individual Smoothness]
    \label{defn:AS_IS}
    We say $f(x,y)=\frac{1}{n}\sum_{i=1}^{n}f_i(x,y)$ or $\{f_i\}_{i=1}^n$ is $L$-averaged smooth ($L$-AS) if each $f_i$ is differentiable, and for any $(x_1,y_1), (x_2,y_2)\in\mathbb{R}^{d_1}\times\mathbb{R}^{d_2}$, we have
    \begin{equation}
        \begin{split}
            \frac{1}{n}\sum_{i=1}^n\autonorm{\nabla f_i(x_1,y_1)-\nabla f_i(x_2,y_2)}^2\leq
        	L^2\autopar{\autonorm{x_1-x_2}^2+\autonorm{y_1-y_2}^2}.
        \end{split}
    \end{equation}
    We say $f$ or $\autobigpar{f_i}_{i=1}^n$ is $L$-individually smooth ($L$-IS) if each $f_i$ is L-Lipschitz smooth.
\end{definition}

Average smoothness is a weaker condition than the common Lipschitz smoothness assumption of each component in finite-sum~/~stochastic minimization \citep{fang2018spider,zhou2019lower}. Similarly in minimax problems, the following proposition summarizes the relationship among these different notions of smoothness. 

\begin{prop} \label{prop regularized smooth} Let $f(x,y)=\frac1n\sum_{i=1}^n f_i(x,y)$. Then we have: (a) If the function $f$ is $L$-IS or $L$-AS, then it is $L$-S. (b) If $f$ is $L$-IS, then it is  $(2L)$-AS. (c) If $f$ %(x,y) = \frac{1}{n}\sum_{i=1}^{n}f_i(x,y)$ 
is L-AS, then $f(x,y) + \frac{\tau_x}{2}\|x-\Tilde{x}\|^2 - \frac{\tau_y}{2}\|y-\Tilde{y}\|^2$ is $\sqrt{2}(L+\max\{\tau_x, \tau_y\})$-AS for any $\Tilde{x}$ and $\Tilde{y}$.
\end{prop}

\begin{definition}[Strong Convexity]
    A differentiable function $g:\mathbb{R}^{d_1}\rightarrow\mathbb{R}$ is convex if $g(x_2)\geq g(x_1)+\autoprod{\nabla g(x_1), x_2-x_1}$ for any $x_1, x_2\in\mathbb{R}^{d_1}$. Given $\mu\geq 0$, 
    we say $f$ is 
    %, it is said to be 
    $\mu$-strongly convex if $ g(x)-\frac{\mu}{2}\autonorm{x}^2 $ is convex, and it is $\mu$-strongly concave if $-g$ is $\mu$-strongly convex.
\end{definition}

Next we introduce main assumptions throughout this paper.

% \paragraph{Main Assumptions}
\begin{assume}[Main Settings] \label{main assumption}
% \begin{assume} \label{main assumption}
    We assume that $f(x,y)$ in \eqref{eq:main_problems} is a \textit{nonconvex-strongly-concave (NC-SC)} function such that $f$ is $L$-S, and $f(x,\cdot)$ is $\mu$-strongly concave for any fixed $x\in\mathbb{R}^{d_1}$; for the finite-sum case, we further assume that $\autobigpar{f_i}_{i=1}^n$ is $L$-AS. We assume that the initial primal suboptimality is bounded: $\Phi(x_0)-\inf_{x}\Phi(x)\leq \Delta$.
\end{assume}

Under Assumption \ref{main assumption}, the primal function $\Phi(\cdot)$ is differentiable and $2\kappa L$-Lipschitz smooth \citep[Lemma 23]{lin2020near} where $\kappa\triangleq\frac{L}{\mu}$. Throughout this paper, we use the stationarity of the primal function as the convergence criterion. 

\begin{definition}[Convergence Criterion]
    For a differentiable function $\Phi$, a point $\Bar{x}\in\dom\Phi$ is called an $\epsilon$-stationary point of $\Phi$ if $\autonorm{\nabla\Phi(\Bar{x})}\leq\epsilon$.
\end{definition}

Another commonly used criterion is the stationarity of $f$, i.e., $\autonorm{\nabla_x f(\Bar{x},\Bar{y})}\leq \epsilon, \autonorm{\nabla_y f(\Bar{x},\Bar{y})}\leq \epsilon$. This is a weaker convergence criterion. We refer readers to \citep[Section 4.3]{lin2020gradient} for the comparison of these two criteria.

\section{Lower Bounds for NC-SC Minimax Problems}
\label{sec:main_LB_NCSC}
In this section, we establish lower complexity bounds (LB) for finding approximate stationary points of NC-SC minimax problems, in both general and finite-sum settings. We first present the basic components of the oracle complexity framework \citep{blair1985problem} and then proceed to the details for each case. For simplicity, in this section only, we denote $x_d$ as the $d$-th coordinate of $x$ and $x^t$ as the variable $x$ in the $t$-th iteration.

\subsection{Framework and Setup}
We study the lower bound of finding primal stationary point under the well-known oracle complexity framework \citep{blair1985problem}, here we first present the basics of the framework.

\paragraph{Function class}
We consider the \textit{nonconvex-strongly-concave (NC-SC)} function class, as defined in Assumption \ref{main assumption}, with parameters $L,\mu,\Delta>0$, denoted by $ \mathcal{F}_{\mathrm{NCSC}}^{L,\mu,\Delta} $.

\paragraph{Oracle class}
We consider different oracles for the general and finite-sum settings. Define $z\triangleq(x,y)$. 
\begin{itemize}
    \item For the general setting, we consider the \textit{first-order oracle (FO)}, denoted as $ \mathbb{O}_{\mathrm{FO}}(f,\cdot) $, that for each query on point $ z $, it returns the gradient 
    $ \mathbb{O}_{\mathrm{FO}}(f,z)\triangleq\autopar{\nabla_x f(x,y), \nabla_y f(x,y)}. $
    \item For the finite-sum setting, \textit{incremental first-order oracle (IFO)} is often used in lower bound analysis \citep{agarwal2015lower}. This oracle for a function $f(x,y)=\frac{1}{n}\sum_{i=1}^nf_i(x,y)$, is such that for each query on point $ z $ and given $i\in[n]$, it %, as a stochastic mapping, will 
    returns the gradient of the $i$-th component, i.e.,
    $ \mathbb{O}_{\mathrm{IFO}}(f,z,i)\triangleq\autopar{\nabla_x f_i(x,y), \nabla_y f_i(x,y)}, $. Here, we consider \textit{averaged smooth IFO} and \textit{individually smooth IFO}, denoted as $ \mathbb{O}_{\mathrm{IFO}}^{L,\mathrm{AS}}(f) $ and $ \mathbb{O}_{\mathrm{IFO}}^{L,\mathrm{IS}}(f) $, where $\autobigpar{f_i}_{i=1}^n$ is $L$-AS or $L$-IS, respectively.
\end{itemize}

\paragraph{Algorithm class}
In this work, we consider \colorword{the class of {\em linear-span algorithms}}{black} interacting with oracle $\mathbb{O}$, denoted as ${\cal A}(\mathbb{O})$. 
%Here we consider linear-span algorithm class $ \mathcal{A} $. 
These algorithms satisfy the following property: if we let $(z^t)_t$ be the sequence of queries by the algorithm, where $z^t=(x^t,y^t)$; %and $\mathbb{O}$ is the used oracle, 
then for all $t$, we have 
\begin{equation}
    \label{eq:alg_protocol_deter}
    z^{t+1}\in\mathrm{Span}\autobigpar{z^0,\cdots,z^t;\mathbb{O}\autopar{f,z^0},\cdots,\mathbb{O}\autopar{f,z^t}}.
\end{equation}
For the finite-sum case, the above protocol fits with many existing deterministic and randomized linear-span algorithms. We distinguish the general and finite-sum setting by specifying the used oracle, which is $\mathbb{O}_{\mathrm{FO}}$ or $\mathbb{O}_{\mathrm{IFO}}$, respectively. Most existing first-order algorithms, including simultaneous and alternating update algorithms, can be formulated as linear-span algorithms. It is worth pointing out that typically the linear span assumption 
%in our setting, the linear-span assumption
is used without loss of generality, 
since there is a standard reduction from deterministic linear-span algorithms to arbitrary oracle based deterministic algorithms \citep{Nemirovski:1991, Nemirovski:1992, ouyang2019lower}. We defer this extension for future work.

\paragraph{Complexity measures}
The efficiency of algorithms is quantified by the \textit{oracle complexity} \citep{blair1985problem} of finding an $ \epsilon $-stationary point of the primal function: for an algorithm $\mathtt{A}\in\mathcal{A}(\mathbb{O})$ interacting with a FO oracle $\mathbb{O}$, an instance $f\in\mathcal{F}$, 
we define 
\begin{equation}
    T_{\epsilon}(f,\mathtt{A})\triangleq\inf
	\autobigpar{T\in\mathbb{N}|\|\nabla \Phi\autopar{x^T}\|\leq\epsilon}
\end{equation}
as the minimum number of oracle calls $\mathtt{A}$ makes to reach stationarity convergence. For the general case, we define the {\em worst-case complexity}
%following the protocol \eqref{eq:alg_protocol_deter}, 
\begin{equation}
    \label{eq:defn_LB_deter}
    % \small
    \mathrm{Compl}_\epsilon\autopar{\mathcal{F},\mathcal{A},\mathbb{O}}
    \triangleq
	\underset{f\in\mathcal{F}}{\sup}\ 
	\underset{\mathtt{A}\in{\mathcal{A}(\mathbb{O})}}{\inf}\ 
	T_{\epsilon}(f,\mathtt{A}).
\end{equation}
For the finite-sum case, we consider the {\em randomized complexity} \citep{braun2017lower}:
\begin{equation}
    \label{eq:defn_LB_FS}
    % \small
	\mathrm{Compl}_{\epsilon}\autopar{\mathcal{F},\mathcal{A},\mathbb{O}}
    \triangleq
	\underset{f\in\mathcal{F}}{\sup}\ 
	\underset{\mathtt{A}\in{\mathcal{A}(\mathbb{O})}}{\inf}\ 
	\mathbb{E}\ T_{\epsilon}(f,\mathtt{A}).
\end{equation}

Following the motivation of analysis discussed in Section \ref{sec:contribution}, we will use the zero-chain argument for the analysis. First we define the notion of (first-order) zero-chain \citep{carmon2019lowerII} and activation as follows.

\begin{definition}[Zero Chain, Activation] A function $ f:\mathbb{R}^d\rightarrow\mathbb{R} $ is a first-order zero-chain if for any $ x\in\mathbb{R}^d $, 
    \begin{equation}
        \mathrm{supp}\{x\}\subseteq\{1,\cdots,i-1\}\ \Rightarrow\ \mathrm{supp}\{\nabla f(x)\}\subseteq\{1,\cdots,i\},
    \end{equation}
    where $ \mathrm{supp}\{x\}\triangleq\{i\in[d]\ | \ x_i\neq 0\} $ and $ [d]=\{1,\cdots,d\} $. 
    \textcolor{black}{For an algorithm initialized at $0\in\mathbb{R}^d$, with iterates $\{x^t\}_t$, 
    we say coordinate $i$ is activated at $x^t$, if $x_i^t\neq 0$ and $x_i^s= 0$, for any $s<t$.}
\end{definition}

\subsection{General NC-SC Problems}
First we consider the \textit{general NC-SC (Gen-NC-SC)} minimax optimization problems. Following the above framework, we choose function class $\mathcal{F}_{\mathrm{NCSC}}^{L,\mu,\Delta}$, oracle $\mathbb{O}_{\mathrm{FO}}$, linear-span algorithms ${\cal A}$, and we analyze the complexity defined in \eqref{eq:defn_LB_deter}. 

\paragraph{Hard instance construction} 
Inspired by the hard instances constructed in~\citep{ouyang2019lower,carmon2019lowerII}, we introduce the following function $ F_d:\mathbb{R}^{d+1}\times\mathbb{R}^{d+2}\rightarrow\mathbb{R} $ ($d\in\mathbb{N}^+$) and 
\begin{equation}
    \label{eq:LB_hard_instance_deter}
	\begin{split}
	    F_d(x,y;\lambda,\alpha)
    	\triangleq
    	\lambda_1\autoprod{B_dx,y}-
    	\lambda_2\|y\|^2-\frac{\lambda_1^2\sqrt{\alpha}}{2\lambda_2}\autoprod{e_1,x}+
    	\frac{\lambda_1^2\alpha}{2\lambda_2}\sum_{i=1}^{d}\Gamma(x_i)-
    	\frac{\lambda_1^2\alpha}{4\lambda_2}x_{d+1}^2+
    	\frac{\lambda_1^2\sqrt{\alpha}}{4\lambda_2},
	\end{split}
\end{equation}
where $ \lambda=(\lambda_1,\lambda_2)\in\mathbb{R}^2 $ is the parameter vector, $ e_1\in\mathbb{R}^{d+1} $ is the unit vector with the only non-zero element in the first dimension, $ \Gamma:\mathbb{R}\rightarrow\mathbb{R} $ and $B_d\in\mathbb{R}^{(d+2)\times(d+1)}$ are
\begin{equation}
    B_d=
	\begin{bmatrix}
		& & & & 1\\
		& & & 1 & -1\\
		& & \iddots & \iddots &\\
		& 1 & -1 & &\\
		1 & -1 & & &\\
		\sqrt[4]{\alpha} & & & &
	\end{bmatrix},
	\quad
	\Gamma(x)
	=
	120\int_{1}^{x}\frac{t^2(t-1)}{1+t^2}dt.
\end{equation}
Matrix $ B_d $ essentially triggers the activation of variables at each iteration, and function $ \Gamma $ introduces nonconvexity in $x$ to the  instance. By the first-order optimality condition of $ F_d(x,\cdot;\lambda,\alpha) $, we can %it is easy to 
compute its primal function, $\Phi_d$:
\begin{equation}
    \label{eq:LB_hard_instance_deter_Phi}
	\begin{split}
	    \Phi_d(x;\lambda,\alpha)
	    \triangleq\ &
	    \max_{y\in\mathbb{R}^{d+1}}F_d\left(x,y;\lambda,\alpha\right)\\
	    =\ &
    	\frac{\lambda_1^2}{2\lambda_2}
    	\autopar{
    	\frac{1}{2}x^\top A_dx-
    	\sqrt{\alpha}x_1+
    	\frac{\sqrt{\alpha}}{2}+
    	\alpha\sum_{i=1}^{d}\Gamma(x_i)+
    	\frac{1-\alpha}{2}x_{d+1}^2
    	},
	\end{split}
\end{equation}
where $A_d\in\mathbb{R}^{(d+1)\times(d+1)}$ is
\begin{equation}
    \label{eq:matrix_A_d_definition}
	\begin{split}
	    A_d
	    =
	    \left(B_d^\top B_d-e_{d+1}e_{d+1}^\top\right)
	    =
    	\begin{bmatrix}
    		1+\sqrt{\alpha} & -1 & & & \\
    		-1 & 2 & -1 & & & \\
    		& -1 & 2 & \ddots & & \\
    		& & \ddots & \ddots & -1 & \\
    		& & & \ddots & 2 & -1 \\
    		& & & & -1 & 1
    	\end{bmatrix}.
	\end{split}
\end{equation}
The resulting primal function resembles the worst-case functions used in lower bound analysis of minimization problems \citep{nesterov2018lectures,carmon2019lowerII}.

\paragraph{Zero-Chain Construction}
First we summarize key properties of the instance and its zero-chain mechanism. We further denote $\hat{e}_i\in\mathbb{R}^{d+2}$ as the unit vector for the variable $y$ and define ($k\geq 1$)
\begin{equation}
	\begin{split}
		\mathcal{X}_k&\triangleq\mathrm{Span}\{e_1,e_2,\cdots,e_k\},\ \ \mathcal{X}_0\triangleq\{0\},
		\\
		\mathcal{Y}_k&\triangleq\mathrm{Span}\{\hat{e}_{d+2},\hat{e}_{d+1},\cdots,\hat{e}_{d-k+2}\},\ \ \mathcal{Y}_0\triangleq\{0\},
	\end{split}
\end{equation}
then we have the following properties for $F_d$.

\begin{lemma}[Properties of $ F_d $]
	\label{LM:NCSC_LB_F_D}
	For any $ d\in\mathbb{N}^+ $ and $\alpha\in\automedpar{0,1}$, $ F_d(x,y;\lambda,\alpha) $ in \eqref{eq:LB_hard_instance_deter} satisfies:
	\begin{itemize}
		\item[(i)] The function $ F_d(x,\cdot;\lambda,\alpha) $ is $ L_F $-Lipschitz smooth where $L_F=\max\autobigpar{\frac{200\lambda_1^2\alpha}{\lambda_2},2\lambda_1,2\lambda_2}$.
		\item[(ii)] For each fixed $ x\in\mathbb{R}^{d+1} $, $ F_d(x,\cdot;\lambda,\alpha) $ is $ \mu_F $-strongly concave where $ \mu_F=2\lambda_2 $.
		\item[(iii)] The following properties hold:
		\begin{enumerate}
			\item[a)] $ x=y=0\quad \Longrightarrow\quad \nabla_x F_d\in\mathcal{X}_1,\ \nabla_y F_d=0 $.
			\item[b)] $ x\in\mathcal{X}_k,\ y\in\mathcal{Y}_k\quad \Longrightarrow\quad \nabla_x F_d\in\mathcal{X}_{k+1},\ \nabla_y F_d\in\mathcal{Y}_k $.
			\item[c)] $ x\in\mathcal{X}_{k+1},\ y\in\mathcal{Y}_k\quad \Longrightarrow\quad \nabla_x F_d\in\mathcal{X}_{k+1},\ \nabla_y F_d\in\mathcal{Y}_{k+1} $.
		\end{enumerate}
		\item[(iv)] For $ L\geq\mu>0 $, if $ \lambda=\lambda^*=(\lambda_1^*,\lambda_2^*)=(\frac{L}{2},\frac{\mu}{2}) $ and $ \alpha\leq\frac{\mu}{100L} $, then $ F_d $ is $ L $-Lipschitz smooth. Moreover for any fixed $ x\in\mathbb{R}^{d+1} $, $ F_d(x,\cdot;\lambda,\alpha) $ is $ \mu $-strongly concave.
	\end{itemize}
\end{lemma}

The proof of Lemma \ref{LM:NCSC_LB_F_D} is deferred to Appendix \ref{sec:Apdx_LM_NCSC_LB_F_D}. The first two properties show that  function $F_d$ is Lipschitz smooth and NC-SC; the third property above suggests that, starting from $ (x,y)=(0,0) $, the activation process follows an "alternating zero-chain" form \citep{ouyang2019lower}. That is, for a linear-span algorithm, when $ x\in\mathcal{X}_k,\ y\in\mathcal{Y}_k$, the next iterate will at most activate the $(k+1)$-th coordinate of $x$ while keeping $y$ fixed; similarly when $x\in\mathcal{X}_{k+1},\ y\in\mathcal{Y}_k$, the next iterate will at most activate the $(d-k+1)$-th element of $y$. We need the following properties of $ \Phi_d $ for the lower bound argument.

\begin{lemma}[Properties of $ \Phi_d $]
	\label{LM:NCSC_LB_PHI}
	For any $ \alpha\in\automedpar{0,1} $ and $ x\in\mathbb{R}^{d+1} $, if $ x_d=x_{d+1}=0 $, we have: 
	\begin{itemize}
	    \item[(i)] $ \autonorm{\nabla\Phi_d(x;\lambda,\alpha)}\geq\frac{\lambda_1^2}{8\lambda_2}\alpha^{3/4} $.
	    \item[(ii)] $\Phi_d\autopar{0;\lambda,\alpha}
        	-
        	\inf_{x\in\mathbb{R}^{d+1}}\Phi_d(x;\lambda,\alpha)
        	\leq
        	\frac{\lambda_1^2}{2\lambda_2}\left(\frac{\sqrt{\alpha}}{2}+10\alpha d\right)$.
	\end{itemize}
\end{lemma}
We defer the proof of Lemma \ref{LM:NCSC_LB_PHI} to Appendix \ref{sec:Apdx_LM_NCSC_LB_PHI}. This lemma indicates that, starting from $ (x,y)=(0,0) $ with appropriate parameter settings, the primal function $ \Phi_d $ will not approximate stationarity until the last two coordinates are activated. Now we are ready to present our final lower bound result for the general NC-SC case.

\begin{theorem}[LB for Gen-NC-SC] 
    \label{THM:LB_NCSC_DETER}
    For the linear-span first-order algorithm class $ \mathcal{A} $, parameters $ L,\mu,\Delta>0 $, 
    % then for small enough accuracy $ \epsilon>0 $,
    and accuracy $\epsilon$ satisfying $\epsilon^2\leq\min\left(\frac{\Delta L}{6400}, \frac{\Delta L\sqrt{\kappa}}{38400}\right)$,
    we have
    \begin{equation}
    % \small
        \mathrm{Compl}_\epsilon
        \autopar{
        \mathcal{F}_{\mathrm{NCSC}}^{L,\mu,\Delta},\mathcal{A},\mathbb{O}_{\mathrm{FO}}
        }
        =
        \Omega\autopar{\sqrt{\kappa} \Delta L \epsilon^{-2}}.
    \end{equation}
\end{theorem}

The hard instance in the proof is established based on $F_d$ in \eqref{eq:LB_hard_instance_deter}. We choose the scaled function $f(x,y)=\eta^2F_d(\frac{x}{\eta},\frac{y}{\eta};\lambda^*,\alpha)$ as the final hard instance, which  preserves the smoothness and strong convexity (by Lemma \ref{lm:scaling}), while appropriate setting of $\eta$ will help to fulfill the requirements on the initial gap and large gradient norm (before thorough activation) of the primal function. The detailed statement and proof of Theorem \ref{THM:LB_NCSC_DETER} are presented in Appendix \ref{sec:Apdx_THM_LB_NCSC_DETER}.

\begin{remark}[Tightness of Theorem \ref{THM:LB_NCSC_DETER}]
	The best-known upper bounds for general NC-SC problems are $ O(\Delta L\kappa^2\epsilon^{-2}) $ \citep{lin2020gradient, boct2020alternating} and $ \Tilde{O}\autopar{\Delta \sqrt{\kappa}L\epsilon^{-2}\log^2\frac{1}{\epsilon}} $ \citep{lin2020near}. Therefore, our result exhibits significant %so there are 
	gaps in terms of the dependence on $ \epsilon $ and $ \kappa $. In order to mitigate these gaps, we propose faster algorithms in Section \ref{sec:main_Catalyst}. On the other hand, compared to the $\Omega(\Delta L\epsilon^{-2})$ lower bound for nonconvex smooth minimization \citep{carmon2019lower}, our result reveals an explicit dependence on $\kappa$.
\end{remark}

\subsection{Finite-Sum NC-SC Problems}
The second case we consider is \textit{finite-sum NC-SC (FS-NC-SC)} minimax problems, for the function class $\mathcal{F}_{\mathrm{NCSC}}^{L,\mu,\Delta}$, the linear-span algorithm class $\mathcal{A}$ and the averaged smooth  IFO class $\mathbb{O}_{\mathrm{IFO}}^{L,\mathrm{AS}}$. The complexity is defined in \eqref{eq:defn_LB_FS}.

\paragraph{Hard instance construction} To derive the finite-sum hard instance, we modify $F_d$ in \eqref{eq:LB_hard_instance_deter} with orthogonal matrices defined as follows.
\begin{definition}[Orthogonal Matrices]
    For positive integers $a,b,n\in\mathbb{N}^+$, we define a matrix sequence $ \{\mathbf{U}^{(i)}\}_{i=1}^n \in \mathrm{\mathbf{Orth}}(a,b,n) $ if for each $ i, j\in\autobigpar{1,\cdots, n} $ and $ i\neq j $, $ \mathbf{U}^{(i)},\mathbf{U}^{(j)}\in\mathbb{R}^{a\times b} $ and  $ \mathbf{U}^{(i)}(\mathbf{U}^{(i)})^\top=\mathbf{I}\in\mathbb{R}^{a\times a} $ and $\mathbf{U}^{(i)}(\mathbf{U}^{(j)})^\top=\mathbf{0}\in\mathbb{R}^{a\times a}$.
\end{definition}

Here the intuition for the finite-sum hard instance is combining $n$ independent copies of the hard instance in the general case \eqref{eq:LB_hard_instance_deter}, then appropriate orthogonal matrices will convert the $n$ independent variables with dimension $d$ into one variable with dimension $n\times d$, which results in the desired hard instance. To preserve the zero chain property, for $\{\mathbf{U}^{(i)}\}_{i=1}^n \in \mathrm{\mathbf{Orth}}(d+1,n(d+1),n)$, $\{\mathbf{V}^{(i)}\}_{i=1}^n \in \mathrm{\mathbf{Orth}}(d+2,n(d+2),n)$, $\forall n,d\in\mathbb{N}^+$ and $x\in\mathbb{R}^{n(d+1)}$, $y\in\mathbb{R}^{n(d+2)}$, we set $\mathbf{U}^{(i)}$ and $\mathbf{V}^{(i)}$ by concatenating $n$ matrices:
\begin{equation}
    \label{eq:LB_FS_Orthogonal_Matrix}
    \begin{split}
        \mathbf{U}^{(i)}
        =\ &
        \begin{bmatrix}
            \mathbf{0}_{d+1} & \cdots & \mathbf{0}_{d+1} & \mathbf{I}_{d+1} & \mathbf{0}_{d+1} &  \cdots & \mathbf{0}_{d+1}
        \end{bmatrix},\\
        \mathbf{V}^{(i)}
        =\ &
        \begin{bmatrix}
            \mathbf{0}_{d+2} & \cdots & \mathbf{0}_{d+2} & \mathbf{I}_{d+2} & \mathbf{0}_{d+2} &  \cdots & \mathbf{0}_{d+2}
        \end{bmatrix},
    \end{split}
\end{equation}
where $\mathbf{0}_d, \mathbf{I}_d\in\mathbb{R}^{d\times d}$ are the zero and identity matrices respectively, while the $i$-th matrix above is the identity matrix. Hence, $\mathbf{U}^{(i)}x$ will be the $(id-d+1)$-th to the $(id)$-th elements of $x$, similar property also holds for $\mathbf{V}^{(i)}y$. 

The hard instance construction here follows the idea of that in the deterministic hard instance \eqref{eq:LB_hard_instance_deter}, the basic motivation is that its primal function will be a finite-sum form of the primal function $\Phi_d$ defined in the deterministic case \eqref{eq:LB_hard_instance_deter_Phi}. We choose the following functions $H_d:\mathbb{R}^{d+1}\times\mathbb{R}^{d+2}\rightarrow\mathbb{R}$, $\Gamma_d^n:\mathbb{R}^{n(d+1)}\rightarrow\mathbb{R}$ and
\begin{equation}
    \begin{split}
        H_d(x,y;\lambda,\alpha)
    	\triangleq\ &
    	\lambda_1\autoprod{B_dx,y}-
    	\lambda_2\|y\|^2-\frac{\lambda_1^2\sqrt{\alpha}}{2\lambda_2}\autoprod{e_1,x}-
    	\frac{\lambda_1^2\alpha}{4\lambda_2}x_{d+1}^2+
    	\frac{\lambda_1^2\sqrt{\alpha}}{4\lambda_2},\\
    	\Gamma_d^n(x)
    	\triangleq\ &
    	\sum_{i=1}^n\sum_{j=i(d+1)-d}^{i(d+1)-1}\Gamma(x_j),
    \end{split}
\end{equation}
then $\bar{f}_i, \bar{f}: \mathbb{R}^{n(d+1)}\times\mathbb{R}^{n(d+2)}\rightarrow\mathbb{R}$, $\{\mathbf{U}^{(i)}\}_{i=1}^n \in \mathrm{\mathbf{Orth}}(d+1,n(d+1),n)$, $\{\mathbf{V}^{(i)}\}_{i=1}^n \in \mathrm{\mathbf{Orth}}(d+2,n(d+2),n)$ and
\begin{equation}
    \label{eq:LB_hard_instance_FS}
	\begin{split}
	    \bar{f}_i(x,y)
    	\triangleq\ &
    	H_d\autopar{\mathbf{U}^{(i)}x,\mathbf{V}^{(i)}y;\lambda,\alpha}+\frac{\lambda_1^2\alpha}{2n\lambda_2}\Gamma_d^n(x),\\
    	\bar{f}(x,y)
    	\triangleq\ &
    	\frac{1}{n}\sum_{i=1}^{n}\bar{f}_i(x,y)
    	=
    	\frac{1}{n}\sum_{i=1}^{n}\automedpar{H_d\autopar{\mathbf{U}^{(i)}x,\mathbf{V}^{(i)}y;\lambda,\alpha}+\frac{\lambda_1^2\alpha}{2n\lambda_2}\Gamma_d^n(x)},
	\end{split}
\end{equation}
note that by denoting $\Gamma_d(x)\triangleq\sum_{i=1}^{d}\Gamma(x_i)$, it is easy to find that
\begin{equation}
    \label{eq:Gamma_nd_equivalent}
    \Gamma_d^n(x)=\sum_{i=1}^n\sum_{j=i(d+1)-d}^{i(d+1)-1}\Gamma(x_j)=\sum_{i=1}^n\Gamma_d\autopar{\mathbf{U}^{(i)}x}=\sum_{i=1}^n\sum_{j=1}^{d}\Gamma\autopar{\autopar{\mathbf{U}^{(i)}x}_j}.
\end{equation}
Define $ u^{(i)}\triangleq\mathbf{U}^{(i)}x $, we summarize the properties of the above functions in the following lemma.

\begin{lemma}[Properties of $\bar{f}$]
    \label{lm:properties_hard_instance_FS_AS}
    For the above functions $\{\bar{f}_i\}_i$ and $\bar{f}$ in \eqref{eq:LB_hard_instance_FS}, they satisfy that:
    \begin{itemize}
        \item[(i)] $\{\bar{f}_i\}_i$ is $L_F$-AS where $L_F=\sqrt{\frac{1}{n}\max\autobigpar{16\lambda_1^2+8\lambda_2^2, \frac{C_\gamma^2\lambda_1^4\alpha^2}{n\lambda_2^2}+\frac{\lambda_1^4\alpha^2}{\lambda_2^2}+8\lambda_1^2}}$.
        \item[(ii)] $\bar{f}$ is $\mu_F$-strongly concave on $y$ where $\mu_F=\frac{2\lambda_2}{n}$.
        \item[(iii)] For $n\in\mathbb{N}^+$, $ L\geq 2n\mu>0 $, if we set $ \lambda=\lambda^*=(\lambda_1^*,\lambda_2^*)=\autopar{\sqrt{\frac{n}{40}}L,\frac{n\mu}{2}} $,
        $\alpha=\frac{n\mu}{50L}\in\automedpar{0,1}$, then $\{\bar{f}_i\}_i$ is $L$-AS and $\bar{f}$ is $\mu$-strongly concave on $y$.
        \item[(iv)] Define $ \bar{\Phi}(x)\triangleq \max_y \bar{f}(x,y)$, then we have
        \begin{equation}
            \bar{\Phi}(x)=\frac{1}{n}\sum_{i=1}^{n}\bar{\Phi}_i(x),
            \quad \text{where}\quad
            \bar{\Phi}_i(x)\triangleq\Phi_d(\mathbf{U}^{(i)}x),
        \end{equation}
    while $\Phi_d$ is defined in \eqref{eq:LB_hard_instance_deter_Phi}.
    \end{itemize}
\end{lemma}
We defer the proof of Lemma \ref{lm:properties_hard_instance_FS_AS} to Appendix \ref{sec:Apdx_hard_instance_FS_AS_properties}. From Lemma \ref{LM:NCSC_LB_PHI}, we have
\begin{equation}
    \label{eq:hard_instance_FS_initial_gap}
	\begin{split}
		\bar{\Phi}(0)-\inf_{x\in\mathbb{R}^{n(d+1)}}\bar{\Phi}(x)
		=\ &
		\sup_{x\in\mathbb{R}^{n(d+1)}}\frac{1}{n}\sum_{i=1}^{n}\autopar{\bar{\Phi}(0)-\bar{\Phi}_i(x)}
		\leq
		\frac{1}{n}\sum_{i=1}^{n}\sup_{x\in\mathbb{R}^{d+1}}\autopar{\bar{\Phi}(0)-\bar{\Phi}_i(x)}\\
		=\ &
		\frac{1}{n}\sum_{i=1}^{n}\autopar{\sup_{x\in\mathbb{R}^{d+1}}\autopar{\Phi_d(0)-\Phi_d(\mathbf{U}^{(i)}x)}}
		\leq
		\frac{\lambda_1^2}{2\lambda_2}\autopar{\frac{\sqrt{\alpha}}{2}+10\alpha d}.
	\end{split}
\end{equation}
Define the index set $ \mathcal{I} $ as all the indices $i\in[n]$ such that $ u^{(i)}_d=u^{(i)}_{d+1}=0,\ \forall i\in\mathcal{I} $. Suppose that $ |\mathcal{I}|>\frac{n}{2} $, by orthogonality and Lemma \ref{LM:NCSC_LB_PHI} we have
\begin{equation}
    \label{eq:LB_FS_nonconvergence}
	\begin{split}
		\autonorm{\nabla \bar{\Phi}(x)}^2
		=\ &
		\autonorm{\frac{1}{n}\sum_{i=1}^n\nabla\bar{\Phi}_i(x)}^2
		=
		\autonorm{\frac{1}{n}\sum_{i=1}^n\nabla\autopar{\Phi_d\autopar{\mathbf{U}^{(i)}x}}}^2
		= 
		\frac{1}{n^2}\autonorm{\sum_{i=1}^n\autopar{\mathbf{U}^{(i)}}^\top\nabla\Phi_d\autopar{\mathbf{U}^{(i)}x}}^2\\
		=\ &
		\frac{1}{n^2}\sum_{i=1}^n\autonorm{\autopar{\mathbf{U}^{(i)}}^\top\nabla\Phi_d\autopar{\mathbf{U}^{(i)}x}}^2
		=
		\frac{1}{n^2}\sum_{i=1}^n\autonorm{\nabla\Phi_d\autopar{u^{(i)}}}^2
		\geq
		\frac{1}{n^2}\sum_{i\in\mathcal{I}}\autonorm{\nabla\Phi_d\autopar{u^{(i)}}}^2\\
		\geq\ &
		\frac{1}{n^2}\frac{n}{2}\autopar{\frac{\lambda_1^2}{8\lambda_2}\alpha^{\frac{3}{4}}}^2
		=
		\frac{\lambda_1^4}{128n\lambda_2^2}\alpha^{\frac{3}{2}}.
	\end{split}
\end{equation}
Now we arrive at our final theorem for the averaged smooth FS-NC-SC case as follows.

\begin{theorem}[LB for AS FS-NC-SC] 
\label{THM:LB_NCSC_FS_AS}
For the linear-span algorithm class $ \mathcal{A} $, parameters $ L,\mu,\Delta>0 $ and component size $n\in\mathbb{N}^+$, if $ L\geq 2n\mu>0 $, the accuracy $\epsilon$ satisfies that 
$
\epsilon^2
\leq
\min\autopar{
\frac{\sqrt{\alpha}L^2\Delta}{76800n\mu}, 
\frac{\alpha L^2\Delta}{1280n\mu}, 
\frac{L^2\Delta}{\mu}
}
$ where $\alpha=\frac{n\mu}{50L}\in\automedpar{0,1}$,
then we have
\begin{equation}
    \mathrm{Compl}_\epsilon\autopar{\mathcal{F}_{\mathrm{NCSC}}^{L,\mu,\Delta},\mathcal{A},\mathbb{O}_{\mathrm{IFO}}^{L,\mathrm{AS}}}
    =
    \Omega\autopar{n+\sqrt{n\kappa} \Delta L \epsilon^{-2}}.
\end{equation}
\end{theorem}

The theorem above indicates that for any $\mathtt{A}\in\mathcal{A}$, we can construct a function $f(x,y)=\frac{1}{n}\sum_{i=1}^{n}f_i(x,y)$, such that $ f\in\mathcal{F}_{\mathrm{NCSC}}^{L,\mu,\Delta} $ and $ \{f_i\}_i $ is $ L $-AS, and $ \mathtt{A} $ requires at least $\Omega\autopar{n+\sqrt{n\kappa} \Delta L \epsilon^{-2}}$ IFO calls to attain an approximate stationary point of its primal function (in terms of expectation). 
The hard instance construction is based on $\bar{f}$ and $\bar{f}_i$ above \eqref{eq:LB_hard_instance_FS}, combined with a scaling trick similar to the one in the general case. Also we remark that lower bound holds for small enough $\epsilon$, while the requirement on $\epsilon$ is comparable to those in existing literature, e.g. \citep{zhou2019lower,han2021lower}. The detailed statement and proof of the theorem are deferred to Appendix \ref{sec:Apdx_THM_LB_NCSC_FS_AS}.

\begin{remark}[Tightness of Theorem \ref{THM:LB_NCSC_FS_AS}]
    The state-of-the-art upper bound for NC-SC finite-sum problems is $\tilde{O}(n+\sqrt{n}\kappa^2 \Delta L\epsilon^{-2})$ when $n \geq \kappa^2$ and $O\autopar{\autopar{n\kappa+\kappa^2} \Delta L\epsilon^{-2}}$ when $n \leq \kappa^2$ \citep{luo2020stochastic,xu2020enhanced}. Note that there is still a large gap between upper and lower bounds on the dependence in terms of $\kappa$ and $n$, which motivates the design of faster algorithms for FS-NC-SC case, we address this in Section \ref{sec:main_Catalyst}. \colorword{Note that a weaker result on the lower bound of nonconvex finite-sum averaged smooth minimization is $\Omega(\sqrt{n}\Delta L\epsilon^{-2})$ \citep{fang2018spider,zhou2019lower,li2020page}; here, our result presents explicitly the dependence on $\kappa$.}{black}
\end{remark}

\section{Faster Algorithms for NC-SC Minimax Problems}
\label{sec:main_Catalyst}

In this section, we introduce a generic Catalyst acceleration scheme that turns existing optimizers for (finite-sum) SC-SC minimax problems into efficient, near-optimal algorithms for (finite-sum) NC-SC minimax optimization.  Rooted in the inexact accelerated proximal point algorithm, the idea of Catalyst acceleration was introduced in \cite{lin2015universal} for convex minimization and later extended to nonconvex minimization in \cite{paquette2018catalyst} and nonconvex-concave minimax optimization in~\cite{yang2020catalyst}.
In stark contrast, we focus on NC-SC minimax problems. 

The backbone of our Catalyst framework is to repeatedly solve regularized subproblems of the form:
$$\min_{x\in \mathbb{R}^{d_1}}\max_{y\in \mathbb{R}^{d_2}} f(x, y) + L\|x-\Tilde{x}_t\|^2 - \frac{\tau}{2}\|y-\Tilde{y}_t\|^2,$$
where $\Tilde{x}_t$ and $\Tilde{y}_t$ are carefully chosen prox-centers, and the parameter $\tau\geq 0$ is selected such that the condition numbers for $x$-component and $y$-component of these subproblems are well-balanced.  Since $f$ is $L$-Lipschitz smooth and $\mu$-strongly concave in $y$,  the above auxiliary problem is $L$-strongly convex in $x$ and $(\mu+\tau)$-strongly concave in $y$. 
Therefore, it can be easily solved by a wide family of off-the-shelf first-order algorithms with linear convergence rate. 

Our Catalyst framework, presented in Algorithm \ref{catalyst ncc 1}, consists of three crucial components: an inexact proximal point step for primal update, an inexact accelerated proximal point step for dual update, and a linear-convergent algorithm for solving the subproblems.

\paragraph{Inexact proximal point step in the primal.} The $x$-update in the outer loop, $\{x_0^t\}_{t=1}^T$, can be viewed as applying an inexact proximal point method to the primal function $\Phi(x)$, requiring to solve the following sequence of auxiliary problems:
%a series of auxiliary problems as follows:
\begin{equation*}  \label{auxiliary prob ncc}
    \min_{x\in \mathbb{R}^{d_1}}\max_{y\in \mathbb{R}^{d_2}} \left[\hat{f}_{t}(x,y)\triangleq  f(x,y) + L\Vert x - x^t_{0}\Vert^2 \right].   \tag{$\star$} 
\end{equation*}
Inexact proximal point methods have been explored in minimax optimization in several work, e.g. \citep{lin2020near, rafique2018non}.  Our scheme is distinct from these work in two aspects: (i) we introduce a new subroutine to approximately solve the auxiliary problems (\ref{auxiliary prob ncc}) with near-optimal complexity, and (ii) the inexactness is measured by an adaptive stopping criterion using the gradient norms:
\begin{equation} \label{ncc criteion}
  \|\nabla \hat{f}_t(x^{t+1}_{0}, y^{t+1}_{0})\|^2 \leq \alpha_t\|\nabla \hat{f}_t(x^t_{0}, y^t_{0})\|^2, 
\end{equation}
where $\{\alpha_t\}_t$ is carefully chosen. Using the adaptive stopping criterion significantly reduces the complexity of solving the auxiliary problems. %As we will show in the next section, 
We will show that 
the number of steps required is only logarithmic in $L, \mu$ without any dependence on target accuracy $\epsilon$.   
 % (b) getting rid of requirement on boundedness of the dual domain, unlike \citep{lin2020near}.
Although the auxiliary problem is $(L, \mu)$-SC-SC and can be solved with linear convergence by algorithms such as extragradient, OGDA, etc., these algorithms are not optimal in terms of the dependency on the condition number when $L>\mu$ \citep{zhang2019lower}. 

\paragraph{Inexact accelerated proximal point step in the dual.}  To solve the auxiliary problem with optimal complexity,  we introduce an inexact accelerated proximal point scheme. The key idea is to add an extra regularization in $y$ to the objective such that the strong-convexity and strong-concavity are well-balanced. Therefore, we propose to iteratively solve the subproblems:
\begin{equation*}  \label{subprob}
    \min_{x\in \mathbb{R}^{d_1}}\max_{y\in \mathbb{R}^{d_2}} \left[\Tilde{f}_{t,k}(x,y)\triangleq  \hat{f}_t(x,y) - \frac{\tau}{2}\Vert y - z_k\Vert^2 \right],  \tag{$\star\star$} 
\end{equation*}
where $\{z_k\}_k$ is updated analogously to Nesterov's accelerated method 
%comes from Nesterov's acceleration 
\citep{nesterov2005smooth} and $\tau\geq 0$ is the regularization parameter. For example, by setting $\tau=L-\mu$, the subproblems become $(L,L)$-SC-SC and can be approximately solved by extragradient method with optimal complexity, to be discussed in more details in next section. Finally, when solving these subproblems, we use the following stopping criterion 
$\|\nabla \Tilde{f}_{t,k}(x, y)\|^2 \leq \epsilon^t_k$
 with time-varying accuracy $\epsilon^t_k$ that decays exponentially with $k$.  
 % The choice of $\tau$, warm-start and stopping criterion of the subproblem separate us from MINIMAX-APPA algorithm in \citep{lin2020near}, which require prefixing the target accuracy in solving the subproblems. We will see such difference enables us to solve (\ref{auxiliary prob ncc}) with the complexity of  $\log\frac{1}{\alpha_t}$ w.r.t. $\alpha_t$,
% %, which is not the case in \citep{lin2020near}, 
% and leveraging existing algorithms in both general and finite-sum problems. 

\begin{algorithm}[t] 
    \caption{Catalyst for NC-SC Minimax Problems}
    \setstretch{1.25}
    \begin{algorithmic}[1] \label{catalyst ncc 1}
        \REQUIRE objective $f$, initial point $(x_0, y_0)$, smoothness constant $L$, strong-concavity const.~$\mu$, and param.~$\tau>0$.
        \STATE Let $(x_{0}^0, y_{0}^0) = (x_0, y_0)$ and $q = \frac{\mu}{\mu+\tau}$.
        \FORALL{$t = 0,1,..., T$}
            \STATE Let $z_1 = y^t_{0}$ and $k=1$.
            \STATE Let $\hat{f}_{t}(x,y)\triangleq  f(x,y) + L\Vert x - x^t_{0}\Vert^2$.
            \REPEAT
               \STATE Find inexact solution $(x^t_{k}, y^t_{k})$ to the problem below by algorithm $\mathcal{M}$ with initial point $(x^t_{k-1},y^t_{k-1})$:
            %   \vspace{-3mm}
             \begin{equation*} 
             \begin{split}
                \min_{x\in \mathbb{R}^{d_1}}\max_{y\in \mathbb{R}^{d_2}}
                \automedpar{\Tilde{f}_{t,k}(x,y)\triangleq  f(x,y) +L\|x-x^t_{0}\|^2- \frac{\tau}{2}\Vert y - z_k\Vert^2}
            \end{split}\tag{$\star\star$} 
            \end{equation*}
            such that $\|\nabla \Tilde{f}_{t,k}(x^t_{k}, y^t_{k})\|^2 \leq \epsilon^t_k$.
               \STATE Let $z_{k+1} = y^t_{k} + \frac{\sqrt{q}-q}{\sqrt{q}+q}(y^t_{k}-y^t_{k-1}), k=k+1$.
            \UNTIL{$\|\nabla \hat{f}_t(x^t_{k}, y^t_{k})\|^2 \leq \alpha_t\|\nabla \hat{f}_t(x^t_{0}, y^t_{0})\|^2$} 
            \STATE Set $(x^{t+1}_{0}, y^{t+1}_{0}) = (x^t_{k}, y^t_{k})$.
        \ENDFOR
        \ENSURE $\hat{x}_{T}$, which is uniformly sampled  from $x^1_{0},...,x^T_{0}$.
    \end{algorithmic}
\end{algorithm}

\paragraph{Linearly-convergent algorithms for SC-SC subproblems.} 
Let $\mathcal{M}$ be any algorithm that solves the subproblem (\ref{subprob}) (denoting $(x^*, y^*)$ as the optimal solution) at a linear convergence rate such that after $N$ iterations: 
\begin{align}
    \Vert x_N - x^*\Vert^2+\Vert y_N-y^*\Vert^2
    \leq \left(1-\frac{1}{\Lambda^{\mathcal{M}}_{ \mu, L}(\tau)}\right)^N[\Vert x_0-x^*\Vert^2 + \Vert y_0-y^*\Vert^2],
\end{align}
if $\mathcal{M}$ is a deterministic algorithm; or 
taking expectation to the left-hand side above
%\begin{align*}
%    &\mathbb{E}[\Vert x_N - x^*\Vert^2+\Vert y_N-y^*\Vert^2] \\
%    &\leq \left(1-\frac{1}{\Lambda^{\mathcal{M}}_{ \mu, L}(\tau)}\right)^N[\Vert x_0-x^*\Vert^2 + \Vert y_0-y^*\Vert^2],
%\end{align*}
if $\mathcal{M}$ is randomized.  The choices for $\mathcal{M}$ include, but are not limited to, extragradient (EG) \citep{tseng1995linear}, optimistic gradient descent ascent (OGDA) \citep{gidel2018variational}, SVRG \citep{balamurugan2016stochastic}, SPD1-VR \citep{tan2018stochastic}, SVRE \citep{chavdarova2019reducing}, Point-SAGA \citep{luo2019stochastic}, and variance reduced prox-method \citep{carmon2019variance}.   For example, in the case of EG, $\Lambda^{\mathcal{M}}_{\mu, L}(\tau) = \frac{L+\max\{2L,\tau\}}{4\min\{L, \mu+\tau\}}$~\citep{tseng1995linear}.

\subsection{Convergence Analysis}
%We first present the outer-loop convergence.

In this section, we analyze the complexity of each of the three components we discussed. Let $T$ denote the outer-loop complexity, $K$  the inner-loop complexity, and $N$  the number of iterations for $\mathcal{M}$ (expected number if $\mathcal{M}$ is randomized) to solve subproblem  (\ref{subprob}). The total complexity of Algorithm \ref{catalyst ncc 1} is computed by multiplying $K, T$ and $N$. Later, we will provide a guideline for choosing parameter $\tau$ to achieve the best complexity, given an algorithm $\mathcal{M}$.

\begin{theorem}[Outer loop] \label{THM CATALYST NCSC}

Suppose function $f$ is NC-SC with strong convexity parameter $\mu$ and L-Lipschitz smooth. If we choose $\alpha_t = \frac{\mu^5}{504L^5}$ for $t>0$ and $\alpha_0 = \frac{\mu^5}{576\max\{1,L^7\}}$, the output $\hat{x}_T$ from Algorithm \ref{catalyst ncc 1} satisfies
\begin{align} \label{ncsc out complexity}
    \mathbb{E}\|\nabla\Phi(\hat{x}_T)\|^2 \leq 
    \frac{268L}{5T}\Delta + \frac{28L}{5T}D_y^0,
\end{align}
where $\Delta = \Phi(x_0) -\inf_{x}\Phi(x), D_y^0 = \|y_0 - y^*(x_0)\|^2 $ and  $y^*(x_0)=\argmax_{y\in\mathbb{R}^{d_2}} f(x_0, y)$. %Furthermore, there exists some point $x_i$ in $\{x_t\}_{t=1}^T$ such that $\|\nabla\Phi(x_i)\|^2$ is no greater than the right hand side of (\ref{ncsc out complexity}).
\end{theorem}

This theorem implies that the algorithm finds an $\epsilon$ stationary point of $\Phi$ after inexactly solving (\ref{auxiliary prob ncc}) for $T = O\left( L(\Delta+D_y^0)\epsilon^{-2} \right)$ times. The dependency on $D_y^0$ can be eliminated if we select the initialization $y_0$ close enough to $y^*(x_0)$,  which only requires an additional logarithmic cost by maximizing a strongly concave function. %Noticeable, since the auxiliary problem is $(L,\mu)$-SC-SC, the complexity of applying Algorithm \ref{catalyst scsc} to solve the auxiliary problem  $\Tilde{O}\left(\Lambda^{\mathcal{M}}_{L, \mu, \Tilde{L}}(\tau)\sqrt{\frac{\mu+\tau}{\mu}}\right)$ by Theorem \ref{thm scsc total}, where $\Tilde{L}$ is the Lipchitz smoothness constant of the auxiliary problem. Now we are able to present the total complexity of Algorithm \ref{catalyst ncc}.

\begin{theorem}[Inner loop] \label{thm catalyst scsc}
Under the same assumptions in Theorem \ref{THM CATALYST NCSC}, if we choose $\epsilon^t_k = \frac{\sqrt{2}\mu}{2}(1-\rho)^k\gap_{\hat{f}_t}(x^t_{0}, y^t_{0})$ with $\rho < \sqrt{q}=\sqrt{\frac{\mu}{\mu+\tau}}$,  we have %$\|\nabla \hat{f}_t(x_{t,k}, y_{t,k})\|^2\leq \left[  \frac{48L^4}{\mu_x^2\mu_y^2(\sqrt{q}-p)^2} + \frac{\sqrt{2}L^2}{\mu_x\mu_y}\right](1-\rho)^{k}\|\nabla \hat{f}_t(x_{t,0}, y_{t,0})\|^2$.
%\iffalse
\begin{align*} \nonumber
    \|\nabla \hat{f}_t(x^t_{k}, y^t_{k})\|^2
    \leq \automedpar{\frac{5508L^2}{\mu^2(\sqrt{q}-\rho)^2} + \frac{18\sqrt{2}L^2}{\mu}}(1-\rho)^{k}\|\nabla \hat{f}_t(x^t_{0}, y^t_{0})\|^2.
\end{align*}
%\fi
\end{theorem}

Particularly, setting $\rho = 0.9\sqrt{q}$, Theorem \ref{thm catalyst scsc} implies after inexactly solving (\ref{subprob}) for $K = \Tilde{O}\left(\sqrt{(\tau+\mu)/\mu}\log\frac{1}{\alpha_t}\right)$ times, the stopping criterion (\ref{ncc criteion}) is satisfied. This complexity decreases with $\tau$. However, we should not choose $\tau$ too small, because the smaller $\tau$ is, the harder it is for $\mathcal{M}$ to solve (\ref{subprob}). The following theorem captures the complexity for algorithm $\mathcal{M}$ to solve the subproblem.

\iffalse
\begin{theorem}[Solving ($\star\star$)]
Under the same assumption as Theorem \ref{thm ncsc out}, with point $(x_0, y_0)$ we initialize $y$ with some algorithm $\mathcal{A}$ to find $\Tilde{y}_0$ such that $\|\Tilde{y}_0-y^*(x_0)\|\leq \frac{\epsilon}{160L}$, and we apply Algorithm \ref{catalyst ncc} with initial point $(x_0, \Tilde{y}_0)$. To find an $\epsilon$-stationary point of $f$, the total number (expected number if $\mathcal{M}$ is stochastic) of gradient evaluations is
\begin{equation*}
     N = \Tilde{O}\left(\frac{\Lambda^{\mathcal{M}}_{L, \mu, \Tilde{L}}(\tau)L\Delta }{\epsilon^2}\sqrt{\frac{\mu_y+\tau}{\mu_y}} + T_{\mathcal{A}}\right),
\end{equation*}
where   $\Tilde{L}$ is the Lipchitz smoothness (or average/individual smooth) constant of the auxiliary problem and $ T_{\mathcal{A}}$ is the complexity of initialization.
\end{theorem}
\fi

% \vspace{-3mm}
\begin{theorem}[Complexity of solving subproblems ($\star\star$)] \label{thm catalyst inner}
Under the same assumptions in Theorem \ref{THM CATALYST NCSC} and the choice of $\epsilon_k^t$ in Theorem \ref{thm catalyst scsc},  the number of iterations (expected number of iterations if $\mathcal{M}$ is stochastic) for $\mathcal{M}$ to solve (\ref{subprob}) such that $\|\nabla \Tilde{f}_{t,k}(x, y)\|^2 \leq \epsilon^t_k$ is $$N = O \left(\Lambda^{\mathcal{M}}_{\mu, L}(\tau) \log\left(\frac{\max\{1,L,\tau\}}{\min\{1,\mu \}}  \right)\right).$$
\end{theorem}

The above result implies that the subproblems can be solved within constant iterations that only depends on $L, \mu, \tau$ and $\Lambda^{\mathcal{M}}_{\mu, L}$. This largely benefits from the use of warm-starting and stopping criterion with time-varying accuracy.  In contrast, other inexact proximal point algorithms in minimax optimization, such as \citep{yang2020catalyst, lin2020near}, fix the target accuracy, thus their complexity of solving the subproblems usually has an extra logarithmic factor in  $1/\epsilon$. 

% The main idea to prove this theorem is to show that the initial point $(x^t_{k-1},y^t_{k-1})$ is not "too far" from the optimal solution of (\ref{subprob}), and therefore we can find an inexact solution satisfying the stopping criterion with constant time only depending on $L, \tau$ and $\Lambda^{\mathcal{M}}_{\mu, L}$. 

The overall complexity of the algorithm follows immediately after combining the above three theorems: 
\begin{corollary} \label{THM CATALYST TOTAL}
Under the same assumptions in Theorem \ref{THM CATALYST NCSC} and setting in Theorem \ref{thm catalyst scsc},
the total number (expected number if $\mathcal{M}$ is randomized) of gradient evaluations for Algorithm \ref{catalyst ncc 1} to find an $\epsilon$-stationary point of $\Phi$, is
\begin{equation}
    \Tilde{O}\left(\frac{\Lambda^{\mathcal{M}}_{\mu, L}(\tau)L(\Delta+D_y^0) }{\epsilon^2}\sqrt{\frac{\mu+\tau}{\mu}} \right).
\end{equation}
\end{corollary}

In order to minimize the total complexity, we should choose the regularization parameter $\tau$ that minimizes $\Lambda^{\mathcal{M}}_{\mu, L}(\tau)\sqrt{\mu+ \tau}$.

\subsection{Specific Algorithms and Complexities} \label{subsec algorithms}

In this subsection, we discuss specific choices for $\mathcal{M}$ and the corresponding optimal choices of $\tau$, as well as the resulting total complexities for solving NC-SC problems.  

\paragraph{Catalyst-EG/OGDA algorithm.} 
%First, the complexity for initialization of $y_0$ with accelerated gradient ascent is $\Tilde{O}\left(\sqrt{\frac{L}{\mu}}\log\left(\frac{D_y^0}{\epsilon}\right)\right)$. 
When solving NC-SC minimax problems in the general setting,  we set $\mathcal{M}$ to be either extra-gradient method (EG) or optimistic gradient descent ascent (OGDA). Hence, we have $\Lambda^{\mathcal{M}}_{\mu, L}(\tau) = \frac{L+\max\{2L,\tau\}}{4\min\{L, \mu+\tau\}}$~\citep{tseng1995linear, gidel2018variational, azizian2020tight}. Minimizing $\Lambda^{\mathcal{M}}_{\mu, L}(\tau)\sqrt{\mu+\tau}$ yields that the optimal choice for $\tau$ is $L - \mu$. This leads to a total complexity of
%$$ \Tilde{O}\left(\frac{L^{\frac{3}{2}}\Delta}{\sqrt{\mu}\epsilon^2}+  \sqrt{\frac{L}{\mu}}\log\left(\frac{D_y^0}{\epsilon}\right)\right).$$
%$$ \Tilde{O}\left(\sqrt{\kappa}L(\Delta+D_y^0)\epsilon^{-2}\right).$$
\begin{equation}
   \Tilde{O}\left(\sqrt{\kappa}L(\Delta+D_y^0)\epsilon^{-2}\right).
\end{equation}

\begin{remark}
The above complexity matches the lower bound in Theorem \ref{THM:LB_NCSC_DETER}, up to a logarithmic factor in $L$ and $\kappa$. It improves over Minimax-PPA \citep{lin2020near} by $\log^2(1/\epsilon)$, GDA \citep{lin2020gradient} by $\kappa^{\frac{3}{2}}$ and therefore achieves the best of two worlds in terms of dependency on $\kappa$ and $\epsilon$. In addition, our Catalyst-EG/OGDA algorithm does not require the bounded domain assumption on $y$, unlike \citep{lin2020near}.
\end{remark}

\paragraph{Catalyst-SVRG/SAGA algorithm.}  When solving NC-SC minimax problems in the  averaged smooth finite-sum setting,  we  set $\mathcal{M}$ to be either SVRG or SAGA. Hence, we have $\Lambda^{\mathcal{M}}_{\mu, L}(\tau) \propto n+ \big(\frac{L+\sqrt{2}\max\{2L,\tau\}}{\min\{L, \mu+\tau\}}\big)^2$~\citep{balamurugan2016stochastic}\footnote{Although \cite{balamurugan2016stochastic} assumes individual smoothness, their analysis can be extended to average smoothness.}. Minimizing $\Lambda^{\mathcal{M}}_{\mu, L}(\tau)\sqrt{\mu+\tau}$, the best choice for $\tau$ is (proportional to) $\max\Big\{{\frac{L}{\sqrt{n}}}-\mu, 0 \Big\}$, which leads to the total complexity of 
%$$\tilde{O}\left(\left(n+n^{\frac{3}{4}}\sqrt{\frac{L}{\mu}}  \right)\frac{L\Delta}{\epsilon^2} + \left(n+n^{\frac{3}{4}}\sqrt{\frac{L}{\mu}}\right)\log\left(\frac{D_y^0}{\epsilon}\right) \right).$$
%$$\tilde{O}\left(\left(n+n^{\frac{3}{4}}\sqrt{\kappa}  \right)L(\Delta+D_y^0)\epsilon^{-2}  \right).$$
\begin{equation}
  \tilde{O}\left(\left(n+n^{\frac{3}{4}}\sqrt{\kappa}  \right)L(\Delta+D_y^0)\epsilon^{-2}  \right).  
\end{equation}

\begin{remark}
According to the lower bound established in Theorem \ref{THM:LB_NCSC_FS_AS}, the dependency on $\kappa$ in the above upper bound is nearly tight, up to logarithmic factors.   Recall that SREDA \citep{luo2020stochastic} and SREDA-boost \citep{xu2020enhanced} achieve the complexity of $\Tilde{O}\left(\kappa^{2} \sqrt{n} \epsilon^{-2}+n+(n+\kappa) \log (\kappa)\right)$ for $n \geq \kappa^{2}$ and $O\left(\left(\kappa^{2}+\kappa n\right) \epsilon^{-2}\right)$ for $n \leq \kappa^{2}$. Hence, our Catalyst-SVRG/SAGA algorithm attains better complexity in the regime $n\leq \kappa^4$. Particularly, in the critical regime $\kappa = \Omega(\sqrt{n})$ arising in statistical learning \citep{shalev2014understanding}, our algorithm performs strictly better. 

%  but it is potentially sub-optimal in $n$.
\end{remark}

\section{Conclusion}
\label{sec:main_conclusion}

In this work, we take an initial step towards understanding the fundamental limits of minimax optimization in the nonconvex-strongly-concave setting for both general and finite-sum cases, and bridge the gaps between lower and upper bounds. It remains interesting to investigate whether the dependence on $n$ can be further tightened in the complexity for finite-sum NC-SC minimax optimization. 

\clearpage

\bibliographystyle{plainnat}
\bibliography{ref-new}

\clearpage

\appendix

\section{Notations}

\label{sec:Apdx_notations}
For convenience, we summarize some of the notations used in the paper.
\begin{itemize}[itemsep=2pt]
    \item SC / C / NC / WC: strongly convex, convex, nonconvex, weakly-convex.
    \item FS: finite-sum.
    \item $L$-S: $L$-Lipschitz smooth. $L$-IS / AS: $L$-Lipschitz individual / averaged smoothness.
    \item SOTA: state-of-the-art, LB / UB: lower / upper bound
    \item FO / IFO: first-order oracle, incremental first-order oracle, denoted by $\mathbb{O}_{\mathrm{FO}}$ and $\mathbb{O}_{\mathrm{IFO}}$.
    \item $\mathcal{A}$:linear-span first-order algorithm class. 
    % \item $\mathcal{A}_{\texttt{D}}$ / $\mathcal{A}_{\texttt{R}}$: linear-span deterministic / randomized first-order algorithm class. 
    \item $\Phi(x)$, $\Psi(y)$: primal and dual functions of $f(x,y)$.
    \item $ \nabla_x f $, $ \nabla_y f $: gradients of a function $F$ with respect to $x$ and $y$. Also we set $\nabla f=\autopar{\nabla_x f, \nabla_y f}$.
    \item $ \nabla_{xx}^2 f $, $ \nabla_{xy}^2 f $, $ \nabla_{yx}^2 F $, $ \nabla_{yy}^2 f $: the Hessian of $F(x,y)$ with respect to different components.
    \item $ \{\mathbf{U}^{(i)}\}_{i=1}^n \in \mathrm{\mathbf{Orth}}(a,b,n) $: a matrix sequence where if for each $ i, j\in[1, n] $ and $ i\neq j $, $ \mathbf{U}^{(i)},\mathbf{U}^{(j)}\in\mathbb{R}^{a\times b} $ and  $ \mathbf{U}^{(i)}(\mathbf{U}^{(i)})^\top=\mathbf{I}\in\mathbb{R}^{a\times a} $ and $\mathbf{U}^{(i)}(\mathbf{U}^{(j)})^\top=\mathbf{0}\in\mathbb{R}^{a\times a}$. Sometimes we use $u^{(i)}\triangleq\mathbf{U}^{(i)}x$.
    \item $e_i$: unit vector with the $i$-th element as $1$.
    \item $0$: zero scalars or vectors.
    \item 
    $
    \mathcal{X}_k=\mathrm{Span}\{e_1,e_2,\cdots,e_k\}, \mathcal{Y}_k=\mathrm{Span}\{e_{d+1},e_d,\cdots,e_{d-k+2}\}, \mathcal{X}_0=\mathcal{Y}_0=\{0\}
    $.
    %\item $\triangleq$: denote, $\Rightarrow$: imply.
    \item $a\vee b\triangleq\max\autobigpar{a, b}$, $a\wedge b\triangleq\min\autobigpar{a, b}$.
    \item $\autonorm{\cdot}$: $\ell_2$-norm.
    \item $\mathbb{N}^+$: all positive integers.
    \item $\mathbb{N}$: all nonnegative integers.
    \item $\dom f$: the domain of a function $f$.
    \item $d_1, d_2\in\mathbb{N}^+$: dimension numbers of $x$ and $y$.
    \item $x_d$: the $d$-th coordinate of $x$, $x^t$: the variable $x$ in the $t$-th iteration (in Section \ref{sec:main_LB_NCSC} and Appendix \ref{sec:Apdx_THM_LB_NCSC} only)
\end{itemize}

\section{Useful Lemmas and Proofs of Section \ref{SEC:MAIN_PRELIM}}

\begin{lemma} [Lemma B.2 \citep{lin2020near}] \label{lin's lemma}
Assume $f(\cdot, y)$ is $\mu_x$-strongly convex for $\forall y\in\mathbb{R}^{d_2}$ and $f(x, \cdot)$ is $\mu_y$-strongly concave for $\forall x\in\mathbb{R}^{d_1}$ (we will later refer to this as $(\mu_x, \mu_y)$-SC-SC)) and $f$ is $L$-Lipschitz smooth. Then we have

\begin{enumerate}[label=\alph*)]
\item $y^*(x) = \argmax_{y\in\mathbb{R}^{d_2}} f(x,y)$ is $\frac{L}{\mu_y}$-Lipschitz;
\item $\Phi(x) = \max_{y\in\mathbb{R}^{d_2}} f(x,y)$ is $\frac{2L^2}{\mu_y}$-Lipschitiz smooth and $\mu_x$-strongly convex with $\nabla\Phi(x) = \nabla_xf(x, y^*(x))$;
\item $x^*(y) = \argmin_{x\in\mathbb{R}^{d_1}} f(x,y)$ is $\frac{L}{\mu_x}$-Lipschitz;
\item $\Psi(y) = \min_{x\in\mathbb{R}^{d_1}} f(x,y)$ is $\frac{2L^2}{\mu_x}$-Lipschitiz smooth and $\mu_y$-strongly concave with $\nabla\Psi(y) = \nabla_yf(x^*(y), y)$.
\end{enumerate}

\end{lemma}

\begin{lemma} \label{criterion relation}
Under the same assumptions as Lemma \ref{lin's lemma}, we have
\begin{enumerate}[label=\alph*)]
\item $\gap_f(x,y) \leq \frac{L^2}{\mu_y}\|x-x^*\|^2 + \frac{L^2}{\mu_x}\|y-y^*\|^2 $, where $(x^*, y^*)$ is the optimal solution to $\min_{x\in \mathbb{R}^{d_1}}\max_{y\in \mathbb{R}^{d_2}}\ f(x,y)$.
\item $\gap_f(x,y) \leq \frac{1}{2\mu_x}\|\nabla_x f(x,y)\|^2 + \frac{1}{2\mu_y}\|\nabla_y f(x,y)\|^2$.
\item $\frac{\mu_x}{2}\|x-x^*\|^2 + \frac{\mu_y}{2}\|y-y^*\|^2 \leq \gap_f(x,y). $
\item $\|\nabla_x f(x, y)\|^2 + \|\nabla_y f(x,y)\|^2 \leq 4L^2(\|x-x^*\|^2+\|y-y^*\|^2)$. 

\end{enumerate}
\end{lemma}
 
\begin{proof}
\begin{enumerate}[label=\alph*)]
\item Because $\Phi(x)$ is $\frac{2L^2}{\mu_y}$-smooth by Lemma \ref{lin's lemma} and $\nabla\Phi(x^*) = 0$, we have
$
    \Phi(x) - \Phi(x^*) \leq \frac{L^2}{\mu_y}\|x-x^*\|^2.
$
Similarly, because $\Psi(y)$ is $\frac{2L^2}{\mu_x}$-smooth and $\Psi(y^*) = 0$, we have $\Psi(y^*)-\Psi(y) \leq \frac{L^2}{\mu_x}\|y-y^*\|^2$. We reach the conclusion by noting that $\gap_f(x,y) = \Phi(x) - \Psi(y) $ and $\Phi(x^*) =\Psi(y^*)$.
\item Because $f(\cdot, y)$ is $\mu_x$-strongly-convex and $\nabla_x f(x^*(y), y) = 0$, we have $f(x,y) - \min_x f(x,y) \leq \langle \nabla_x f(x^*(y), y), x-x^*(y)\rangle + \frac{1}{2\mu_x}\|\nabla_x f(x, y)-\nabla_x f(x^*(y), y)\|^2\leq\frac{1}{2\mu_x}\|\nabla_x f(x, y)\|^2$. Similarly, we have $\max_y f(x,y) - f(x,y) \leq \frac{1}{2\mu_y}\|\nabla_y f(x,y)\|^2$. Then we note that $\gap_f(x,y)=\max_y f(x,y) - f(x,y)  +  f(x,y) - \min_x f(x,y)$.
\item Because $\Phi(x)$ is $\mu_x$ strongly-convex and $\nabla\Phi(x^*) = 0$, we have $\Phi(x)\geq \Phi(x^*) +\frac{\mu_x}{2}\|x-x^*\|^2$. Similarly, because $\Psi(y)$ is $\mu_y$ strongly-concave and $\nabla\Psi(y^*) = 0$, we have $\Psi(y^*)-\Psi(y) \geq \frac{\mu_y}{2}\|y-y^*\|^2$.
\item By definition of Lipschitz smoothness, $\|\nabla_x f(x, y)\|^2 = \|\nabla_x f(x, y)-\nabla_x f(x^*, y^*)\|^2 \leq L^2(\|x-x^*\|+\|y-y^*\|)^2\leq 2L^2(\|x-x^*\|^2+\|y-y^*\|^2)$ and $\|\nabla_y f(x, y)\|^2 = \|\nabla_y f(x, y)-\nabla_y f(x^*, y^*)\|^2 \leq L^2(\|x-x^*\|+\|y-y^*\|)^2\leq 2L^2(\|x-x^*\|^2+\|y-y^*\|^2)$.
\end{enumerate}
\end{proof}

\noindent\textbf{Proof of Proposition \ref{prop regularized smooth}}

\begin{proof}
(a) and (b) directly follow from the definition of averaged smoothness and individual smoothness. \\

\noindent(c) Denote 
$$\Bar{f}(x, y)=f(x,y) + \frac{\tau_x}{2}\|x-\Tilde{x}\|^2 - \frac{\tau_y}{2}\|y-\Tilde{y}\|^2 = \frac{1}{n}\sum_{i=1}^n \left[f_i(x,y) + \frac{\tau_x}{2}\|x-\Tilde{x}\|^2 - \frac{\tau_y}{2}\|y - \Tilde{y}\|^2 \right] \triangleq \frac{1}{n}\sum_{i=1}^n \Bar{f}_i(x, y),$$
where $\Bar{f}_i(x, y) = f_i(x,y) + \frac{\tau_x}{2}\|x-\Tilde{x}\|^2 - \frac{\tau_y}{2}\|y - \Tilde{y}\|^2$. Note that for any $(x_1, y_1)$ and $(x_2, y_2)$,
\begin{align*}
    \|\nabla_x \Bar{f}_i(x_1, y_1) - \nabla_x \Bar{f}_i(x_2, y_2) \|^2\leq 2\|\nabla_x f_i(x_1, y_1) - \nabla_x f_i(x_2, y_2)\|^2 + 2\tau_x^2\|x_1-x_2\|^2,\\
    \|\nabla_y \Bar{f}_i(x_1, y_1) - \nabla_y \Bar{f}_i(x_2, y_2) \|^2\leq 2\|\nabla_y f_i(x_1, y_1) - \nabla_y f_i(x_2, y_2)\|^2 + 2\tau_y^2\|y_1-y_2\|^2. 
\end{align*}
Therefore, 
\begin{align*}
    \frac{1}{n}\sum_{i=1}^n\autonorm{\nabla \Bar{f}_i(x_1,y_1)-\nabla \Bar{f}_i(x_2,y_2)}^2\leq& \frac{2}{n}\sum_{i=1}^n\autonorm{\nabla f_i(x_1,y_1)-\nabla f_i(x_2,y_2)}^2 + 2[\tau_x^2\|x_1-x_2\|^2 +\tau_y^2\|y_1-y_2\|^2]\\
    \leq  &\left(2L^2+2\max\{\tau_x^2,\tau_y^2\}\right) \left(\|x_1-x_2\|^2+\|y_1-y_2\|^2  \right).
\end{align*}
\end{proof}

An important trick to transform the basic hard instance into the final hard instance is scaling, which will preserve the smoothness of the original function while extend the domain of the function to a high dimension, i.e., enlarging $d$, which helps to increase the lower bound. The properties of scaling is summarized in the following lemma.

\begin{lemma}[Scaling and Smoothness]
\label{lm:scaling}
    For a function $\Bar{g}(x,y)$ defined on $\mathbb{R}^{d_1}\times\mathbb{R}^{d_2}$, if $\Bar{g}$ is $L$-smooth, then for the following scaled function:
    \begin{equation}
        g(x,y)=\eta^2\Bar{g}\autopar{\frac{x}{\eta},\frac{y}{\eta}},
    \end{equation}
    then $g$ is also $L$-smooth. Furthermore if the function $\Bar{g}$ has a finite-sum form: $\Bar{g}(x,y)=\frac{1}{n}\sum_{i=1}^n\Bar{g}_i(x,y)$, if $\{\Bar{g}_i\}_{i=1}^n$ is $L$-averaged smooth, then for the following functions:
    \begin{equation}
        g_i(x,y)=\eta^2\Bar{g}_i\autopar{\frac{x}{\eta},\frac{y}{\eta}},
        \quad
        \text{and}
        \quad
        g(x,y)=
        \frac{1}{n}\sum_{i=1}^n g_i(x,y)
        =
        \frac{1}{n}\sum_{i=1}^n\eta^2\Bar{g}_i\autopar{\frac{x}{\eta},\frac{y}{\eta}},
    \end{equation}
    $\{g_i\}_{i=1}^n$ is also $L$-averaged smooth. If we further assume $\{\Bar{g}_i\}_{i=1}^n$ is $L$-individually smooth, then $\{g_i\}_{i=1}^n$ is also $L$-individually smooth.
\end{lemma}

\begin{proof}
    For the first statement, note that $\nabla g\autopar{x,y}=\eta\nabla\Bar{g}\autopar{\frac{x}{\eta},\frac{y}{\eta}}$, so for any $(x_1,y_1), (x_2,y_2)\in\mathbb{R}^{d_1}\times\mathbb{R}^{d_2}$,
    \begin{equation}
        \begin{split}
            \autonorm{\nabla_x g(x_1,y_1)-\nabla_x g(x_2,y_2)}
            =\ &
            \autonorm{\eta\nabla_x g\autopar{\frac{x_1}{\eta},\frac{y_1}{\eta}}-\eta\nabla_x g\autopar{\frac{x_2}{\eta},\frac{y_2}{\eta}}}\\
            \leq\ &
            \eta L\autopar{\autonorm{\frac{x_1}{\eta}-\frac{x_2}{\eta}}+\autonorm{\frac{y_1}{\eta}-\frac{y_2}{\eta}}}
            =
            L\autopar{\autonorm{x_1-x_2}+\autonorm{y_1-y_2}},
        \end{split}
    \end{equation}
    similar conclusion also holds for $\nabla_y g$, which verifies the first conclusion. 
    
    For the averaged smooth finite-sum statement, note that $\nabla g_i\autopar{x,y}=\eta\nabla\Bar{g}_i\autopar{\frac{x}{\eta},\frac{y}{\eta}}$, so for any $(x_1,y_1), (x_2,y_2)\in\mathbb{R}^{d_1}\times\mathbb{R}^{d_2}$,
    \begin{equation}
        \begin{split}
            &\mathbb{E}\automedpar{
            \autonorm{\nabla g_i(x_1,y_1)-\nabla g_i(x_2,y_2)}^2
            }\\
            =\ &
            \mathbb{E}\automedpar{\autonorm{\eta\nabla\Bar{g}_i\autopar{\frac{x_1}{\eta}, \frac{y_1}{\eta}}-\eta\nabla\Bar{g}_i\autopar{\frac{x_2}{\eta}, \frac{y_2}{\eta}}}^2}\\
            =\ &
            \eta^2\mathbb{E}\automedpar{\autonorm{\nabla\Bar{g}_i\autopar{\frac{x_1}{\eta}, \frac{y_1}{\eta}}-\nabla\Bar{g}_i\autopar{\frac{x_2}{\eta}, \frac{y_2}{\eta}}}^2}\\
            \leq\ &
            \eta^2L^2\autopar{\autonorm{\frac{x_1}{\eta}-\frac{x_2}{\eta}}^2+\autonorm{\frac{y_1}{\eta}-\frac{y_2}{\eta}}^2}
            =
            L^2\autopar{\autonorm{x_1-x_2}^2+\autonorm{y_1-y_2}^2},
        \end{split}
    \end{equation}
    so $\{g_i\}_{i=1}^n$ is $L$-averaged smooth. 
    
    For the individually smooth case statement, note that each $g_i$ is a scaled version of $\Bar{g}_i$, which is $L$-smooth, by the conclusion for the first statement, it implies that $g_i$ is also $L$-smooth, which concludes the proof.
\end{proof}

\section{Proof of NC-SC Lower Bound}
\label{sec:Apdx_THM_LB_NCSC}
Similar to Section \ref{sec:main_LB_NCSC} in the main text, here in this section only, we denote $x_d$ as the $d$-th coordinate of $x$ and $x^t$ as the variable $x$ in the $t$-th iteration.

\subsection{Deterministic NC-SC Lower Bound}

We start from the proof several important lemmas, then proceed to the analysis of Theorem \ref{THM:LB_NCSC_DETER}.

\subsubsection{Proof of Lemma \ref{LM:NCSC_LB_F_D}}
\label{sec:Apdx_LM_NCSC_LB_F_D}

\begin{proof}
	Recall the definition of $F_d$ in \eqref{eq:LB_hard_instance_deter}, define $\Gamma_d(x)\triangleq\sum_{i=1}^{d}\Gamma(x_i)$, note that $x_i^2=x^\top e_i e_i^\top x$, and
	\begin{equation}
		\begin{split}
			\nabla_xF_d(x,y;\lambda,\alpha)
			=\ &
			\lambda_1 B_d\top y-
			\frac{\lambda_1^2\sqrt{\alpha}}{2\lambda_2}e_1+
			\frac{\lambda_1^2\alpha}{2\lambda_2}\nabla\Gamma_d(x)-
			\frac{\lambda_1^2\alpha}{2\lambda_2}e_{d+1}e_{d+1}^\top x
			\\
			\nabla_yF_d(x,y;\lambda,\alpha)
			=\ &
			\lambda_1 B_dx-2\lambda_2y,
		\end{split}
	\end{equation}
	where $ \nabla\Gamma_d(x)=(\nabla\Gamma(x_1), \nabla\Gamma(x_2), \cdots, \nabla\Gamma(x_d))^\top $. Then for the matrix norm of $B_d$, note that $\alpha\in\automedpar{0,1}$ and
	\begin{equation}
	    \begin{split}
	        &\autonorm{B_dx}=\sqrt{x_{d+1}^2+(x_d-x_{d+1})^2+\cdots+(x_1-x_2)^2+(\sqrt[4]{\alpha}x_1)^2}\\
    	    \leq\ &
    	    \sqrt{x_{d+1}^2+2\autopar{x_d^2+x_{d+1}^2+x_{d-1}^2+x_{d}^2+\cdots+x_2^2+x_3^2+x_1^2+x_2^2}+x_1^2}\\
    	    \leq\ &
    	    \sqrt{4\autopar{x_{d+1}^2+x_{d}^2+x_{d-1}^2+\cdots+x_2^2+x_1^2}}
    	    =
    	    2\autonorm{x},
	    \end{split}
	\end{equation}
	similarly we have $\autonorm{B_d^Ty}\leq 2\autonorm{y}$. Denote $C_\gamma\triangleq 360,$\footnote{The choice of $C_\gamma$ follows the setting in \citep[Proposition 3.11]{zhou2019lower}, which is an upper bound of the Lipschitz smoothness parameter of $ \Gamma_d(x) $ in \cite[Lemma 2]{carmon2019lowerII}.} so because $0\leq\alpha\leq 1$ and $\|B_d\|\leq 2$, we have ($\|\cdot\|$ here denotes the spectral norm of a matrix)
	\begin{equation}
		\|\nabla_{xx}^2F_d\|\leq \frac{\lambda_1^2}{2\lambda_2}(C_\gamma\alpha+\alpha)\leq \frac{400\lambda_1^2\alpha}{2\lambda_2}=\frac{200\lambda_1^2\alpha}{\lambda_2},
		\quad
		\|\nabla_{xy}^2F_d\|\leq 2\lambda_1,
		\quad
		\|\nabla_{yx}^2F_d\|\leq 2\lambda_1,
		\quad
		\|\nabla_{yy}^2F_d\|= 2\lambda_2,
	\end{equation}
	which proves the first two statements (i) and (ii). 
	
	For (iii),  due to the structure of $ B_d $ and concerning the activation status defined in $\mathcal{X}_k$ and $\mathcal{Y}_k$, it is easy to verify that if $ x\in\mathcal{X}_{k_1}, y\in\mathcal{Y}_{k_2} $ for $k_1, k_2\in\mathbb{N}$ and $k_1, k_2\leq d$, we have 
	$$B_dx\in\mathcal{Y}_{k_1},\quad B_d^\top y\in\mathcal{X}_{k_2+1}.$$ 
	Since the remaining components in the gradient do not affect the activation with the initial point $ (0,0)\in\mathbb{R}^{d+1}\times\mathbb{R}^{d+2} $, this proves (iii).
	
	For (iv), by substituting the parameter settings, we have $\frac{200\lambda_1^2\alpha}{\lambda_2}= L$, $ 2\lambda_1=L $ and $ 2\lambda_2=\mu $, so the function $F_d$ is $\mu$-strongly concave in $y$ and $L$-Lipschitz smooth, which concludes the proof.
\end{proof}

\subsubsection{Proof of Lemma \ref{LM:NCSC_LB_PHI}}
\label{sec:Apdx_LM_NCSC_LB_PHI}
\begin{proof}
	Recall the primal function $\Phi_d$ of $F_d$ \eqref{eq:LB_hard_instance_deter}:
	\begin{equation}
		\Phi_d(x;\lambda,\alpha)
		=
		\underbrace{\frac{\lambda_1^2}{2\lambda_2}
		\left(
		\frac{1}{2}x^\top A_dx-
		\sqrt{\alpha}x_1+
		\frac{\sqrt{\alpha}}{2}+
		\alpha\sum_{i=1}^{d}\Gamma(x_i)\right)}_{\triangleq\Phi_{d1}(x)}+
		\underbrace{\frac{(1-\alpha)\lambda_1^2}{4\lambda_2}x_{d+1}^2}_{\triangleq\Phi_{d2}(x)}.
	\end{equation}
	
	For the first statement, because $ x_d=x_{d+1}=0 $, we have
	\begin{equation}
		\nabla\Phi_d(x;\lambda,\alpha)
		=
		\nabla\Phi_{d1}(x;\lambda,\alpha)+\nabla\Phi_{d2}(x;\lambda,\alpha)
		=
		\nabla\Phi_{d1}(x;\lambda,\alpha),
	\end{equation}
	which corresponds to the hard instance in \cite[Equation 9]{carmon2019lowerII} with an extra coefficient $\frac{\lambda_1^2}{2\lambda_2}$, then we apply \cite[Lemma 3]{carmon2019lowerII} therein to attain the desired large gradient norm result, i.e.
	\begin{equation}
	    \autonorm{\nabla\Phi_d(x;\lambda,\alpha)}
	    \geq
	    \frac{\lambda_1^2}{2\lambda_2}\times \frac{\alpha^{\frac{3}{4}}}{4}
	    =
	    \frac{\lambda_1^2}{8\lambda_2}\alpha^{\frac{3}{4}}.
	\end{equation}
	
	For the second statement, we have
	\begin{equation}
		\begin{split}
			&\Phi_d(0;\lambda,\alpha)-\inf_{x\in\mathbb{R}^{d+1}}\Phi_d(x;\lambda,\alpha)\\
			=\ &
			\Phi_{d1}(0;\lambda,\alpha)-\inf_{x\in\mathbb{R}^{d+1}}\automedpar{\Phi_{d1}(x;\lambda,\alpha)+\Phi_{d2}(x;\lambda,\alpha)}\\
			\leq\ &
			\Phi_{d1}(0;\lambda,\alpha)-\inf_{x\in\mathbb{R}^{d+1}}\Phi_{d1}(x;\lambda,\alpha)\\
			\leq\ &
			\frac{\lambda_1^2}{2\lambda_2}\left(\frac{\sqrt{\alpha}}{2}+10\alpha d\right),
		\end{split}
	\end{equation}
	where the first inequality uses that $\Phi_{d2}(x;\lambda,\alpha)\geq 0$ because $ \alpha\in[0,1] $, and the last inequality applies \cite[Lemma 4]{carmon2019lowerII}, which proves the second statement.
\end{proof}

\subsubsection{Proof of Theorem \ref{THM:LB_NCSC_DETER}}
\label{sec:Apdx_THM_LB_NCSC_DETER}

The complexity for deterministic nonconvex-strongly-concave problems is defined as

\begin{equation}
    \begin{split}
        \mathrm{Compl}_\epsilon
        \autopar{
        \mathcal{F}_{\mathrm{NCSC}}^{L,\mu,\Delta},\mathcal{A},\mathbb{O}_{\mathrm{FO}}
        }
        \triangleq\ &
        \underset{f\in\mathcal{F}_{\mathrm{NCSC}}^{L,\mu,\Delta}}{\sup}\ 
    	\underset{\mathtt{A}\in{\mathcal{A}\autopar{\mathbb{O}_{\mathrm{FO}}}}}{\inf}\ 
    	T_{\epsilon}(f,\mathtt{A})
    	\\
    	=\ &
    	\underset{f\in\mathcal{F}_{\mathrm{NCSC}}^{L,\mu,\Delta}}{\sup}\ 
    	\underset{\mathtt{A}\in{\mathcal{A}\autopar{\mathbb{O}_{\mathrm{FO}}}}}{\inf}\ 
    	\inf
    	\autobigpar{T\in\mathbb{N}\ \Big|\ \autonorm{\nabla \Phi\autopar{x^T}}\leq\epsilon}.
    \end{split}
\end{equation}

As a helper lemma, we first discuss the primal function of the scaled hard instance.

\begin{lemma}[Primal of the Scaled Hard Instance]
    \label{lm:Primal_Scaled_Deter_Hard_Instance}
    With the function $F_d$ defined in \eqref{eq:LB_hard_instance_deter}, $\Phi_d$ defined in \eqref{eq:LB_hard_instance_deter_Phi} and any $\eta\in\mathbb{R}$, for the following function:
    \begin{equation}
        f(x,y)=\eta^2F_d\left(\frac{x}{\eta},\frac{y}{\eta};\lambda,\alpha\right),
    \end{equation}
    then for its primal function $\Phi(x)\triangleq\max_{y\in\mathbb{R}^{d+2}}f(x,y)$, we have
    \begin{equation}
        \Phi(x)=\eta^2\Phi_d\autopar{\frac{x}{\eta};\lambda,\alpha}.
    \end{equation}
\end{lemma}

\begin{proof}
    Check the scaled function,
    \begin{equation}
    \label{eq:LB_hard_instance_deter_scaled}
	\begin{split}
	    & f(x,y)\\
	    =\ &
    	\eta^2
    	\Bigg(
    	\lambda_1\autoprod{B_d\frac{x}{\eta},\frac{y}{\eta}}-
    	\lambda_2\autonorm{\frac{y}{\eta}}^2-\frac{\lambda_1^2\sqrt{\alpha}}{2\lambda_2}\autoprod{e_1,\frac{x}{\eta}}+
    	\frac{\lambda_1^2\alpha}{2\lambda_2}\sum_{i=1}^{d}\Gamma\autopar{\frac{x_i}{\eta}}
    	-\frac{\lambda_1^2\alpha}{4\lambda_2}\autopar{\frac{x_{d+1}}{\eta}}^2+
    	\frac{\lambda_1^2\sqrt{\alpha}}{4\lambda_2}
    	\Bigg)\\
    	=\ &
    	\lambda_1\autoprod{B_dx,y}-
    	\lambda_2\autonorm{y}^2
    	+
    	\eta^2
    	\Bigg(
    	-\frac{\lambda_1^2\sqrt{\alpha}}{2\lambda_2}\autoprod{e_1,\frac{x}{\eta}}+
    	\frac{\lambda_1^2\alpha}{2\lambda_2}\sum_{i=1}^{d}\Gamma\autopar{\frac{x_i}{\eta}}
    	-\frac{\lambda_1^2\alpha}{4\lambda_2}\autopar{\frac{x_{d+1}}{\eta}}^2+
    	\frac{\lambda_1^2\sqrt{\alpha}}{4\lambda_2}
    	\Bigg)
    	,
	\end{split}
\end{equation}
check the gradient over $y$ and set it to be $0$ to solve for $y^*(x)$, we have
\begin{equation}
    \nabla_y f(x,y^*(x))=\lambda_1B_dx-2\lambda_2y^*(x)=0
    \quad\Longrightarrow \quad
    y^*(x)=\frac{\lambda_1}{2\lambda_2}B_dx,
\end{equation}
so the primal function is
\begin{equation}
    \begin{split}
        & \Phi(x)=f\autopar{x,y^*(x)}\\
        =\ &
        \lambda_1\autoprod{B_dx,y^*(x)}-
    	\lambda_2\autonorm{y^*(x)}^2
    	+
    	\eta^2
    	\autopar{
    	-\frac{\lambda_1^2\sqrt{\alpha}}{2\lambda_2}\autoprod{e_1,\frac{x}{\eta}}
    	+
    	\frac{\lambda_1^2\alpha}{2\lambda_2}\sum_{i=1}^{d}\Gamma\autopar{\frac{x_i}{\eta}}
    	-
    	\frac{\lambda_1^2\alpha}{4\lambda_2}\autopar{\frac{x_{d+1}}{\eta}}^2
    	+
    	\frac{\lambda_1^2\sqrt{\alpha}}{4\lambda_2}
    	}\\
    	=\ &
    	\frac{\lambda_1^2}{4\lambda_2}\autonorm{B_dx}^2
    	+
    	\eta^2
    	\autopar{
    	-\frac{\lambda_1^2\sqrt{\alpha}}{2\lambda_2}\autoprod{e_1,\frac{x}{\eta}}+
    	\frac{\lambda_1^2\alpha}{2\lambda_2}\sum_{i=1}^{d}\Gamma\autopar{\frac{x_i}{\eta}}
    	-
    	\frac{\lambda_1^2\alpha}{4\lambda_2}\autopar{\frac{x_{d+1}}{\eta}}^2+
    	\frac{\lambda_1^2\sqrt{\alpha}}{4\lambda_2}
    	}\\
    	=\ &
    	\eta^2
    	\autopar{
    	\frac{\lambda_1^2}{4\lambda_2}\autonorm{B_d\frac{x}{\eta}}^2
    	-\frac{\lambda_1^2\sqrt{\alpha}}{2\lambda_2}\autoprod{e_1,\frac{x}{\eta}}+
    	\frac{\lambda_1^2\alpha}{2\lambda_2}\sum_{i=1}^{d}\Gamma\autopar{\frac{x_i}{\eta}}
    	-
    	\frac{\lambda_1^2\alpha}{4\lambda_2}\autopar{\frac{x_{d+1}}{\eta}}^2+
    	\frac{\lambda_1^2\sqrt{\alpha}}{4\lambda_2}
    	}\\
    	=\ &
    	\eta^2\Phi_d\autopar{\frac{x}{\eta};\lambda,\alpha},
        \end{split}
\end{equation}
which concludes the proof.
\end{proof}

Now we come to the formal statement and proof of the main theorem.

\begin{theorem}[Lower Bound for General NC-SC, Restate Theorem \ref{THM:LB_NCSC_DETER}] For any linear-span first-order algorithm $ \mathtt{A}\in\mathcal{A} $ and parameters $ L,\mu,\Delta>0 $, with a desired accuracy $ \epsilon>0 $, for the following function $ f:\mathbb{R}^{d+1}\times\mathbb{R}^{d+1}\rightarrow\mathbb{R} $:
\begin{equation}
	f(x,y)\triangleq\eta^2F_d\left(\frac{x}{\eta},\frac{y}{\eta};\lambda^*,\alpha\right),
\end{equation}
where $F_d$ is defined in \eqref{eq:LB_hard_instance_deter}, with a primal function $ \Phi(x)\triangleq\max_{y\in\mathbb{R}^{d+1}}f(x,y) $, for a small enough $ \epsilon>0 $ satisfying	
$$ \epsilon^2\leq\min\left(\frac{\Delta L}{64000}, \frac{\Delta L\sqrt{\kappa}}{38400}\right), $$ 
if we set
\begin{equation}
	\lambda^*=\left(\frac{L}{2},\frac{\mu}{2}\right), \quad
	\eta=\frac{16\mu}{L^2}\alpha^{-3/4}\epsilon, \quad
	\alpha=\frac{\mu}{100L}\in\automedpar{0,1}, \quad
	d=\left\lfloor\frac{\Delta L\sqrt{\kappa}}{12800}\epsilon^{-2}\right\rfloor\geq 3,
\end{equation}
we have 
\begin{itemize}
	\item The proposed function $ f\in\mathcal{F}_{\mathrm{NCSC}}^{L,\mu,\Delta} $.
	\item To obtain a point $ \hat{x}\in\mathbb{R}^{d+1} $ such that $ \|\nabla\Phi(\hat{x})\|\leq\epsilon $, the number of FO queries required by the algorithm $ \mathtt{A}\in\mathcal{A} $ is at least
	$2d-1=\Omega\left(\sqrt{\kappa} \Delta L \epsilon^{-2}\right)$, namely, 
	\begin{equation}
        \mathrm{Compl}_\epsilon
        \autopar{
        \mathcal{F}_{\mathrm{NCSC}}^{L,\mu,\Delta},\mathcal{A},\mathbb{O}_{\mathrm{FO}}
        }
        =
        \Omega\autopar{\sqrt{\kappa} \Delta L \epsilon^{-2}}.
    \end{equation}
\end{itemize}
\end{theorem}

\begin{proof}
	First, we verify the smoothness and strong concavity of the function $f$. According to Lemma \ref{LM:NCSC_LB_F_D}, $ \alpha\leq\frac{\mu}{100L} $ implies that $ F_d(x,y;\lambda^*,\alpha) $ is $ L $-smooth and $ \mu $-strongly concave in $ y $. Given that $ f $ is a scaled version of $ F_d $, by Lemma \ref{lm:scaling}, it is easy to verify that $ f $ is also $ L $-smooth and $ \mu $-strongly concave in $ y $.
	
	Then by Lemma \ref{lm:Primal_Scaled_Deter_Hard_Instance}, we have
	\begin{equation}
	    \Phi(x)=\eta^2\Phi_d\autopar{\frac{x}{\eta};\lambda^*,\alpha},
	\end{equation}
	where $\Phi_d$ is defined in \eqref{eq:LB_hard_instance_deter_Phi}. Next we check the initial primal function gap, by Lemma \ref{LM:NCSC_LB_PHI} and parameter substitution,
	\begin{equation}
		\Phi(0)-\inf_x\Phi(x)=\eta^2\left(\Phi_d(0)-\inf_x\Phi_d(x)\right)
		\leq
		\frac{\eta^2L^2}{4\mu}\left(\frac{\sqrt{\alpha}}{2}+10\alpha d\right)
		=
		\frac{64\mu}{L^2}\autopar{\frac{1}{2\alpha}+\frac{10d}{\sqrt{\alpha}}}\epsilon^2,
	\end{equation}
	by substituting $ \alpha $ and $ d $ into the RHS above, we have
	\begin{equation}
		\begin{split}
			\frac{64\mu}{L^2}\autopar{\frac{1}{2\alpha}+\frac{10d}{\sqrt{\alpha}}}\epsilon^2
			\leq\ &
			\frac{64\mu}{L^2}\left(\frac{50L}{\mu}+100\sqrt{\frac{L}{\mu}}\cdot\frac{\Delta L\sqrt{\kappa}}{12800}\epsilon^{-2}\right)\epsilon^2\\
			\leq\ &
			\frac{64}{L}\left(50+\frac{\Delta L}{128}\epsilon^{-2}\right)\epsilon^2
			\leq
			\frac{64}{L}\left(\frac{\Delta L}{64}\epsilon^{-2}\right)\epsilon^2=\Delta.
		\end{split}
	\end{equation}
	The second inequality holds because $ \epsilon $ above is set to be small enough than $\frac{\Delta L}{6400}$. We conclude that $ f\in\mathcal{F}_{\mathrm{NCSC}}^{L,\mu,\Delta} $.
	
	%Then reflect on the nonconvergence condition,
	We now discuss the lower bound argument. Based on Lemma \ref{LM:NCSC_LB_PHI} and the setting of $ \eta $, we have when $ x_d=x_{d+1}=0 $, 
	\begin{equation}
		\label{eq:NCSC_LB_nonconvergence}
		\left\|\nabla\Phi(x)\right\|
		=
		\eta\left\|\nabla \Phi_d\left(\frac{x}{\eta};\lambda^*,\alpha\right)\right\|
		\geq
		\frac{\eta L^2}{16\mu}\alpha^{3/4}=\epsilon.
	\end{equation}
	So starting from $ (x,y)=(0,0)\in\mathbb{R}^{d+1}\times\mathbb{R}^{d+2} $, we cannot get the primal stationarity convergence at least until $ x_d\neq 0 $. By the ``alternating zero-chain" mechanism\footnote{Also known as the ``Domino argument" in \cite{ibrahim2020linear}.} in Lemma \ref{LM:NCSC_LB_F_D}, each update with the linear-span algorithm interacting with the FO oracle call will activate exactly one coordinate alternatively between $ x $ and $ y $. Therefore the algorithm $ \mathtt{A} $ requires at least $ 2d-1 $ queries to FO to activate the $d$-th element of $x$, i.e., $x_d$, which implies the lower bound is (note that $\epsilon$ is small enough such that $d\geq 3$)
	\begin{equation}
	    2d-1=\Omega\autopar{\sqrt{\kappa} \Delta L \epsilon^{-2}},
	\end{equation}
	which concludes the proof. Notice that this argument works even for randomized algorithms, as long as they satisfy the linear-span assumption.
\end{proof}

\subsection{Averaged Smooth Finite-Sum NC-SC Lower Bound}
\label{sec:Apdx_LM_NCSC_LB_FS_AS}

Similar to the deterministic NC-SC case, here we still start from several important lemmas and proceed to the proof of Theorem \ref{THM:LB_NCSC_FS_AS}.

\subsubsection{Hard Instance Construction}
\label{sec:Apdx_hard_instance_FS_AS_properties}

Recall the (unscaled) hard instance in averaged smooth finite-sum case in \eqref{eq:LB_hard_instance_FS}: $H_d:\mathbb{R}^{d+2}\times\mathbb{R}^{d+1}\rightarrow\mathbb{R}$, $\Gamma_d^n:\mathbb{R}^{n(d+1)}\rightarrow\mathbb{R}$ and
\begin{equation}
    \begin{split}
        H_d(x,y;\lambda,\alpha)
    	\triangleq\ &
    	\lambda_1\autoprod{B_dx,y}-
    	\lambda_2\|y\|^2-\frac{\lambda_1^2\sqrt{\alpha}}{2\lambda_2}\autoprod{e_1,x}-
    	\frac{\lambda_1^2\alpha}{4\lambda_2}x_{d+1}^2+
    	\frac{\lambda_1^2\sqrt{\alpha}}{4\lambda_2},\\
    	\Gamma_d^n(x)
    	\triangleq\ &
    	\sum_{i=1}^n\sum_{j=i(d+1)-d}^{i(d+1)-1}\Gamma(x_j),
    \end{split}
\end{equation}
then $\bar{f}_i, \bar{f}: \mathbb{R}^{n(d+1)}\times\mathbb{R}^{n(d+2)}\rightarrow\mathbb{R}$, $\{\mathbf{U}^{(i)}\}_{i=1}^n \in \mathrm{\mathbf{Orth}}(d+1,n(d+1),n)$, $\{\mathbf{V}^{(i)}\}_{i=1}^n \in \mathrm{\mathbf{Orth}}(d+2,n(d+2),n)$ and
\begin{equation}
	\begin{split}
	    \bar{f}_i(x,y)
    	\triangleq\ &
    	H_d\autopar{\mathbf{U}^{(i)}x,\mathbf{V}^{(i)}y;\lambda,\alpha}+\frac{\lambda_1^2\alpha}{2n\lambda_2}\Gamma_d^n(x),\\
    	\bar{f}(x,y)
    	\triangleq\ &
    	\frac{1}{n}\sum_{i=1}^{n}\bar{f}_i(x,y)
    	=
    	\frac{1}{n}\sum_{i=1}^{n}\automedpar{H_d\autopar{\mathbf{U}^{(i)}x,\mathbf{V}^{(i)}y;\lambda,\alpha}+\frac{\lambda_1^2\alpha}{2n\lambda_2}\Gamma_d^n(x)}.
	\end{split}
\end{equation}
i.e., by denoting $u^{(i)}\triangleq\mathbf{U}^{(i)}x$ and note that $\autonorm{y}^2=\sum_{i=1}^{n}\autonorm{\mathbf{V}^{(i)}y}^2$,
\begin{equation}
    \label{eq:LB_AS_FS_Hard_Instance_Detailed}
    \begin{split}
        &\bar{f}(x,y)\\
        =\ &
        \frac{1}{n}\sum_{i=1}^{n}
        \Bigg[
        \lambda_1\autoprod{B_d\mathbf{U}^{(i)}x,\mathbf{V}^{(i)}y}-
    	\lambda_2\|\mathbf{V}^{(i)}y\|^2-\frac{\lambda_1^2\sqrt{\alpha}}{2\lambda_2}\autoprod{e_1,\mathbf{U}^{(i)}x}+
    	\frac{\lambda_1^2\alpha}{2n\lambda_2}\Gamma_d^n(x)-
    	\frac{\lambda_1^2\alpha}{4\lambda_2}\autopar{u^{(i)}_{d+1}}^2+
    	\frac{\lambda_1^2\sqrt{\alpha}}{4\lambda_2}\Bigg]\\
    	=\ &
    	-\frac{\lambda_2}{n}\|y\|^2
    	+
    	\frac{1}{n}\sum_{i=1}^{n}
        \Bigg[\lambda_1\autoprod{B_d\mathbf{U}^{(i)}x,\mathbf{V}^{(i)}y}-
    	\frac{\lambda_1^2\sqrt{\alpha}}{2\lambda_2}\autoprod{e_1,\mathbf{U}^{(i)}x}+
    	\frac{\lambda_1^2\alpha}{2n\lambda_2}\Gamma_d^n(x)-
    	\frac{\lambda_1^2\alpha}{4\lambda_2}\autopar{u^{(i)}_{d+1}}^2+
    	\frac{\lambda_1^2\sqrt{\alpha}}{4\lambda_2}\Bigg]
    	,
    \end{split}
\end{equation}
so $\bar{f}$ is $\frac{2\lambda_2}{n}$-strongly concave in $y$. Recall the gradient of $f_i$:
    \begin{equation}
        \begin{split}
            \nabla_x f_i(x,y)=\ &
            \lambda_1\autopar{\mathbf{U}^{(i)}}^\top B_d^\top \mathbf{V}^{(i)}y-
        	\frac{\lambda_1^2\sqrt{\alpha}}{2\lambda_2}\autopar{\mathbf{U}^{(i)}}^\top e_1+
        	\frac{\lambda_1^2\alpha}{2n\lambda_2}\nabla\Gamma_d^n(x)
        	-\frac{\lambda_1^2\alpha}{2\lambda_2}\autopar{\mathbf{U}^{(i)}}^\top e_{d+1}e_{d+1}^\top\mathbf{U}^{(i)}x,\\
        	\nabla_y f_i(x,y)=\ &
        	\lambda_1\autopar{\mathbf{V}^{(i)}}^\top B_d\mathbf{U}^{(i)}x-2\lambda_2\autopar{\mathbf{V}^{(i)}}^\top\mathbf{V}^{(i)}y,
        \end{split}
    \end{equation}
then we discuss the smoothness of $\{\bar{f}_i\}_i$.

\begin{lemma}[Properties of $\bar{f}$]
    \label{lm:LB_FS_Hard_Instance_bar_f_base_AS}
    For $n\in\mathbb{N}^+$, $ L\geq 2n\mu>0 $, if we set
    \begin{equation}
        \lambda=\lambda^*=(\lambda_1^*,\lambda_2^*)=\autopar{\sqrt{\frac{n}{40}}L,\frac{n\mu}{2}} 
        \quad\text{ and }\quad
        \alpha=\frac{n\mu}{50L},
    \end{equation}
    then the function $\{\bar{f}_i\}_i$ is $L$-averaged smooth, and $ \bar{f}(x,\cdot) $ is $ \mu $-strongly concave for any fixed $ x\in\mathbb{R}^{d+1} $.
\end{lemma}

\begin{proof}
    For the strong concavity, note that $\bar{f}$ is $\frac{2\lambda_2}{n}$-strongly concave, so by substitution we have $\bar{f}$ is $\mu$-strongly concave in $y$. Then for the average smoothness, by definition, we have for any $(x_1,y_1), (x_2,y_2)\in\mathbb{R}^{d+1}\times\mathbb{R}^{d+1}$,
    \begin{equation}
        \begin{split}
            &\frac{1}{n}\sum_{i=1}^{n}\autonorm{\nabla f_i(x_1,y_1)-\nabla f_i(x_2,y_2)}^2\\
            =\ &
            \frac{1}{n}\sum_{i=1}^{n}\automedpar{\autonorm{\nabla_x f_i(x_1,y_1)-\nabla_x f_i(x_2,y_2)}^2+\autonorm{\nabla_y f_i(x_1,y_1)-\nabla_y f_i(x_2,y_2)}^2},
        \end{split}
    \end{equation}
    then note that $\Gamma_d^n$ and $\Gamma_d$ enjoys the same Lipschitz smoothness parameter as that of $\Gamma$, so we have
    \begin{equation}
        \begin{split}
            &\autonorm{\nabla_x f_i(x_1,y_1)-\nabla_x f_i(x_2,y_2)}^2\\
            \leq\ &
            4\autonorm{\lambda_1\autopar{\mathbf{U}^{(i)}}^\top B_d^\top \mathbf{V}^{(i)}\autopar{y_1-y_2}}^2+
        	4\autonorm{\frac{\lambda_1^2\alpha}{2n\lambda_2}\autopar{\nabla\Gamma_d^n(x_1)-\nabla\Gamma_d^n(x_2)}}^2\\
        	&\qquad\qquad\qquad\qquad\qquad\qquad\qquad +
        	4\autonorm{\frac{\lambda_1^2\alpha}{2\lambda_2}\autopar{\mathbf{U}^{(i)}}^\top e_{d+1}e_{d+1}^\top\mathbf{U}^{(i)}\autopar{x_1-x_2}}^2\\
        	=\ &
        	4\lambda_1^2\autonorm{B_d^\top \mathbf{V}^{(i)}\autopar{y_1-y_2}}^2
        	+
        	\frac{\lambda_1^4\alpha^2}{n^2\lambda_2^2}\autonorm{\nabla\Gamma_d^n(x_1)-\nabla\Gamma_d^n(x_2)}^2
        	+
        	\frac{\lambda_1^4\alpha^2}{\lambda_2^2}\autonorm{e_{d+1}e_{d+1}^\top\mathbf{U}^{(i)}\autopar{x_1-x_2}}^2\\
        	\leq\ &
        	16\lambda_1^2\autonorm{\mathbf{V}^{(i)}\autopar{y_1-y_2}}^2
        	+
        	\frac{C_\gamma^2\lambda_1^4\alpha^2}{n^2\lambda_2^2}\autonorm{x_1-x_2}^2
        	+
        	\frac{\lambda_1^4\alpha^2}{\lambda_2^2}\autonorm{\mathbf{U}^{(i)}\autopar{x_1-x_2}}^2,
        \end{split}
    \end{equation}
    and
    \begin{equation}
        \begin{split}
            &\autonorm{\nabla_y f_i(x_1,y_1)-\nabla_y f_i(x_2,y_2)}^2
            =\ 
            \autonorm{\lambda_1\autopar{\mathbf{V}^{(i)}}^\top B_d\mathbf{U}^{(i)}\autopar{x_1-x_2}-2\lambda_2\autopar{\mathbf{V}^{(i)}}^\top\mathbf{V}^{(i)}\autopar{y_1-y_2}}^2\\
            \leq\ &
            2\autonorm{\lambda_1\autopar{\mathbf{V}^{(i)}}^\top B_d\mathbf{U}^{(i)}\autopar{x_1-x_2}}^2
            +
            2\autonorm{2\lambda_2\autopar{\mathbf{V}^{(i)}}^\top\mathbf{V}^{(i)}\autopar{y_1-y_2}}^2\\
            \leq\ &
            8\lambda_1^2\autonorm{\mathbf{U}^{(i)}\autopar{x_1-x_2}}^2
            +
            8\lambda_2^2\autonorm{\mathbf{V}^{(i)}\autopar{y_1-y_2}}^2,
        \end{split}
    \end{equation}
    so we have
    \begin{equation}
        \begin{split}
            &\frac{1}{n}\sum_{i=1}^{n}\autonorm{\nabla f_i(x_1,y_1)-\nabla f_i(x_2,y_2)}^2\\
            \leq\ &
            \frac{1}{n}\sum_{i=1}^{n}
            \automedpar{\autopar{16\lambda_1^2+8\lambda_2^2}\autonorm{\mathbf{V}^{(i)}\autopar{y_1-y_2}}^2
        	+
        	\autopar{\frac{\lambda_1^4\alpha^2}{\lambda_2^2}+8\lambda_1^2}\autonorm{\mathbf{U}^{(i)}\autopar{x_1-x_2}}^2+\frac{C_\gamma^2\lambda_1^4\alpha^2}{n^2\lambda_2^2}\autonorm{x_1-x_2}^2}\\
        	=\ &
        	\frac{1}{n}\autopar{16\lambda_1^2+8\lambda_2^2}\sum_{i=1}^{n}\automedpar{\autonorm{\mathbf{V}^{(i)}\autopar{y_1-y_2}}^2}+\frac{1}{n}\autopar{\frac{\lambda_1^4\alpha^2}{\lambda_2^2}+8\lambda_1^2}\sum_{i=1}^{n}\automedpar{\autonorm{\mathbf{U}^{(i)}\autopar{x_1-x_2}}^2}+\frac{C_\gamma^2\lambda_1^4\alpha^2}{n^2\lambda_2^2}\autonorm{x_1-x_2}^2\\
        	=\ &
        	\frac{1}{n}\autopar{16\lambda_1^2+8\lambda_2^2}\autonorm{y_1-y_2}^2+\frac{1}{n}\autopar{\frac{\lambda_1^4\alpha^2}{\lambda_2^2}+8\lambda_1^2}\autonorm{x_1-x_2}^2+\frac{C_\gamma^2\lambda_1^4\alpha^2}{n^2\lambda_2^2}\autonorm{x_1-x_2}^2\\
        	\leq\ &
        	\frac{1}{n}\max\autobigpar{16\lambda_1^2+8\lambda_2^2, \frac{C_\gamma^2\lambda_1^4\alpha^2}{n\lambda_2^2}+\frac{\lambda_1^4\alpha^2}{\lambda_2^2}+8\lambda_1^2}\autopar{\autonorm{x_1-x_2}^2+\autonorm{y_1-y_2}^2},
        \end{split}
    \end{equation}
    then note that $\alpha\in\automedpar{0,1}$ because we set $ L\geq 2n\mu \geq \frac{1}{50}n\mu$, so substitute parameters into the above, we have
    \begin{equation}
        \begin{split}
            & \frac{1}{n}\max\autobigpar{16\lambda_1^2+8\lambda_2^2, \frac{C_\gamma^2\lambda_1^4\alpha^2}{n\lambda_2^2}+\frac{\lambda_1^4\alpha^2}{\lambda_2^2}+8\lambda_1^2}\\
            =\ &
            \frac{1}{n}\max\autobigpar{16\lambda_1^2+2n^2\mu^2, \frac{4C_\gamma^2\lambda_1^4\alpha^2}{n^3\mu^2}+\frac{4\lambda_1^4\alpha^2}{n^2\mu^2}+8\lambda_1^2}\\
            \leq\ &
            \frac{1}{n}\max\autobigpar{16\lambda_1^2+2n^2\mu^2, 1000000\alpha^2\frac{\lambda_1^4}{n^3\mu^2}+8\lambda_1^2}\\
            =\ &
            \frac{1}{n}\max\autobigpar{\frac{16nL^2}{40}+2n^2\mu^2, 1000000\cdot\frac{n^2\mu^2}{2500L^2}\cdot\frac{n^2L^4}{1600n^3\mu^2}+\frac{8nL^2}{40}}\\
            \leq\ &
            \max\autobigpar{\frac{2L^2}{5}+\frac{L^2}{2}, \frac{L^2}{4}+\frac{L^2}{5}}\\
            \leq\ &
            \max\autobigpar{\frac{9L^2}{10}, \frac{9L^2}{20}}
            \leq L^2,
        \end{split}
    \end{equation}
    where the first inequality is attained by the computation with the value of $C_\gamma=360$, the second inequality comes from the assumption $L\geq 2n\mu\geq 2\sqrt{n}\mu$; the last equality is attained by parameter substitution, which verifies the conclusion.
\end{proof}

Next we discuss the primal function of the finite-sum hard instance.

\begin{lemma}[Primal of Averaged Smooth Finite-Sum Hard Instance]
    \label{lm:primal_hard_instance_FS_AS}
    For the function $\bar{f}=\frac{1}{n}\sum_{i=1}^n\bar{f}_i$ defined in \eqref{eq:LB_hard_instance_FS}, define $ \bar{\Phi}(x)\triangleq \max_y \bar{f}(x,y)$, then we have
    \begin{equation}
        \bar{\Phi}(x)=\frac{1}{n}\sum_{i=1}^{n}\bar{\Phi}_i(x),
        \quad \text{where}\quad
        \bar{\Phi}_i(x) \triangleq \Phi_d\autopar{\mathbf{U}^{(i)}x},
    \end{equation}
    while $\Phi_d$ is defined in \eqref{eq:LB_hard_instance_deter_Phi}.
\end{lemma}

\begin{proof}
    By the expression of $\bar{f}$ in \eqref{eq:LB_AS_FS_Hard_Instance_Detailed}, take the gradient over $y$ and set it as $0$, denote the maximizer as $y^*(x)$, we have
    \begin{equation}
        -\frac{2\lambda_2}{n}y^*(x)+\frac{1}{n}\sum_{i=1}^{n}
        \lambda_1\autopar{\mathbf{V}^{(i)}}^\top B_d\mathbf{U}^{(i)}x=0
        \quad\Longrightarrow\quad
        y^*(x)=\frac{\lambda_1}{2\lambda_2}\sum_{i=1}^{n}\autopar{\mathbf{V}^{(i)}}^\top B_d\mathbf{U}^{(i)}x,
    \end{equation}
    so we have
    \begin{equation}
        \begin{split}
            &\bar{\Phi}(x)=\bar{f}\autopar{x,y^*(x)}\\
            =\ &
            \frac{1}{n}\sum_{i=1}^{n}
            \automedpar{\frac{\lambda_1^2}{4\lambda_2}\autonorm{B_d\mathbf{U}^{(i)}x}^2-
        	\frac{\lambda_1^2\sqrt{\alpha}}{2\lambda_2}\autoprod{e_1,\mathbf{U}^{(i)}x}+
        	\frac{\lambda_1^2\alpha}{2n\lambda_2}\Gamma_d^n(x)-
        	\frac{\lambda_1^2\alpha}{4\lambda_2}\autopar{u^{(i)}_{d+1}}^2+
        	\frac{\lambda_1^2\sqrt{\alpha}}{4\lambda_2}}\\
            =\ &
            \frac{1}{n}\sum_{i=1}^{n}
            \automedpar{\frac{\lambda_1^2}{4\lambda_2}\autonorm{B_d\mathbf{U}^{(i)}x}^2-
        	\frac{\lambda_1^2\sqrt{\alpha}}{2\lambda_2}\autoprod{e_1,\mathbf{U}^{(i)}x}+
        	\frac{\lambda_1^2\alpha}{2n\lambda_2}\sum_{j=1}^{n}\Gamma_d\autopar{\mathbf{U}^{(j)}x}-
        	\frac{\lambda_1^2\alpha}{4\lambda_2}\autopar{u^{(i)}_{d+1}}^2+
        	\frac{\lambda_1^2\sqrt{\alpha}}{4\lambda_2}}\\
            =\ &
            \frac{1}{n}\sum_{i=1}^{n}
            \automedpar{\frac{\lambda_1^2}{4\lambda_2}\autonorm{B_d\mathbf{U}^{(i)}x}^2-
        	\frac{\lambda_1^2\sqrt{\alpha}}{2\lambda_2}\autoprod{e_1,\mathbf{U}^{(i)}x}+
        	\frac{\lambda_1^2\alpha}{2\lambda_2}\Gamma_d\autopar{\mathbf{U}^{(i)}x}-
        	\frac{\lambda_1^2\alpha}{4\lambda_2}\autopar{u^{(i)}_{d+1}}^2+
        	\frac{\lambda_1^2\sqrt{\alpha}}{4\lambda_2}}\\
        	=\ &
        	\frac{1}{n}\sum_{i=1}^{n}
            \automedpar{\frac{\lambda_1^2}{2\lambda_2}
            \autopar{\frac{1}{2}\autopar{\mathbf{U}^{(i)}x}^\top A_d\mathbf{U}^{(i)}x-
        	\sqrt{\alpha}\autoprod{e_1,\mathbf{U}^{(i)}x}+
        	\alpha\Gamma_d\autopar{\mathbf{U}^{(i)}x}
        	+
        	\frac{1-\alpha}{2}\autopar{u^{(i)}_{d+1}}^2+
        	\frac{\sqrt{\alpha}}{2}}
            }\\
            =\ &
            \frac{1}{n}\sum_{i=1}^{n}\Phi_d\autopar{\mathbf{U}^{(i)}x},
        \end{split}
    \end{equation}
    where the third equality follows from \eqref{eq:Gamma_nd_equivalent}, and $A_d$ and $\Phi_d$ are defined in \eqref{eq:matrix_A_d_definition} and \eqref{eq:LB_hard_instance_deter_Phi}, which concludes the proof.
\end{proof}

The above two lemmas proves the statements in Lemma \ref{lm:properties_hard_instance_FS_AS}. Before we present the main theorem, we first discuss the behavior of the scaled hard instance, which will be used in the final lower bound analysis.

\begin{lemma}[Primal of the Scaled Finite-Sum Hard Instance]
    \label{lm:Primal_Scaled_FS_AS_Hard_Instance}
    With the function $\bar{f}(x,y)$ and $\bar{f}_i(x,y)$ defined in \eqref{eq:LB_hard_instance_FS}, $ \bar{\Phi}(x)\triangleq \max_y \bar{f}(x,y)$, then for any $\eta\in\mathbb{R}$ and the following function:
    \begin{equation}
        f(x,y)=\frac{1}{n}\sum_{i=1}^{n}f_i(x,y)
        =\frac{1}{n}\sum_{i=1}^{n}\eta^2\bar{f}_i\autopar{\frac{x}{\eta},\frac{y}{\eta}},
    \end{equation}
    then for its primal function $\Phi(x)\triangleq\max_{y\in\mathbb{R}^{d+1}}f(x,y)$, we have
    \begin{equation}
        \Phi(x)=\frac{1}{n}\sum_{i=1}^{n}\Phi_i(x),
        \quad \text{where}\quad
        \Phi_i(x) =
        \eta^2\bar{\Phi}_i\autopar{\frac{x}{\eta}}
        =
        \eta^2\Phi_d\autopar{\frac{1}{\eta}\mathbf{U}^{(i)}x}.
    \end{equation}
\end{lemma}

\begin{proof}
    Based on \eqref{eq:LB_AS_FS_Hard_Instance_Detailed}, we can write out the formulation of $f$:
    \begin{equation}
        \label{eq:LB_AS_FS_Hard_Instance_Scaled_Detailed}
        \begin{split}
            &f(x,y)
            =
            \eta^2\bar{f}\autopar{\frac{x}{\eta},\frac{y}{\eta}}\\
            =\ &
            \eta^2\Bigg(
        	-\frac{\lambda_2}{n}\autonorm{\frac{y}{\eta}}^2
        	+
        	\frac{1}{n}\sum_{i=1}^{n}
            \Bigg[\lambda_1\autoprod{B_d\mathbf{U}^{(i)}\frac{x}{\eta},\mathbf{V}^{(i)}\frac{y}{\eta}}-
        	\frac{\lambda_1^2\sqrt{\alpha}}{2\lambda_2}\autoprod{e_1,\mathbf{U}^{(i)}\frac{x}{\eta}}+
        	\frac{\lambda_1^2\alpha}{2n\lambda_2}\Gamma_d^n\autopar{\frac{x}{\eta}}\\
        	&\qquad\qquad\qquad\qquad\qquad\qquad\qquad\qquad\qquad\quad
        	-
        	\frac{\lambda_1^2\alpha}{4\lambda_2}\autopar{\frac{u^{(i)}_{d+1}}{\eta}}^2+
        	\frac{\lambda_1^2\sqrt{\alpha}}{4\lambda_2}\Bigg]\Bigg)\\
        	=\ &
        	-\frac{\lambda_2}{n}\autonorm{y}^2
        	+
        	\frac{1}{n}\sum_{i=1}^{n}
            \lambda_1\autoprod{B_d\mathbf{U}^{(i)}x,\mathbf{V}^{(i)}y}
            +
        	\eta^2\Bigg(
        	\frac{1}{n}\sum_{i=1}^{n}
            \Bigg[-
        	\frac{\lambda_1^2\sqrt{\alpha}}{2\lambda_2}\autoprod{e_1,\mathbf{U}^{(i)}\frac{x}{\eta}}
        	+
        	\frac{\lambda_1^2\alpha}{2n\lambda_2}\Gamma_d^n\autopar{\frac{x}{\eta}}\\
        	&\qquad\qquad\qquad\qquad\qquad\qquad\qquad\qquad\qquad\quad
        	-
        	\frac{\lambda_1^2\alpha}{4\lambda_2}\autopar{\frac{u^{(i)}_{d+1}}{\eta}}^2+
        	\frac{\lambda_1^2\sqrt{\alpha}}{4\lambda_2}\Bigg]\Bigg)
        	,
        \end{split}
    \end{equation}
    check the gradient over $y$ and set it to be $0$ to solve for $y^*(x)$, we have
    \begin{equation}
        \nabla_y f(x,y^*(x))=-\frac{2\lambda_2}{n}y^*(x)+\frac{\lambda_1}{n}\sum_{i=1}^{n}
        \autopar{\mathbf{V}^{(i)}}^\top B_d\mathbf{U}^{(i)}x=0
        \quad\Longrightarrow \quad
        y^*(x)=\frac{\lambda_1}{2\lambda_2}\sum_{i=1}^{n}
        \autopar{\mathbf{V}^{(i)}}^\top B_d\mathbf{U}^{(i)}x,
    \end{equation}
    which implies that
    \begin{equation}
        \begin{split}
            \mathbf{V}^{(i)}y^*(x)
            =\ &
            \frac{\lambda_1}{2\lambda_2}\sum_{j=1}^{n}
        	\mathbf{V}^{(i)}\autopar{\mathbf{V}^{(j)}}^\top B_d\mathbf{U}^{(j)}x
        	=
        	\frac{\lambda_1}{2\lambda_2} B_d\mathbf{U}^{(i)}x\\
        	\autonorm{y^*(x)}^2
        	=\ &
        	\frac{\lambda_1^2}{4\lambda_2^2}\sum_{i=1}^{n}\autonorm{B_d\mathbf{U}^{(i)}x}^2,
        \end{split}
    \end{equation}
    so the primal function is
    \begin{equation}
        \begin{split}
            &\Phi(x)=f(x,y^*(x))
            =
            \eta^2\bar{f}\autopar{\frac{x}{\eta},\frac{y^*(x)}{\eta}}\\
        	=\ &
        	\frac{\lambda_1^2}{4\lambda_2n}\sum_{i=1}^{n}\autonorm{B_d\mathbf{U}^{(i)}x}^2
            +
        	\frac{\eta^2}{n}\sum_{i=1}^{n}
            \Bigg[-
        	\frac{\lambda_1^2\sqrt{\alpha}}{2\lambda_2}\autoprod{e_1,\mathbf{U}^{(i)}\frac{x}{\eta}}
        	+
        	\frac{\lambda_1^2\alpha}{2n\lambda_2}\Gamma_d^n\autopar{\frac{x}{\eta}}
        	-
        	\frac{\lambda_1^2\alpha}{4\lambda_2}\autopar{\frac{u^{(i)}_{d+1}}{\eta}}^2+
        	\frac{\lambda_1^2\sqrt{\alpha}}{4\lambda_2}\Bigg]\\
        	=\ &
        	\frac{\eta^2}{n}\sum_{i=1}^{n}
            \Bigg[
            \frac{\lambda_1^2}{4\lambda_2}\autonorm{B_d\mathbf{U}^{(i)}\frac{x}{\eta}}^2
            -
        	\frac{\lambda_1^2\sqrt{\alpha}}{2\lambda_2}\autoprod{e_1,\mathbf{U}^{(i)}\frac{x}{\eta}}
        	+
        	\frac{\lambda_1^2\alpha}{2n\lambda_2}\Gamma_d^n\autopar{\frac{x}{\eta}}
        	-
        	\frac{\lambda_1^2\alpha}{4\lambda_2}\autopar{\frac{u^{(i)}_{d+1}}{\eta}}^2+
        	\frac{\lambda_1^2\sqrt{\alpha}}{4\lambda_2}\Bigg]\\
        	=\ &
        	\frac{1}{n}\sum_{i=1}^{n}\autopar{\eta^2\Phi_d\autopar{\frac{1}{\eta}\mathbf{U}^{(i)}x}}
        	,
        \end{split}
    \end{equation}
    where the last equality directly applies the conclusion in Lemma \ref{lm:primal_hard_instance_FS_AS}, which concludes the proof.
\end{proof}

\subsubsection{Proof of Theorem \ref{THM:LB_NCSC_FS_AS}}
\label{sec:Apdx_THM_LB_NCSC_FS_AS}

Recall that the complexity for averaged smooth finite-sum nonconvex-strongly-concave problems is defined as

\begin{equation}
    \begin{split}
        \mathrm{Compl}_\epsilon\autopar{\mathcal{F}_{\mathrm{NCSC}}^{L,\mu,\Delta},\mathcal{A},\mathbb{O}_{\mathrm{IFO}}^{L,\mathrm{AS}}}
    	\triangleq\ &
        \underset{f\in\mathcal{F}_{\mathrm{NCSC}}^{L,\mu,\Delta}}{\sup}\ 
    	\underset{\mathtt{A}\in{\mathcal{A}\autopar{\mathbb{O}_{\mathrm{IFO}}^{L,\mathrm{AS}}}}}{\inf}\
    	\mathbb{E}\ T_{\epsilon}(f,\mathtt{A})\\
    	=\ &
    	\underset{f\in\mathcal{F}_{\mathrm{NCSC}}^{L,\mu,\Delta}}{\sup}\
    	\underset{\mathtt{A}\in{\mathcal{A}\autopar{\mathbb{O}_{\mathrm{IFO}}^{L,\mathrm{AS}}}}}{\inf}\
    	\mathbb{E}\ \inf
    	\autobigpar{T\in\mathbb{N}\ \Big|\ \autonorm{\nabla \Phi\autopar{x^T}}\leq\epsilon}.
    \end{split}
\end{equation}
Based on the discussion of the properties of the hard instance, we come to the final statement and proof of the theorem. 

\begin{theorem}[Lower Bound for Finite-Sum AS NC-SC, Restate Theorem \ref{THM:LB_NCSC_FS_AS}] 
For any linear-span first-order algorithm $ \mathtt{A}\in\mathcal{A} $, and parameters $ L,\mu,\Delta>0 $ with a desired accuracy $ \epsilon>0 $, for the following function $ f:\mathbb{R}^{(d+1)}\times\mathbb{R}^{(d+1)}\rightarrow\mathbb{R} $:
	\begin{equation}
		f_i(x,y)=\eta^2\bar{f}_i\left(\frac{x}{\eta},\frac{y}{\eta}\right),
		\quad
		f(x,y)=\frac{1}{n}\sum_{i=1}^{n}f_i(x,y)
	\end{equation}
	where $\bar{f}_i$ is defined as \eqref{eq:LB_hard_instance_FS} and $\autobigpar{\mathbf{U}^{(i)}}_{i=1}^n\in\mathrm{\mathbf{Orth}}\autopar{d+1,(d+1)n,n}$ is defined in \eqref{eq:LB_FS_Orthogonal_Matrix}, with its primal function $ \Phi(x)\triangleq\max_{y\in\mathbb{R}^{d+1}}f(x,y) $, for small enough $ \epsilon>0 $ satisfying	
	\begin{equation}
	    \epsilon^2\leq\min\autopar{
    	\frac{\sqrt{\alpha}L^2\Delta}{76800n\mu}, \frac{\alpha L^2\Delta}{1280n\mu}, \frac{L^2\Delta}{\mu}
    	},
	\end{equation}
	if we set $L\geq 2n\mu>0$ and
	\begin{equation}
		\lambda^*=\autopar{\sqrt{\frac{n}{40}}L,\frac{n\mu}{2}}, \quad
		\eta=\frac{160\sqrt{2n}\mu}{L^2}\alpha^{-\frac{3}{4}}\epsilon, \quad
		\alpha=\frac{n\mu}{50L}, \quad
		d=\left\lfloor\frac{\sqrt{\alpha}L^2\Delta}{25600n\mu}\epsilon^{-2}\right\rfloor\geq 3,
	\end{equation}
	we have 
	\begin{itemize}
		\item The function $ f\in\mathcal{F}_{\mathrm{NCSC}}^{L,\mu,\Delta} $, $ \{f_i\}_{i=1}^n $ is $ L $-averaged smooth.
		\item In the worst case, the algorithm $ \mathcal{A} $ requires at least $\Omega\autopar{n+\sqrt{n\kappa} \Delta L \epsilon^{-2}}$ IFO calls to attain a point $ \hat{x}\in\mathbb{R}^{d+1} $ such that $ \mathbb{E}\|\nabla\Phi(\hat{x})\|\leq\epsilon $, i.e.,
		\begin{equation}
		    \mathrm{Compl}_\epsilon\autopar{\mathcal{F}_{\mathrm{NCSC}}^{L,\mu,\Delta},\mathcal{A},\mathbb{O}_{\mathrm{IFO}}^{L,\mathrm{AS}}}=\Omega\autopar{n+\sqrt{n\kappa} \Delta L \epsilon^{-2}}.
		\end{equation}
	\end{itemize}
\end{theorem}

\begin{proof}
    We divide our proof into two cases.
    
    \paragraph{Case 1}
	The first case builds an $\Omega(n)$ lower bound from a special case of NC-SC function. Consider the following function: $x,y\in\mathbb{R}^d$ and
    \begin{equation}
        \label{eq:NCSC_FS_hard_instance_Case_1}
        h_i(x,y)\triangleq \theta\autoprod{v_i,x}+L\autoprod{x,y}-\frac{\mu}{2}\|y\|^2,
        \quad
        h(x,y)\triangleq\frac{1}{n}\sum_{i=1}^n h_i(x,y),
    \end{equation}
    where $\theta\leq\sqrt{\frac{2L^2n^2\Delta}{\mu d}}$, $0<\mu\leq L$,  the dimension number $d$ is set as a multiple of $n$, and $v_i\in\mathbb{R}^d$ is defined as 
    \begin{equation}
        v_i\triangleq
        \begin{bmatrix}
            0 & \cdots & 0 & 1 & \cdots & 1 & 0 & \cdots & 0
        \end{bmatrix}^\top,
    \end{equation}
    such that elements with indices from $\frac{i-1}{n}d+1$ to $\frac{i}{n}d$ are 1 and the others are all $0$, namely,  there are $\frac{d}{n}$ non-zero elements. 
    
    It is easy to see that the function $h_i$ is $\mu$-strongly convex and $L$-smooth in both $x$ and $y$. For the initial value gap, denote $\varphi\triangleq\max_y h$. We have
    \begin{equation}
        \varphi(x)
        =
        \frac{1}{n}\sum_{i=1}^n\left(\theta\autoprod{v_i,x}+\frac{L^2}{2\mu}\|x\|^2\right)
        =
        \frac{L^2}{2\mu}\|x\|^2+\frac{\theta}{n}\sum_{i=1}^n\autoprod{v_i,x},
    \end{equation}
    which is a strongly convex function, and its optimal point $x^*$ is 
    \begin{equation}
        x^*=-\frac{\mu\theta}{L^2n}\sum_{i=1}^n v_i,
        \quad
        \varphi^*=-\frac{\mu\theta^2}{2L^2n^2}\left\|\sum_{i=1}^n v_i\right\|^2.
    \end{equation}
    % so we have 
    Based on the setting of $\theta$,
    \begin{equation}
        \varphi(0)-\varphi^*=\frac{\mu\theta^2}{2L^2n^2}\left\|\sum_{i=1}^n v_i\right\|^2
        =
        \frac{\mu\theta^2d}{2L^2n^2}
        \leq
        \Delta.
    \end{equation}
    Hence,  we have $h\in\mathcal{F}_{\mathrm{NCSC}}^{L,\mu,\Delta}$. Then based on the expression of $\nabla_x h_i$ and $\nabla_y h_i$, we have that, starting from $(x,y)=(0,0)$ and denoting $\{i_1, i_2, \cdots, i_t\}$ as the index of IFO sequence for $t$ queries, then the output $(\hat{x}_t,\hat{y}_t)$ will be
    \begin{equation}
        \hat{x}_t,\hat{y}_t\in\mathrm{Span}\{v_{i_1}, v_{i_2}, \cdots, v_{i_t}\}.
    \end{equation}
    then note that each $v_i$ contains only $\frac{d}{n}$ non-zero elements, by the expression of the gradient of the primal function $\nabla \varphi$, we have that if $t\leq n/2$, then there must be at least $\frac{n}{2}\times\frac{d}{n}=\frac{d}{2}$ zero elements in $\hat{x}_t$, which implies that for $\epsilon^2\leq\frac{L^2\Delta}{\mu}$,
    \begin{equation}
        \|\nabla \varphi(\hat{x}_t)\|
        =
        \left\|\frac{L^2}{\mu}\hat{x}_t+\frac{\theta}{n}\sum_{i=1}^n v_i\right\|
        \geq
        \frac{\theta}{n}\sqrt{\frac{d}{2}}
        \geq
        \epsilon,
    \end{equation}
    where we follow the setting of $\theta$ above. So we proved that it requires $\Omega(n)$ IFO calls to find an $\epsilon$-stationary point.
    
    \paragraph{Case 2}
    The second case provides an $\Omega(\sqrt{n\kappa}\Delta L\epsilon^{-2})$ lower bound concerning the second term in the result. Throughout the case, we assume $L\geq 2n\mu>0$ as that in Lemma \ref{lm:LB_FS_Hard_Instance_bar_f_base_AS}. 
	
	Here we still use the hard instance constructed in \eqref{eq:LB_hard_instance_FS}, note that $ \nabla f_i(x,y)=\eta\nabla \bar{f}_i\autopar{\frac{x}{\eta},\frac{y}{\eta}} $ is a scaled version of $\bar{f}_i$, which is $L$-averaged smooth by Lemma \ref{lm:LB_FS_Hard_Instance_bar_f_base_AS}, so by Lemma \ref{lm:scaling} we have $ \{f_i\}_i $ is also $ L $-average smooth. The for the strong concavity, note that $\bar{f}$ is $ \mu $-strongly concave on $ y $, so as the scaled version, $f$ is also $ \mu $-strongly concave on $ y $. 
	
	Then for the primal function of $f$, let $ \Phi(x)\triangleq\max_y f(x,y) $, by Lemma \ref{lm:primal_hard_instance_FS_AS} and Lemma \ref{lm:Primal_Scaled_FS_AS_Hard_Instance}, we have 
	\begin{equation}
	    \Phi(x)
	    =
	    \eta^2\bar{\Phi}\autopar{\frac{x}{\eta}}
	    =
	    \frac{1}{n}\sum_{i=1}^{n}\eta^2\bar{\Phi}_i\autopar{\frac{x}{\eta}},
	\end{equation}
	where $\bar{\Phi}$ and $\bar{\Phi}_i$ follow the definition in Lemma \ref{lm:primal_hard_instance_FS_AS},
	
	%Then first we use the nonconvergence argument, 
	We first justify the lower bound argument by lower bounding the norm of the gradient. Recall the definition of $\mathcal{I}$ (see \eqref{eq:LB_FS_nonconvergence}), which is the index set such that $ u^{(i)}_d=u^{(i)}_{d+1}=0,\ \forall i\in\mathcal{I} $ while $ u^{(i)}=\mathbf{U}^{(i)}x $. By substituting the parameters in the statement above into \eqref{eq:LB_FS_nonconvergence} and Lemma \ref{LM:NCSC_LB_PHI}, we have that when the size of the set $\mathcal{I}$, i.e., $ |\mathcal{I}|>n/2 $ (note that scaling does not affect the activation status),
	\begin{equation}
		\begin{split}
		    \|\nabla \Phi(x)\|^2
    		=\ &
    		\autonorm{\eta\nabla \bar{\Phi}\autopar{\frac{x}{\eta}}}^2
    		=
    		\eta^2\autonorm{\nabla \bar{\Phi}\autopar{\frac{x}{\eta}}}^2\\
    		\geq\ &
    		\frac{51200n\mu^2}{L^4}\alpha^{-\frac{3}{2}}\epsilon^2\cdot \frac{\lambda_1^4}{128n\lambda_2^2}\alpha^{\frac{3}{2}}\\
    		=\ &
    		\frac{51200n\mu^2}{L^4}\alpha^{-\frac{3}{2}}\epsilon^2\cdot \frac{L^4}{51200n\mu^2}\alpha^{\frac{3}{2}}
    		=
    		\epsilon^2.
		\end{split}
	\end{equation}
	
	Next, we upper bound the starting optimality gap. By substitution of parameter settings and the initial gap of $\bar{\Phi}$ in \eqref{eq:hard_instance_FS_initial_gap}, also recall the setting of $\epsilon$, we have
	\begin{equation}
		\begin{split}
			\Phi(0)-\Phi^*
			=\ &
			\eta^2\left(\bar{\Phi}(0)-\inf_{x\in\mathbb{R}^{d+1}}\bar{\Phi}(x)\right)
			=
			\frac{51200n\mu^2}{L^4}\alpha^{-\frac{3}{2}}\epsilon^2\cdot\frac{nL^2}{40n\mu}\left(\frac{\sqrt{\alpha}}{2}+10\alpha d\right)\\
			=\ &
			\frac{1280n\mu}{L^2}\left(\frac{1}{2\alpha}+\frac{10d}{\sqrt{\alpha}}\right)\epsilon^2
			=
			\frac{640n\mu\epsilon^2}{\alpha L^2}+\frac{12800n\mu d\epsilon^2}{L^2\sqrt{\alpha}}\\
			\leq\ &
			\frac{640n\mu}{\alpha L^2}\cdot\frac{\alpha L^2\Delta}{1280n\mu}+\frac{12800n\mu\epsilon^2}{L^2\sqrt{\alpha}}\cdot\frac{\sqrt{\alpha}L^2\Delta}{25600n\mu}\epsilon^{-2}\\
			\leq\ &
			\frac{\Delta}{2}+\frac{\Delta}{2}
			= \Delta,
		\end{split}
	\end{equation}
	so we conclude that $ f\in\mathcal{F}_{\mathrm{NCSC}}^{L,\mu,\Delta} $, i.e. the function class requirement is satisfied.
	
	To show the lower bound, by previous analysis and the choice of \eqref{eq:LB_FS_Orthogonal_Matrix}, the activation process for each component will also mimic the "alternating zero-chain" mechanism (see Lemma \ref{LM:NCSC_LB_F_D}) independently. So we have, by the lower bound argument \eqref{eq:LB_FS_nonconvergence}, it requires to activate at least half of the components through until their $d$-th elements (or at least half of $\{u^{(i)}\}_i$ are not activated through until the $d$-th element, note that each $u^{(i)}$ corresponds to an unique part of $x$ with length $(d+1)$) for the primal stationarity convergence of the objective function, which takes (note that $2\lfloor x\rfloor-1\geq x$ when $x\geq 3$)
	\begin{equation}
		T=\frac{n}{2}(2d-1)
		\geq
		\frac{n}{2}\cdot\frac{\sqrt{\alpha}L^2\Delta}{25600n\mu}\epsilon^{-2}
		=
		\Omega\autopar{\sqrt{\alpha}\Delta L\kappa\epsilon^{-2}}
		=
		\Omega\autopar{\sqrt{n \kappa}\Delta L\epsilon^{-2}}
	\end{equation}
	IFO oracle queries. So we found that for any fixed index sequence $\{i_t\}_{t=1}^T$, the output $ z^{T+1} $ from a randomized algorithm\footnote{Note that randomization does not affect the lower bound, as long as the algorithm satisfies the linear-span assumption.} must not be an approximate stationary point, which verifies the $\Omega\autopar{n\vee \sqrt{n\kappa}\Delta L\epsilon^{-2}}$ or $\Omega\autopar{n+\sqrt{n\kappa}\Delta L\epsilon^{-2}}$ lower bound by combining the two cases discussed above together. We conclude the proof by applying Yao's minimax theorem \citep{yao1977probabilistic}, the lower bound will also hold for a randomized index sequence incurred by IFOs.
\end{proof}

\section{Proof of NC-SC Catalyst}

\subsection{Outer-loop Complexity}
In this section, we first introduce a few useful definitions. The Moreau envelop of a function $F$ with a positive parameter $\lambda>0$ is: 
$$
    F_{\lambda}(x) = \min_{z\in\mathbb{R}^{d_1}}  F(z) + \frac{1}{2\lambda}\Vert z - x \Vert^2.$$
We also define the proximal point of $x$:
$$\operatorname{prox}_{\lambda F}(x)=\argmin_{z\in\mathbb{R}^{d_1}}\left\{F(z)+\frac{1}{2 \lambda}\|z-x\|^{2}\right\}.$$
When $F$ is differentiable and $\ell$-weakly convex, for $\lambda \in(0,1 / \ell)$ we have
\begin{equation}
    \nabla F(\prox_{\lambda F} (x)) = \nabla  F_\lambda(x) = \lambda^{-1}(x - \prox_{\lambda F}(x)).
\end{equation}
Thus a small gradient $\|\nabla F_\lambda (x)\|$ implies that $x$ is near a point $\prox_{\lambda F}(x)$ that is nearly stationary for $F$. Therefore $\|\nabla F_\lambda(x)\|$ is also commonly used as a measure of stationarity. We refer readers to \citep{drusvyatskiy2019efficiency} for more discussion on Moreau envelop.

In this subsection, we use $(x^t, y^t)$ as a shorthand for $(x^t_0, y^t_0)$. We will denote $(\hat{x}^t, \hat{y}^t)$ as the optimal solution to the auxiliary problem (\ref{auxiliary prob ncc}) at $t$-th iteration: $\min_{x\in \mathbb{R}^{d_1}}\max_{y\in \mathbb{R}^{d_2}} \left[\hat{f}_{t}(x,y)\triangleq  f(x,y) + L\Vert x - x^t\Vert^2 \right]$. It is easy to observe that $\hat{x}^t = \prox_{\Phi/2L}(x^t)$. Define $\hat{\Phi}_t(x) = \max_y f(x,y) + L\|x-x^t\|^2 $. In the following theorem, we show the convergence of the Moreau envelop $\|\nabla\Phi_{1/2L}(x)\|^2$ when we replace the inexactness measure (\ref{ncc criteion}) by another inexactness measure $\gap_{\hat{f}_{t}}(x^{t+1}, y^{t+1})\leq \beta_t(\|x^t-\hat{x}^t\|+\|y^t-\hat{y}^t\|^2)$. Later we will show this inexactness measure can be implied by (\ref{ncc criteion}) with our choice of $\beta_t$ and $\alpha_t$.

\begin{theorem} \label{ncc moreau complexity}
Suppose function $f$ is NC-SC with strong convexity parameter $\mu$ and L-Lipschitz smooth. If we replace the stopping criterion (\ref{ncc criteion}) by $\gap_{\hat{f}_{t}}(x^{t+1}, y^{t+1})\leq \beta_t(\|x^t-\hat{x}^t\|^2+\|y^t-\hat{y}^t\|^2)$ with $\beta_t = \frac{\mu^4}{28L^3}$ for $t>0$ and $\beta_0 = \frac{\mu^4}{32L^4\max\{1,L\}}$, then iterates from Algorithm \ref{catalyst ncc 1} satisfy
\begin{align} \label{ncsc moreau sum}
    \sum_{t=0}^{T-1} \|\nabla\Phi_{1/2L}(x^t)\|^2 \leq 
    \frac{87L}{5}\Delta_0 + \frac{7L}{5}D_y^0,
\end{align}
where $D_y^0 = \|y^0 - y^*(x^0)\|^2$ and $\Delta_0 = \Phi(x^0) -\inf_{x}\Phi(x) $. 
\end{theorem}

\begin{proof}
Define $b_{t+1} = \gap_{\hat{f_t}}(x^{t+1}, y^{t+1})$. %From the stopping criterion of solving subproblem in Algorithm \ref{catalyst ncc} and Lemma \ref{criterion relation}, we have 
%\begin{equation}
%    b_{t+1} \leq 
%\end{equation}
By Lemma 4.3 in \citep{drusvyatskiy2019efficiency}, 
\begin{align} \nonumber
    \Vert \nabla \Phi_{1/2L}(x^t)\Vert^2 =  4L^2 \Vert x^t - \prox_{\Phi/2L}(x^t)\Vert^2 \leq & 8L [\hat{\Phi}_t(x^t) - \hat{\Phi}_t(\prox_{\Phi/2L}(x^t))] \\ \nonumber
    \leq & 8L [\hat{\Phi}_t(x^t) -\hat{\Phi}_t(x^{t+1}) + b_{t+1}] \\ \nonumber
    = & 8L\big\{\Phi(x^t) - \left[\Phi(x^{t+1})+L\Vert x^{t+1}-x^t\Vert^2 \right]+b_{t+1}\big\}\\  \label{bound x to prox 3}
    \leq & 8L [\Phi(x^t) - \Phi(x^{t+1}) +b_{t+1}],
\end{align}
where in the first inequality we use $L$-strongly convexity of $\hat{\Phi}_t$. Then, for $t\geq 1$
\begin{align*}
    \|y^t - \hat{y}^t\|^2 \leq& 2\|y^t - \hat{y}^{t-1}\|^2 + 2\|y^*(\hat{x}^{t-1})-y^*(\hat{x}^t)\|^2 \\
    \leq & 2\|y^t - \hat{y}^{t-1}\|^2 + 2\left(\frac{L}{\mu}\right)^2\|\hat{x}^t-\hat{x}^{t-1}\|^2\\
    \leq &  2\|y^t - \hat{y}^{t-1}\|^2 + 4\left(\frac{L}{\mu}\right)^2\|\hat{x}^t-x^t\|^2 + 4\left(\frac{L}{\mu}\right)^2\|x^t-\hat{x}^{t-1}\|^2 \\
    \leq &  \frac{8L}{\mu^2}b_{t} + 4\left(\frac{L}{\mu}\right)^2\|\hat{x}^t-x^t\|^2,
\end{align*}
where we use Lemma \ref{lin's lemma} in the second inequality, and $(L,\mu)$-SC-SC of $\Tilde{f}_{t-1}(x,y)$ and Lemma \ref{criterion relation} in the last inequality. Therefore, 
\begin{equation} \label{ncc dist dif}
    \|x^t - \hat{x}^t\|^2 + \|y^t - \hat{y}^t\|^2 \leq \frac{8L}{\mu^2}b_{t} + \left(\frac{4L^2}{\mu^2} + 1\right)\|\hat{x}^t-x^t\|^2.
\end{equation}
%Since the auxiliary problem is $(L,\mu)$-SC-SC and $3L$-smooth, by Lemma \ref{criterion relation},
By our stopping criterion and $ \Vert \nabla \Phi_{1/2L}(x^t)\Vert^2 =  4L^2 \Vert x^t - \hat{x}^t\Vert^2$, for $t\geq 1$
\begin{equation*}
    b_{t+1} \leq \beta_t \left[ \|x_{t} - \hat{x}^t\|^2 + \|y_{t} - \hat{y}^t\|^2 \right] \leq \frac{8L \beta_t}{\mu^2}b_{t} + \beta_t\left(\frac{1}{\mu^2} + \frac{1}{4L^2}\right)\Vert \nabla \Phi_{1/2L}(x^t)\Vert^2.
\end{equation*}
%where we use stopping criterion of the auxiliary problem (\ref{auxiliary prob ncc}) in the second inequality. Combining with (\ref{ncc dist dif}) and noting that $\|\hat{x}^t-x^t\| = \frac{1}{4L^2}\|\nabla \Phi_{1/2L}(x^t)\|$, for $t\geq 1$
%\begin{equation}
%    b_{t+1}  \leq \frac{72L^3 \alpha_t}{\mu^3}b_{t} + \frac{9L^2\alpha_t}{\mu}\left(\frac{1}{\mu^2} + \frac{1}{4L^2}\right)\Vert \nabla \Phi_{1/2L}(x^t)\Vert^2.
%\end{equation}
Define $\theta = \frac{2}{7}$ and $w = \frac{5\mu^2}{112L^3}$. It is easy to verify that as $\beta_t = \frac{\mu^4}{28L^3}$, then $\frac{8L \beta_t}{\mu^2}\leq \theta$ and $\beta_t\left(\frac{1}{\mu^2} + \frac{1}{4L^2}\right)\leq w$. We conclude the following recursive bound
%So we can write 
\begin{equation}\label{eqn:rec_outer_loop}
b_{t+1} \leq \theta b_t + w\Vert \nabla \Phi_{1/2L}(x^t)\Vert^2.    
\end{equation} 
For $t=0$,
\begin{equation}  \label{ncc y0 bound}
    \|y^0-\hat{y}^0\|^2\leq 2\|y^0-y^*(x^0)\|^2 + 2\|\hat{y}^0-y^*(x^0)\|^2\leq 2\|y^0-y^*(x^0)\|^2  + 2\left(\frac{L}{\mu}\right)^2\|x^0-\hat{x}^0\|^2.
\end{equation}
Because $\Phi(x)+L\|x-x^0\|^2$ is $L$-strongly convex, we have
\begin{equation*}
    \left(\Phi(\hat{x}^0) +L\|\hat{x}^0-x^0\|^2 \right) + \frac{L}{2}\|\hat{x}^0-x^0\|^2\leq \Phi(x^0) = \Phi^* + (\Phi(x^0)-\Phi^*) \leq \Phi(\hat{x}^0) + (\Phi(x^0)-\Phi^*).
\end{equation*}
This implies $\|\hat{x}^0-x^0\|^2\leq \frac{L}{2}(\Phi(x^0)-\Phi^*)$. Then combining with (\ref{ncc y0 bound})
\begin{equation*}
    \|y^0-\hat{y}^0\|^2 + \|x^0-\hat{x}^0\|^2 \leq \left(\frac{L^3}{\mu^2}+\frac{L}{2}\right)(\Phi(x^0)-\Phi^*) + 2\|y^0-y^*(x^0)\|^2.
\end{equation*}
Hence, by the stopping criterion,
\begin{equation*}
    b_1  \leq \beta_0\left(\frac{L^3}{\mu^2}+\frac{L}{2}\right)(\Phi(x^0)-\Phi^*) + 2\beta_0\|y^0-y^*(x^0)\|^2.
\end{equation*}
Define $\theta_0 = \frac{\mu^2}{16L^2}$ . With $\beta_0 = \frac{\mu^4}{32L^4\max\{1,L\}}$, $\beta_0\left(\frac{L^3}{\mu^2}+\frac{L}{2}\right)\leq \theta_0$ and $2\beta_0\leq \theta_0$. So we can write 
$$b_1 \leq \theta_0(\Phi(x^0)-\Phi^*) + \theta_0\|y^0-y^*(x^0)\|^2.$$
Unravelling \eqref{eqn:rec_outer_loop},
%Iterating updates of $\{b_t\}_t$, 
we have for $t\geq1$
\begin{align} 
    b_{t+1} \leq \theta^tb_1 + w\sum_{k=1}^t\theta^{t-k}\|\nabla\Phi_{1/2L}(x_k)\|^2\leq  \theta^{t}\theta_0(\Phi(x^0)-\Phi^*) + \theta^{t}\theta_0\|y^0-y^*(x^0)\|^2 +w\sum_{k=1}^t\theta^{t-k}\|\nabla\Phi_{1/2L}(x_k)\|^2.
\end{align}
Summing from $t=0$ to $T-1$,
\begin{align} \nonumber
    \sum_{t=0}^{T-1}b_{t+1} &= \sum_{t=1}^{T-1}b_t + b_1\\ \nonumber
    &\leq \theta_0\sum_{t=0}^{T-1}\theta^t[\Phi(x^0)-\Phi^*] +\theta_0\sum_{t=0}^{T-1}\theta^t\|y^0-y^*(x^0)\|^2 + w\sum_{t=1}^{T-1}\sum_{k=1}^t\theta^{t-k}\|\nabla\Phi_{1/2L}(x_k)\|^2 \\ \label{rewrite b_t}
    &\leq \theta_0\sum_{t=0}^{T-1}\theta^t[\Phi(x^0)-\Phi^*] +\theta_0\sum_{t=0}^{T-1}\theta^t\|y^0-y^*(x^0)\|^2 +w\sum_{t=1}^{T-1}\frac{1}{1-\theta}\|\nabla\Phi_{1/2L}(x^t)\|^2,
\end{align}
where we use $\sum_{t=1}^{T-1}\sum_{k=1}^t\theta^{t-k}\|\nabla\Phi_{1/2L}(x_k)\|^2 = \sum_{k=1}^{T-1}\sum_{t=k}^T\theta^{t-k}\|\nabla\Phi_{1/2L}(x_k)\|^2 \leq \sum_{k=1}^{T-1}\frac{1}{1-\theta}\|\nabla\Phi_{1/2L}(x_k)\|^2$.
Now, by telescoping (\ref{bound x to prox 3}),
\begin{equation*}
    \frac{1}{8L}\sum_{t=0}^{T-1}\|\nabla\Phi_{1/2L}(x^t)\|^2 \leq \Phi(x^0)-\Phi^* + \sum_{t=0}^{T-1}b_{t+1}.
\end{equation*}
Plugging (\ref{rewrite b_t}) in,
\begin{equation}
    \frac{1}{8L}\sum_{t=0}^{T-1}\|\nabla\Phi_{1/2L}(x^t)\|^2 -w\sum_{t=1}^{T-1}\frac{1}{1-\theta}\|\nabla\Phi_{1/2L}(x^t)\|^2 \leq \left(1+\frac{\theta_0}{1-\theta}\right)[\Phi(x^0)-\Phi^*]+\frac{\theta_0}{1-\theta}\|y^0-y^*(x^0)\|^2.
\end{equation}
Plugging in $w\leq \frac{5}{112L}$, $\frac{1}{1-\theta}=\frac{7}{5}$ and $\theta_0\leq \frac{1}{16}$
\begin{align*}
    \frac{1}{16L}\sum_{t=0}^{T-1}\|\nabla\Phi_{1/2L}(x^t)\|^2 \leq \frac{87}{80}[\Phi(x^0)-\Phi^*]+\frac{7}{80}\|y^0-y^*(x^0)\|^2.
\end{align*}
\end{proof}

\noindent\textbf{Proof of Theorem \ref{THM CATALYST NCSC}}

\begin{proof}
We first show that criterion (\ref{ncc criteion}) implies the criterion in Theorem \ref{ncc moreau complexity}. By Lemma \ref{criterion relation}, as $\hat{f}_t$ is $(L, \mu)$-SC-SC and $3L$-smooth,
\begin{equation*}
    2\mu \gap_{\hat{f}_{t}}(x^{t+1}, y^{t+1})\leq \|\nabla \hat{f}_t(x^{t+1}, y^{t+1})\|^2 \leq \alpha_t\|\nabla \hat{f}_t(x^t, y^t)\|^2 \leq 36L^2\alpha_t(\|x^t-\hat{x}^t\|^2+\|y^t-\hat{y}^t\|^2),
\end{equation*}
therefore, 
\begin{equation*}
    \gap_{\hat{f}_{t}}(x^{t+1}, y^{t+1}) \leq \frac{18L^2\alpha_t}{\mu}(\|x^t-\hat{x}^t\|^2+\|y^t-\hat{y}^t\|^2),
\end{equation*}
which implies $\gap_{\hat{f}_{t}}(x^{t+1}, y^{t+1})\leq \beta_t(\|x^t-\hat{x}^t\|^2+\|y^t-\hat{y}^t\|^2)$ by our choice of $\{\beta_t\}_t$ and $\{\alpha_t\}_t$. 

We still use $b_{t+1} = \gap_{\hat{f_t}}(x^{t+1}, y^{t+1})$ as in the proof of Theorem (\ref{ncc moreau complexity}). First, note that 
\begin{align} \nonumber
    \|\nabla\Phi(x^{t+1})\|^2 &\leq 2\|\nabla\Phi(x^{t+1})-\nabla \Phi(\hat{x}^t)  \|^2 + 2\|\nabla\Phi(\hat{x}^t)\|^2\\ \nonumber
    &\leq 2\left(\frac{2L^2}{\mu} \right)\|x^{t+1}-\hat{x}^t\|^2 + 2\|\nabla \Phi_{1/2L}(x^t)\|^2 \\ 
    &\leq \frac{16L^3}{\mu^2}b_{t+1} + 2\|\nabla \Phi_{1/2L}(x^t)\|^2.
\end{align}
where in the second inequality we use Lemma \ref{lin's lemma} and Lemma 4.3 in \citep{drusvyatskiy2019efficiency}. Summing from $t=0$ to $T-1$, we have
\begin{equation} \label{moreau to primal convergence}
    \sum_{t=0}^{T-1} \|\nabla\Phi(x^{t+1})\|^2 \leq \frac{16L^3}{\mu^2}\sum_{t=0}^{T-1} b_{t+1} + 2\sum_{t=0}^{T-1} \|\nabla \Phi_{1/2L}(x^t)\|^2.
\end{equation}
Applying (\ref{rewrite b_t}), we have
\begin{align*}
    \frac{16L^3}{\mu^2}\sum_{t=0}^{T-1} b_{t+1} \leq \frac{16L^3\theta_0}{\mu^2}\sum_{t=0}^{T-1}\theta^t[\Phi(x^0)-\Phi^*] +\frac{16L^3\theta_0}{\mu^2}\sum_{t=0}^{T-1}\theta^t\|y^0-y^*(x^0)\|^2 +\frac{16L^3w}{\mu^2}\sum_{t=1}^{T-1}\frac{1}{1-\theta}\|\nabla\Phi_{1/2L}(x^t)\|^2.
\end{align*}
Plugging in $\theta_0 = \frac{\mu^2}{16L^2}$, $\theta = \frac{2}{7}$ and $w = \frac{5\mu^2}{112L^3}$, 
\begin{equation*}
    \frac{16L^3}{\mu^2}\sum_{t=0}^{T-1} b_{t+1} \leq \frac{7L}{5}[\Phi(x^0)-\Phi^*] + \frac{7L}{5}\|y^0-y^*(x^0)\|^2 + \sum_{t=1}^{T-1}\|\nabla\Phi_{1/2L}(x^t)\|^2.
\end{equation*}
Plugging back into (\ref{moreau to primal convergence}),
\begin{equation*}
    \sum_{t=0}^{T-1} \|\nabla\Phi(x^{t+1})\|^2 \leq \frac{7L}{5}[\Phi(x^0)-\Phi^*] + \frac{7L}{5}\|y^0-y^*(x^0)\|^2 + 3\sum_{t=0}^{T-1}\|\nabla\Phi_{1/2L}(x^t)\|^2.
\end{equation*}
Applying Theorem \ref{ncc moreau complexity},
\begin{equation*}
    \frac{1}{T}\sum_{t=1}^T \|\nabla\Phi(x^{t+1})\|^2 \leq \frac{268L}{5T}[\Phi(x^0)-\Phi^*] + \frac{28L}{5T}\|y^0-y^*(x^0)\|^2.
\end{equation*}
\end{proof}

\subsection{Complexity of solving auxiliary problem (\ref{auxiliary prob ncc}) and proof of Theorem \ref{thm catalyst scsc}}

In this layer, we apply an inexact proximal point algorithm to solve the $(L,\mu)$-SC-SC and $3L$-smooth auxiliary problem: $\min_x\max_y \hat{f}_t(x,y)$. Throughout this subsection, we suppress the outer-loop index $t$ without confusion, i.e. we use $\hat{f}$ instead of $\hat{f}_t$ and $\Tilde{f}_{k} = \hat{f}(x,y) - \frac{\tau}{2}\Vert y - z_k\Vert^2$ instead of $\Tilde{f}_{t,k}$. Accordingly, we also omit the superscript in $(x^t_k, y^t_k)$ and $\epsilon^t_k$.

Before we prove Theorem \ref{thm catalyst scsc}, we present a lemma from \citep{lin2018catalyst}. The inner loop to solve (\ref{auxiliary prob ncc}) can be considered as applying Catalyst for strongly-convex minimization in \citep{lin2018catalyst} to the function $-\hat{\Psi}(y) = -\min_{x\in\mathbb{R}^{d}}\hat{f}(x, y)$. The following lemma captures the convergence of Catalyst framework in minimization, which we present in Algorithm \ref{catalyst min}. 

\begin{algorithm}[ht] 
    \caption{ Catalyst for Strongly-Convex Minimization}
    \setstretch{1.25}
    \begin{algorithmic}[1] \label{catalyst min}
        \REQUIRE function $h$, initial point $x_0$, strong-convexity constant $\mu$, parameter $\tau>0$
        \STATE Initialization: $q = \frac{\mu}{\mu +\tau},  z_1 = x_0$,  $\alpha_1 = \sqrt{q}$.
        \FORALL{$k = 1,2,..., K$}
            \STATE Find an inexact solution $x_k$ to the following problem with algorithm $\mathcal{M}$ 
            \begin{equation*} 
                \min_{x\in\mathbb{R}^d}\Tilde{h}_k(x) \triangleq \left[h(x) + \frac{\tau}{2}\Vert x - z_k\Vert^2 \right]
            \end{equation*}
            % \vspace{-2mm}
            such that 
            \begin{align} 
                \Tilde{h}_k(x_k) -  \min_{x\in\mathbb{R}^d}\Tilde{h}_k(x)\leq \epsilon_k.
            \end{align}
            % \vspace{-3mm}
            \STATE Choose $\alpha_{k+1}\in [0,1]$ such that 
           $
                \alpha_{k+1}^2 = (1-\alpha_{k+1})\alpha_k^2 + q\alpha_{k+1}.
           $
           \STATE $z_{k+1} = x_k + \beta_k(x_k-x_{k-1})$ where $\beta_k = \frac{\alpha_k(1-\alpha_k)}{\alpha_k^2 + \alpha_{k+1}}.$
        \ENDFOR
        \ENSURE $x_K$.
    \end{algorithmic}
\end{algorithm}

\begin{lemma} [\citep{lin2018catalyst}] \label{lemma catalyst min}
Consider the problem $\min_{x\in\mathbb{R}^d} h(x)$. Assume function $h$ is $\mu$-strongly convex. Define $A_k = \prod_{i=1}^k (1-\alpha_i)$, $\eta_k = \frac{\alpha_k-q}{1-q}$ and a sequence $\{v_t \}_t$ with $v_0 = x_0$ and $v_k = x_{k-1} + \frac{1}{\alpha_k}(x_k- x_{k-1})$ for $k>1$. Consider the potential function: $S_k = h(x_k) - h(x^*) + \frac{\eta_{k+1}\alpha_{k+1}\tau}{2(1-\alpha_{k+1})}\|x^*-v_k\|^2$, where $x^*$ is the optimal solution. After running Algorithm \ref{catalyst min} for $K$ iterations, we have
\begin{equation} 
    \frac{1}{A_K}S_K \leq  \left(\sqrt{S_0}+ 2\sum_{t=1}^K \sqrt{\frac{\epsilon_k}{A_k}}\right)^2.
\end{equation}

\end{lemma}

\bigskip
Before we step into the proof of Theorem \ref{thm catalyst scsc}, we introduce several notations. We denote the dual function of $\hat{f}$ by $\hat{\Psi}(y) = \min_x \hat{f}(x, y)$. We denote the dual function of $\Tilde{f}_k(x,y)$ by  $\Tilde{\Psi}_k(y) =\min_x\Tilde{f}_k(x,y)= \min_{x}\hat{f}(x, y) - \frac{\tau}{2}\Vert y-z_k\Vert^2 = \hat{\Psi}(y) - \frac{\tau}{2}\Vert y-z_k\Vert^2$.  Let $y_k^* = \argmax_y \Tilde{\Psi}_k(y)$. We also define $(x^*, y^*)$ as the optimal solution to $\min_x\max_y \hat{f}(x,y)$

\bigskip

{\noindent \textbf{Proof of Theorem \ref{thm catalyst scsc}}}\\
\begin{proof}
When the criterion $\|\nabla \Tilde{f}_{k}(x^{k}, y^{k})\|^2 \leq \epsilon_k$ is satisfied, by Lemma \ref{criterion relation}, $$\gap_{\Tilde{f}_k}(x_k, y_k) \leq \frac{1}{2\mu}\|\nabla \Tilde{f}_{k}(x^{k}, y^{k})\|^2 \leq \frac{1}{2\mu}\epsilon_k =\frac{\sqrt{2}}{4}(1-\rho)^k\gap_{\hat{f}}(x_0, y_0) =\hat{\epsilon}_k,$$
where we define $\hat{\epsilon}_k = \frac{\sqrt{2}}{4}(1-\rho)^k\gap_{\hat{f}}(x_0, y_0) $.

The auxiliary problem (\ref{subprob}) can be considered as $\max_{y} \hat{\Psi}(y)$. We see $\gap_{\Tilde{f}_k}(x_k, y_k) \leq \hat{\epsilon}_k$ implies $\max_{y} \Tilde{\Psi}_k(y) - \Tilde{\Psi}_k(y_k) \leq \hat{\epsilon}_k$. By choosing $\alpha_1 = \sqrt{q}$ in Algorithm \ref{catalyst min}, it is easy to check that $\alpha_k = \sqrt{q}$ and $\beta_k = \frac{\sqrt{q}-q}{\sqrt{q}+q}$, for all $k$.  So this inner loop can be considered as applying Algorithm \ref{catalyst min} to $-\Tilde{\Psi}(y)$ and Lemma \ref{lemma catalyst min} can guarantee the convergence of the dual function. Define $S_k = \hat{\Psi}(y^*) - \hat{\Psi}(y_k) + \frac{\eta_{t+1}\alpha_{t+1}\tau}{2(1-\alpha_{t+1})}\|y^*-v_k\|^2$ with $\eta_k = \frac{\alpha_k-q}{1-q}$. Lemma \ref{lemma catalyst min} gives rise to 
\begin{equation}  \label{dec 5}
    \frac{1}{A_K}S_K \leq \left(\sqrt{S_0}+ 2\sum_{k=1}^K \sqrt{\frac{\hat{\epsilon}_k}{A_k}}\right)^2.
\end{equation}
Note that $A_k = \prod_{i=1}^k (1-\alpha_i) = (1-\sqrt{q})^k$ and 
$$\frac{\eta_k\alpha_k\tau}{2(1-\alpha_k)} = \frac{\sqrt{q}-q}{1-q}\frac{\sqrt{q}\tau}{2(1-\sqrt{q})} = \frac{\sqrt{q}-q}{\tau/(\mu+\tau)}\frac{\sqrt{q}\tau}{2(1-\sqrt{q})} = \frac{q(\mu+\tau)}{2} = \frac{\mu}{2}.$$
So $S_0 = \hat{\Psi}(y^*) - \hat{\Psi}(y_0) + \frac{\mu}{2}\|y^* - y_0\|^2 \leq 2(\hat{\Psi}(y^*) - \hat{\Psi}(y_0))$. Then with $\epsilon_k = \frac{\sqrt{2}}{4}(1-\rho)^k\gap_{\hat{f}}(x_0, y_0)$, and we have
\begin{align}
    \text{Right-hand side of } (\ref{dec 5}) \leq &\left( \sqrt{2(\hat{\Psi}(y^*) - \hat{\Psi}(y_0))}  + \sum_{t=1}^T\sqrt{2\left(\frac{1-\rho}{1-\sqrt{q}}\right)^t\gap_{\hat{f}}(x_0, y_0)}\right)^2 \\
    \leq & 2\left(1+ \sum_{k=1}^K\left( \sqrt{\frac{1-\rho}{1-\sqrt{q}}}\right)^k \right)^2 \gap_{\hat{f}}(x_0, y_0) \\
    \leq & 2\left( \frac{\left(\sqrt{\frac{1-\rho}{1-\sqrt{q}}}\right)^{K+1}}{\sqrt{\frac{1-\rho}{1-\sqrt{q}}}-1} \right)^2 \gap_{\hat{f}}(x_0, y_0) \leq 2\left( \frac{\sqrt{\frac{1-\rho}{1-\sqrt{q}}}}{\sqrt{\frac{1-\rho}{1-\sqrt{q}}}-1} \right)^2 \left( \frac{1-\rho}{1-\sqrt{q}} \right)^K \gap_{\hat{f}}(x_0, y_0).
\end{align}
Plugging back into (\ref{dec 5}),
\begin{equation} \label{dec 6}
    S_K \leq  2\left( \frac{1}{\sqrt{1-\rho} - \sqrt{1-\sqrt{q}}} \right)^2 (1-\rho)^{K+1} \gap_{\hat{f}}(x_0, y_0) \leq  \frac{8}{(\sqrt{q}-\rho)^2}(1-\rho)^{K+1} \gap_{\hat{f}}(x_0, y_0),
\end{equation}
where the second inequality is due to $\sqrt{1-x} + \frac{x}{2}$ is decreasing in $[0,1]$. Note that 
\begin{align} \nonumber
    \|x_K-x^*\|^2 \leq & 2\|x_K - x^*(y_K)\|^2 + 2\|x^*(y_K)-x^*(y^*)\|^2 \\ \nonumber
    \leq &\frac{4}{L}[\hat{f}(x_K, y_K) - \hat{f}(x^*(y_K),y_K)] + 18\|y_K-y^*\|^2 \\
    \leq & \frac{4}{L}\hat{\epsilon}_K + 18\|y_K-y^*\|^2.
\end{align}
where in the second inequality we use Lemma \ref{lin's lemma}. Then,
\begin{align}
     \|x_K-x^*\|^2 +  \|y_K-y^*\|^2 \leq 19\|y_K-y^*\|^2 +\frac{4}{L}\hat{\epsilon}_K.
\end{align}
Because $\|y_K-y^*\|^2 \leq \frac{2}{\mu}[\hat{\Psi}(y^*)-\hat{\Psi}(y_K)] \leq \frac{2}{\mu}S_K$,  by plugging in (\ref{dec 6}) and the definition of $\hat{\epsilon}_k$, we get 
\begin{align} \nonumber
    \Vert x_K-x^*\Vert^2 + \Vert y_K- y^*\Vert^2 \leq \left( \frac{306}{\mu(\sqrt{q}-\rho)^2} + \frac{\sqrt{2}}{L} \right)(1-\rho)^{K}\gap_{\hat{f}}(x_0, y_0). 
\end{align}
By Lemma \ref{criterion relation}, we have 
$$
\Vert x_K-x^*\Vert^2 + \Vert y_K- y^*\Vert^2 \geq \frac{1}{36L^2}\|\nabla \hat{f}(x_{K}, y_{K})\|^2 
\quad \text{ and }\quad
\gap_{\hat{f}}(x_0, y_0) \leq \frac{1}{2\mu}\|\nabla \hat{f}(x_{0}, y_{0})\|^2.
$$ 
Then we finish the proof.
\end{proof}

\subsection{Complexity of solving subproblem (\ref{subprob}) and proof of Theorem \ref{thm catalyst inner}}

As in the previous subsection, we suppress the outer-loop index $t$. Define $\hat{\Phi}(x) = \max_y\hat{f}(x,y)$, $\hat{\Psi}(y) = \min_x \hat{f}(x,y)$ and $\hat{\Phi}^* = \min_x \hat{\Phi}(x) = \max_y\hat{\Psi}(y) = \hat{\Psi}^*$. We still define $\Tilde{\Psi}_k(y) =\min_x\Tilde{f}_k(x,y)= \min_{x}\hat{f}(x, y) - \frac{\tau}{2}\Vert y-z_k\Vert^2 = \hat{\Psi}(y) - \frac{\tau}{2}\Vert y-z_k\Vert^2$, and $\Tilde{\Phi}_k(x) = \max_y \Tilde{f}_k(x,y)$. Let $(x^*, y^*)$ be the optimal solution to $\min_x\max_y \hat{f}(x,y)$ and $(x_k^*, y_k^*)$  the optimal solution to $\min_x\max_y \Tilde{f}_k(x,y)$. Also, in this subsection, we denote $x^*(y) = \argmin_x \hat{f}(x, y)$ and $y^*(x) = \argmax_y \hat{f}(x,y)$. Recall that we defined a potential function $S_k = \hat{\Psi}(y^*) - \hat{\Psi}(y_k) + \frac{\mu}{2}\|y^*-v_k\|^2$ in the proof of Theorem \ref{thm catalyst scsc}.

The following lemma shows that the initial point we choose to solve (\ref{subprob}) for $\mathcal{M}$ at iteration $k$ is not far from the optimal solution of (\ref{subprob}) if the stopping criterion is satisfied for every iterations before $k$.

\begin{lemma} [Initial distance of the warm-start] \label{warm start scsc}
Under the same assumptions as Theorem \ref{thm catalyst scsc}, with accuracy $\epsilon_k$ specified in Theorem \ref{thm catalyst scsc}, we assume that for $\forall i<k$, $\|\nabla \Tilde{f}_{i}(x_{i}, y_{i})\|^2 \leq \epsilon_i$.
%we solve the subproblem (\ref{subprob}) till $\epsilon^{(t)}$ specified in Theorem \ref{thm catalyst scsc} for each $t$. As we set the 
At iteration $k$, solving the subproblem \eqref{subprob} from initial point $(x_{k-1}, y_{k-1})$, we have 
\begin{equation*}
    \|x_{k-1} - x_k^*\|^2 +  \|y_{k-1} - y_k^*\|^2 \leq C_k \epsilon_k,
\end{equation*}
where $C_1 = \left[ \frac{72\sqrt{2}}{\mu^2} + \frac{74\sqrt{2}}{(2\tau+\mu)\mu}\right] \frac{1}{1-\rho}$, $C_t = \frac{2}{\mu\min\{L, \mu+\tau \}}\frac{1}{1-\rho} + \frac{288\sqrt{2}\tau^2\max\{40L^2, 9\tau^2 + 4L^2 \}}{(\mu+\tau)^2L^2\mu^2(\sqrt{q}-\rho)^2)} \frac{1}{(1-\rho)^2}   $ for $t>1$.
\end{lemma}

\begin{proof} We separate the proof into two cases: $k=1$ and $k>1$. \\
\textbf{Case $k=1$: } 
Note that $z_1 = y_0$, and therefore the subproblem at the first iteration is
\begin{equation} 
\min_{x}\max_{y}\left[  \Tilde{f}_1(x, y)  = \hat{f}(x,y) - \frac{\tau}{2}\|y-y_0\|^2\right].
\end{equation}
Since $x_1^* = \argmin_x \Tilde{f}_1(x, y_1^*) = \argmin_x \hat{f}(x, y_1^*)$ and $x^* = \argmin_x \hat{f}(x, y^*)$, by Lemma \ref{lin's lemma} we have $\|x^* - x_1^*\| \leq 3\|y^* - y_1^*\|$. Furthermore, 
\begin{align}\nonumber
     \|x_0 - x_1^*\|^2+\|y_0-y_1^*\|^2 \leq & 2\|x_0-x^*\|^2 + 2\|x^* - x_1^*\|^2 + \|y_0-y_1^*\|^2  \\   \nonumber
    \leq & 2\|x_0-x^*\|^2 + 18\|y^* - y_1^*\|^2 + \|y_0-y_1^*\|^2 \\ \nonumber
    \leq & 2\|x_0-x^*\|^2 + 36\|y_0 - y^*\|^2 + 37\|y_0 - y_1^*\|^2 \\ \label{initial dist iter 1}
    \leq & \frac{72}{\mu}\gap_{\hat{f}}(x_0, y_0) + 37\|y_0 - y_1^*\|^2,
\end{align}
where in the last inequality we use Lemma \ref{criterion relation}. It remains to bound $\|y_0 - y_1^*\|$. Since $\hat{\Psi}(y) - \frac{\tau}{2}\|y - y_0\|^2$ is $(\mu+\tau)$ strongly-concave w.r.t.~$y$, we have
\begin{equation}
    \left( \hat{\Psi}(y_1^*) - \frac{\tau}{2}\|y_1^* - y_0\|^2 \right) - \frac{\tau+\mu}{2}\|y_1^* - y_0\|^2 \geq \hat{\Psi}(y_0) = \hat{\Psi}^* - [\hat{\Psi}^* - \hat{\Psi}(y_0)] \geq \hat{\Psi}(y_1^*)  - [\hat{\Psi}^* - \hat{\Psi}(y_0)],
\end{equation}
It further implies
\begin{equation}
    \left(\tau + \frac{\mu}{2}\right) \|y_1^* - y_0\|^2 \leq \hat{\Psi}^* - \hat{\Psi}(y_0) \leq \gap_{\hat{f}}(x_0, y_0).
\end{equation}
Plugging back into (\ref{initial dist iter 1}), we have
\begin{align*}
    \|x_0 - x_1^*\|^2+\|y_0-y_1^*\|^2 \leq & \left[ \frac{72}{\mu} + \frac{74}{2\tau+\mu}\right]\gap_{\hat{f}}(x_0, y_0) \\
    \leq &  \left[ \frac{72\sqrt{2}}{\mu^2} + \frac{74\sqrt{2}}{(2\tau+\mu)\mu}\right] \frac{1}{1-\rho}\epsilon_1.
\end{align*}

{\noindent\textbf{Case $k>1$:}} From the proof of Theorem \ref{thm catalyst scsc}, we see that $\|\nabla \Tilde{f}_{i}(x^t_{i}, y_{i})\|^2 \leq \epsilon_i$ implies $\gap_{\Tilde{f}_i}(x_i, y_i) \leq \hat{\epsilon}_i$ where $\hat{\epsilon}_i =\frac{\sqrt{2}}{4}(1-\rho)^i\gap_{\hat{f}}(x_0, y_0)$. 
Note that $\Tilde{f}_k$ is $(L, \mu+\tau)$-SC-SC and $(L+\max\{2L,\tau\})$-smooth. Then
\begin{equation} \label{initial dist x}
    \begin{split}
        \Vert x_{k-1} - x_k^*\Vert^2 
        \leq\ &
        2\Vert x_{k-1} - x^*(y_{k-1}^*)\Vert^2 + 2\Vert x^*(y_{k-1}^*)-x^*(y_k^*)\Vert^2 \\
        \leq\ & 
        2\Vert x_{k-1} - x_{k-1}^*\Vert^2 + 2\left(\frac{L+\max\{2L,\tau\}}{L}\right)^2 \Vert y_k^* - y_{k-1}^*\Vert^2.
    \end{split}
\end{equation}
Furthermore,
\begin{align} \nonumber
   & \|x_{k-1} - x_k^*\|^2+ \|y_{k-1}- y_k^*\|^2 \leq \|x_{k-1} - x_k^*\|^2+ 2\|y_{k-1}- y_{k-1}^*\|^2 +2\|y_{k-1}^*- y_k^*\|^2\\ \nonumber
   \leq & 2\Vert x_{k-1} - x_{k-1}^*\Vert^2 + 2\|y_{k-1}- y_{k-1}^*\|^2 + 2\left[\left(\frac{L+\max\{2L,\tau\}}{L}\right)^2+1\right] \Vert y_k^* - y_{k-1}^*\Vert^2 \\\label{initial dist bound}
   \leq & \frac{4\hat{\epsilon}_{k-1}}{\min\{L, \mu+\tau \}} + \max\left\{20, \frac{9\tau^2}{2L^2}+2 \right\} \Vert y_k^* - y_{k-1}^*\Vert^2.
\end{align}
Now we want to bound $\|y_{k-1}^* - y_k^*\|$. By optimality condition, we have for $\forall y$,
\begin{equation}
    (y-y_k^*)^\top \nabla \Tilde{\Psi}_t(y_k^*) \leq 0, \quad (y-y_{k-1}^*)^\top \nabla \Tilde{\Psi}_{t-1}(y_{k-1}^*) \leq 0.
\end{equation}
Choose $y$ in the first inequality to be $y_{k-1}^*$, $y$ in the second inequality to be $y_k^*$, and sum them together, we have 
\begin{equation}
    (y_k^* - y_{k-1}^*)^\top (\nabla \Tilde{\Psi}_{k-1}(y_{k-1}^*) - \nabla \Tilde{\Psi}_k (y_k^*))\leq 0.
\end{equation}
Using $\nabla \Tilde{\Psi}_k(y) = \nabla_y \hat{f}(x^*(y), y) - \tau(y-z_k)$, we have
\begin{equation} \label{y optimality}
    (y_k^* - y_{k-1}^*)^\top (\nabla_y \hat{f}(x^*(y_{k-1}^*), y_{k-1}^*) - \tau(y_{k-1}^* - z_{k-1}) - \nabla_y \hat{f}(x^*(y_k^*), y_k^*) + \tau(y_k^* -z_k))\leq 0.
\end{equation}
By strong concavity of $\hat{\Psi}(y) = \max_{x}\hat{f}(x, y)$, we have 
\begin{equation}
    (y_k^* - y_{k-1}^*)^\top (\nabla \hat{\Psi}(y_k^*) - \nabla \hat{\Psi}(y_{k-1}^*)) \leq -\mu\|y_k^* - y_{k-1}^*\|^2.
\end{equation}
Adding to (\ref{y optimality}), we have
\begin{equation}
    (y_k^* - y_{k-1}^*)^\top [\tau(y_k^* - z_k) - \tau(y_{k-1}^* - z_{k-1})]\leq -\mu\|y_k^* - y_{k-1}^*\|^2
\end{equation}
Rearranging, 
\begin{equation}
    \frac{\tau}{\mu+\tau}(y_k^*-y_{k-1}^*)^\top(z_{k-1}-z_k) \geq \|y_k^* - y_{k-1}^*\|^2. 
\end{equation}
Further with $(y_k^*-y_{k-1}^*)^\top(z_{k-1}-z_k) \leq \|y_k^*-y_{k-1}^*\|\|z_{k-1}-z_k\| $, we have 
\begin{equation}  \label{optimal y dist bound by z}
    \|y_k^* - y_{k-1}^*\| \leq \frac{\tau}{\mu+\tau}\|z_{k-1}-z_k\|.
\end{equation}
From updates of $\{z_k\}_k$, we have for $t>2$
\begin{align} \nonumber 
    \|z_k - z_{k-1}\| = &\left\Vert y_{k-1} + \frac{\sqrt{q}-q}{\sqrt{q}+q}(y_{k-1}-y_{k-2}) - y_{k-2} - \frac{\sqrt{q}-q}{\sqrt{q}+q}(y_{k-2}-y_{k-3})\right\Vert \\ \nonumber 
    \leq & \left(1+\frac{\sqrt{q}-q}{\sqrt{q}+q}\right)\|y_{k-1} - y_{k-2}\| + \frac{\sqrt{q}-q}{\sqrt{q}+q}\|y_{k-2}-y_{k-3}\| \\ \nonumber 
    \leq & 2\|y_{k-1} - y_{k-2}\| + \|y_{k-2}-y_{k-3}\| \\  
    \leq & 6\max\{\|y_{k-1} - y^*\|, \|y_{k-2} - y^*\| , \|y_{k-3} - y^*\| \} 
\end{align}
Therefore,
\begin{align*}
    \|z_k - z_{k-1}\|^2 \leq & 36\max\{\|y_{k-1} - y^*\|^2, \|y_{k-2} - y^*\|^2 , \|y_{k-3} - y^*\|^2 \}  \\
    \leq & \frac{72}{\mu}\max\{\hat{\Psi}(y_{k-1})-\hat{\Psi}^*, \hat{\Psi}(y_{k-2})-\hat{\Psi}^*, \hat{\Psi}(y_{k-3})-\hat{\Psi}^*\} \\
    \leq & \frac{72}{\mu} \max\{S_{k-1}, S_{k-2},S_{k-3}\},
\end{align*}
where in the second inequality we use strongly concavity of $\hat{\Psi}$ and in the last we use $\hat{\Psi}(y_k) - \hat{\Psi}^*\leq S_k$. Combining with (\ref{optimal y dist bound by z}) and (\ref{initial dist bound}), we have 
\begin{equation}
     \|x_{k-1} - x_k^*\|^2+ \|y_{k-1} - y_k^*\|^2 \leq \frac{4\hat{\epsilon}_{k-1}}{\min\{L, \mu+\tau \}} + \frac{36\tau^2\max\{40L^2, 9\tau^2 + 4L^2 \}}{(\mu+\tau)^2L^2\mu} \max\{S_{k-1}, S_{k-2},S_{k-3}\}.
\end{equation}
Plugging in $S_{k} \leq \frac{8}{(\sqrt{q}-\rho)^2}(1-\rho)^{k+1} \gap_{\hat{f}}(x_0, y_0)$ from the proof of Theorem \ref{thm catalyst scsc} and from definition of $\epsilon_k$ and $\hat{\epsilon}_k$, we have 
\begin{equation}
     \|x_{k-1}- x_k^*\|^2+ \|y_{k-1} - y_k^*\|^2 \leq \bigg\{ \frac{2}{\mu\min\{L, \mu+\tau \}}\frac{1}{1-\rho} + \frac{288\sqrt{2}\tau^2\max\{40L^2, 9\tau^2 + 4L^2 \}}{(\mu+\tau)^2L^2\mu^2(\sqrt{q}-\rho)^2)} \frac{1}{(1-\rho)^2}  \bigg\} \epsilon_k.
\end{equation}
It is left to discuss the case $t=2$. Similarly, we have
\begin{equation*} 
    \|z_2 - z_{1}\| = \left\Vert y_{1} + \frac{\sqrt{q}-q}{\sqrt{q}+q}(y_{1}-y_{0}) - y_0\right\Vert   =  \left(1+\frac{\sqrt{q}-q}{\sqrt{q}+q}\right)\|y_{1} - y_{0}\| \leq  4\max\{\|y_{1} - y^*\|, \|y_{0} - y^*\| \} 
\end{equation*}
Then 
\begin{equation*}
    \begin{split}
        & \|z_2 - z_{1}\|^2 \leq  16\max\{\|y_{1} - y^*\|^2, \|y_{0} - y^*\|^2  \}\\  
        \leq\ & 
        \frac{32}{\mu}\max\{\hat{\Psi}(y_{1})-\hat{\Psi}^*, \hat{\Psi}(y_{0})-\hat{\Psi}^* \} 
        \leq  
        \frac{32}{\mu} \max\{S_1, \gap_{\hat{f}}(x_0, y_0)  \},
    \end{split}
\end{equation*}
Combining with (\ref{optimal y dist bound by z}) and (\ref{initial dist bound}), we have 
\begin{equation}
     \|x_1 - x_2^*\|^2+ \|y_1 - y_2^*\|^2 \leq \frac{4\hat{\epsilon}_{1}}{\min\{L, \mu+\tau \}} + \frac{16\tau^2\max\{40L^2, 9\tau^2 + 4L^2 \}}{(\mu+\tau)^2L^2\mu}  \max\{S_1, \gap_{\hat{f}}(x_0, y_0)\}.
\end{equation}
Plugging in $S_{1} \leq \frac{8}{(\sqrt{q}-\rho)^2}(1-\rho)^{2} \gap_{\hat{f}}(x_0, y_0)$ and definition of $\epsilon_2$ and $\hat{\epsilon}_1$, we have 
\begin{equation}
     \|x_1 - x_2^*\|^2+ \|y_1 - y_2^*\|^2 \leq \bigg\{ \frac{2}{\mu\min\{L, \mu+\tau \}}\frac{1}{1-\rho} + \frac{128\sqrt{2}\tau^2\max\{40L^2, 9\tau^2 + 4L^2 \}}{(\mu+\tau)^2L^2\mu^2(\sqrt{q}-\rho)^2)}  \bigg\} \epsilon_2.
\end{equation}

\end{proof}

\iffalse
\begin{lemma}[Inner-loop complexity for SC-SC objectives] \label{inner-loop complexity scsc}
Under the same assumptions as Theorem \ref{thm catalyst scsc}, we assume algorithm $\mathcal{M}$ can solve the subproblem (\ref{subprob}) when $f$ is $(\mu_x,\mu)$-SC-SC and $L$-Lipschitiz at a linear convergence rate depending on $\tau$: (here $(x_k, y_k)$ and $(x^*, y^*)$ denote iterate and optimal solution in solving the subproblem)
\begin{equation} \label{deter M rate}
    \Vert x_k - x^*\Vert^2+\Vert y_k-y^*\Vert^2 
    \leq \left(1-\frac{1}{\Lambda^{\mathcal{M}}_{\mu_x, \mu, L}(\tau)}\right)^k[\Vert x_0-x^*\Vert^2 + \Vert y_0-y^*\Vert^2]
\end{equation}
if $\mathcal{M}$ is deterministic, and 
\begin{equation} \label{stoc M rate}
    \mathbb{E}[\Vert x_k - x^*\Vert^2+\Vert y_k-y^*\Vert^2]\leq \left(1-\frac{1}{\Lambda^{\mathcal{M}}_{\mu_x, \mu, L}(\tau)}\right)^k[\Vert x_0-x^*\Vert^2 + \Vert y_0-y^*\Vert^2]
\end{equation}
if $\mathcal{M}$ is stochastic. Let $K_t(\epsilon)$ denote the number of iterations (expected number of iterations if $\mathcal{M}$ is stochastic) for $\mathcal{M}$ to find a point $(x,y)$ such that $\gap_{\Tilde{f}_t}(x, y) \leq \epsilon$ for subproblem (\ref{subprob}) at $t$-th iteration. Then
$K_t(\epsilon^{(t)})$ is $O \left(\Lambda^{\mathcal{M}}_{L, \mu, \Tilde{L}}(\tau) \log\left(\frac{\max\{1,\ell,\tau\}}{\min\{1,\mu_x,\mu \}}  \right)\right).$
\end{lemma}
\fi

{\noindent\textbf{Proof of Theorem \ref{thm catalyst inner}}}\\
\begin{proof}
We separate our arguments for the deterministic and stochastic settings. Inside this proof, $(x_{(i)}, y_{(i)})$ denotes the $i$-th iterate of $\mathcal{M}$ in solving the subproblem: $\min_x\max_y \Tilde{f}_k(x,y)$. We use $(x_k^*, y_k^*)$ to denote the optimal solution as before. We pick $(x_{(0)}, y_{(0)})$ to be $(x_{k-1}, y_{k-1})$.

 \paragraph{Deterministic setting.} 
 The subproblem is $(L+\max\{2L, \tau\})$-Lipschitz smooth and $(L,\mu+\tau)$-SC-SC. By Lemma \ref{criterion relation}, after $N$ iterations of algorithm $\mathcal{M}$, 
\begin{align*} 
    \|\nabla\Tilde{f}_k(x_{(N)}, y_{(N)})\|^2 &\leq 4(L+\max\{2L, \tau\})^2[\Vert x_{(N)}-x_k^*\Vert^2 + \Vert y_{(N)}-y_k^*\Vert^2] \\  \label{linear convergence after t}
    & \leq 4(L+\max\{2L, \tau\})^2 \left(1-\frac{1}{\Lambda^\mathcal{M}_{ \mu, L}(\tau)}  \right)^N[\Vert x_{k-1}-x_k^*\Vert^2 + \Vert y_{k-1}-y_k^*\Vert^2].
\end{align*}
 
Choosing 
\begin{align*}
N = &\Lambda^{\mathcal{M}}_{\mu, L}(\tau)\log\frac{4(L+\max\{2L, \tau\})^2 (\Vert x_{k-1}-x_k^*\Vert^2 +\Vert y_{k-1}-y_k^*\Vert^2)}{\epsilon_k} \\
\leq &  \Lambda^{\mathcal{M}}_{\mu, L}(\tau)\log\frac{4(L+\max\{2L, \tau\})^2 C_t\epsilon_k}{\epsilon_k} = \Lambda^{\mathcal{M}}_{\mu_x, \mu, L}(\tau)\log\left( 4(L+\max\{2L, \tau\})^2C_t \right) ,
\end{align*}
where $C_t$ is specified in Lemma \ref{warm start scsc},  we have $\|\nabla\Tilde{f}_k(x_{(N)}, y_{(N)})\|^2  \leq \epsilon_k$. 

\paragraph{Stochastic setting.}  With the same reasoning as in deterministic setting and applying Appendix B.4 of \citep{lin2018catalyst}, after
$$N = \Lambda^{\mathcal{M}}_{\mu, L}(\tau)\log\frac{4(L+\max\{2L, \tau\})^2 (\Vert x_{k-1}-x_k^*\Vert^2 +\Vert y_{k-1}-y_k^*\Vert^2)}{\epsilon_k}+1 $$
iterations of $\mathcal{M}$, we have $\|\nabla\Tilde{f}_k(x_{(N)}, y_{(N)})\|^2  \leq \epsilon_k$. 

\end{proof}

\subsection{Total complexity}

\textbf{Proof of Corollary \ref{THM CATALYST TOTAL}}

\begin{proof}
From Theorem \ref{THM CATALYST NCSC}, the number of outer-loop calls to find an $\epsilon$-stationary point of $\Phi$ is $T = O\left( L(\Delta+D_y^0)\epsilon^{-2} \right)$. From Theorem \ref{thm catalyst scsc}, by picking $\rho =0.9\sqrt{q}= 0.9\sqrt{\mu/(\mu +\tau)} $, we have
\begin{equation}
    \|\nabla \hat{f}_t(x^t_{k}, y^t_{k})\|^2  
    \leq \left[  \frac{5508L^2}{\mu^2(\sqrt{q}-\rho)^2} + \frac{18\sqrt{2}L^2}{\mu}\right](1-\rho)^{k}\|\nabla \hat{f}_t(x^t_{0}, y^t_{0})\|^2.
\end{equation}
Therefore, to achieve $ \|\nabla \hat{f}_t(x^{t}_{K}, y^{t}_{K})\|^2 \leq \alpha_t\|\nabla \hat{f}_t(x^t_{0}, y^t_{0})\|^2$, we need to solve (\ref{subprob})
$$K = 0.9\sqrt{(\tau+\mu)/\mu}\log\frac{\left[  \frac{5508L^2}{\mu^2(\sqrt{q}-\rho)^2} + \frac{18\sqrt{2}L^2}{\mu}\right]}{\alpha_t} = O\left(\sqrt{(\tau+\mu)/\mu} \log\left(\frac{\max\{1,L,\tau\}}{\min\{1,\mu \}}  \right)\right)$$
times, where $\alpha_t$ is defined as in Theorem \ref{THM CATALYST NCSC}. Finally, Theorem \ref{thm catalyst inner} implies that solving (\ref{subprob}) needs $N = O\left(\Lambda^{\mathcal{M}}_{\mu, L}(\tau) \log\left(\frac{\max\{1,L,\tau\}}{\min\{1,\mu \}}  \right)\right)$ gradient oracles. The total complexity is .
\begin{equation}
    T\cdot K\cdot N = O\left(\frac{\Lambda^{\mathcal{M}}_{\mu, L}(\tau)L(\Delta+D_y^0) }{\epsilon^2}\sqrt{\frac{\mu+\tau}{\mu}}  \log^2\left(\frac{\max\{1,L,\tau\}}{\min\{1,\mu \}}\right)\right).
\end{equation}

\end{proof}

\end{document}